\documentclass[12pt]{amsart}

\usepackage{amsrefs}
\usepackage{amsfonts}
\usepackage{amsmath}
\usepackage{amssymb}
\usepackage{mathrsfs}
\usepackage{graphicx}
\usepackage[all]{xy}
\usepackage{amsthm}
\usepackage{fancyhdr}
\usepackage{fullpage}
\usepackage{hyperref}
\usepackage{colortbl}
\usepackage{color}
\usepackage{array}

\newcommand{\sumsub}[1]{\sum_{\substack{#1}}} 

\newcommand{\onto}{\twoheadrightarrow} 

\newcommand{\newword}[1]{\emph{#1}}
\newcommand{\tleq}{\trianglelefteq\,}

\newcommand{\cleq}{\preccurlyeq}

\newcommand{\ZZ}{\mathbb{Z}}
\newcommand{\NN}{\mathbb{N}}
\newcommand{\XX}{\mathbb{X}}
\newcommand{\YY}{\mathbb{Y}}
\newcommand{\NNinf}{\NN \cup \{\infty\}}
\newcommand{\NNl}{\NN^\lev}
\newcommand{\NNinfl}{(\NNinf)^\lev}
\newcommand{\Q}{\mathbb{Q}}
\newcommand{\PP}{\mathbb{P}} 

\newcommand{\perm}{\mathfrak{S}} 

\newcommand{\vect}[1]{\mathbf{#1}} 
\newcommand{\vectinf}{\boldsymbol{\infty}}
\newcommand{\svect}[1]{\boldsymbol{#1}} 

\DeclareMathOperator{\clr}{\gamma}

\newcommand{\lev}{\mathit{l}} 
\newcommand{\colvec}[1]{\left( \begin{smallmatrix} #1 \end{smallmatrix} \right)}

\newcommand{\aaa}{\mathbf{a}}
\newcommand{\bb}{\mathbf{b}}
\newcommand{\cc}{\mathbf{c}}

\newcommand{\ee}{\mathbf{e}}

\newcommand{\etabasis}{\eta}

\newcommand{\emptycomp}{(\,)}

\DeclareMathOperator{\thetamap}{\Theta}

\DeclareMathOperator{\thetamapk}{\thetamap^{\vect{k}}}

\newcommand{\zetaQ}{\zeta_{\calQ}}  
\newcommand{\zetabarQ}{\bar{\zeta}_{\calQ}}
\newcommand{\zetaQk}{\zeta_{\calQ}^{\vect{k}}} 
\newcommand{\nuQ}{\nu_{\calQ}}
\newcommand{\nubarQ}{\bar{\nu}_{\calQ}}
\newcommand{\nuQk}{\nu_{\calQ}^{\vect{k}}}

\newcommand{\mon}[2]{(#1, #2)}

\newcommand{\ecmon}{\mon{\ee}{\cc}}

\newcommand{\calQ}{\mathcal{Q}}

\DeclareMathOperator{\Comp}{Comp}
\DeclareMathOperator{\OComp}{OComp}   

\DeclareMathOperator{\Des}{Des}   
\DeclareMathOperator{\Peak}{Peak}
\DeclareMathOperator{\dofI}{\mathbf{d}}
\DeclareMathOperator{\cofI}{\mathbf{c}}
\DeclareMathOperator{\pofI}{\mathbf{p}}

\DeclareMathOperator{\wDes}{\mathbf{Des}} 

\DeclareMathOperator{\odd}{odd}

\newcommand{\comp}{\,\vDash\,}  
\newcommand{\rev}[1]{\overleftarrow{#1}}

\DeclareMathOperator{\lst}{last}
\DeclareMathOperator{\std}{\mathrm{st}}
\DeclareMathOperator{\len}{\ell}
\DeclareMathOperator{\mdeg}{\mathbf{deg}}

\DeclareMathOperator{\supp}{supp} 
\DeclareMathOperator{\spp}{\mathit{sp}} 

\DeclareMathOperator{\apode}{\mathsf{s}} 
\DeclareMathOperator{\apodeQ}{\apode_{\mathcal{Q}}}  

\DeclareMathOperator{\rk}{\mathbf{rk}} 

\newcommand{\QSym}{\mathit{QSym}}   
\newcommand{\QSyml}{\QSym^{(\lev)}}  
\newcommand{\QSymll}{\QSym^{[ \lev ] }}
\newcommand{\NSym}{\mathbf{Sym}} 
\newcommand{\NSyml}{\NSym^{(\lev)}} 
\newcommand{\FQSym}{\mathbf{FQSym}}
\newcommand{\FQSyml}{\FQSym^{(\lev)}}
\newcommand{\Sym}{\mathit{Sym}}      
\newcommand{\Syml}{\Sym^{(\lev)}}
\newcommand{\NPl}{\mathbf{P}^{(\lev)}}

\newcommand{\GPosl}{\mathcal{R}^{(\lev)}}
\newcommand{\GPoslk}{\mathcal{R}^{(\lev),\vect{k}}}

\newcommand{\ideal}{\mathcal{I}}
\newcommand{\Ssubalg}{\mathcal{S}}

\newcommand{\abring}{\field\langle \aaa, \bb \rangle}  
\newcommand{\kodd}{\mathcal{O}}    
\newcommand{\keven}{\mathcal{E}}
\newcommand{\hopf}{\mathcal{H}}

\def\field{\mathbb{Q}}

\newcommand{\lyndon}{\mathcal{L}}

\theoremstyle{plain}
\newtheorem{theorem}{Theorem}[section]
\newtheorem{prop}[theorem]{Proposition}
\newtheorem{lemma}[theorem]{Lemma}
\newtheorem{corollary}[theorem]{Corollary}

\theoremstyle{definition}
\newtheorem{definition}[theorem]{Definition}
\newtheorem{example}[theorem]{Example}

\theoremstyle{definition}
\newtheorem{remark}[theorem]{Remark}

\theoremstyle{remark}

\numberwithin{equation}{section}

\renewcommand{\phi}{\varphi}

\title[Multigraded Hopf algebras and odd and even subalgebras]{Multigraded combinatorial Hopf algebras and refinements of odd and even subalgebras}
\author{Samuel K. Hsiao}
\address{Mathematics Program, Bard College, Annandale-on-Hudson, NY 12504}
\email{hsiao@bard.edu}

\author{Gizem Karaali}
\address{Department of Mathematics, Pomona College, Claremont, CA 91711}
\email{gizem.karaali@pomona.edu}

\thanks{This work began while S.K. Hsiao was supported by an NSF Postdoctoral Research Fellowship}

\begin{document}

\maketitle

\begin{abstract}
We develop a theory of multigraded (i.e., $\NNl$-graded) combinatorial Hopf algebras modeled on the theory of graded combinatorial Hopf algebras developed by Aguiar, Bergeron, and Sottile [Compos.\ Math.\ 142 (2006), 1--30]. In particular we introduce the notion of canonical $\vect{k}$-odd and $\vect{k}$-even subalgebras associated with any multigraded combinatorial Hopf algebra, extending simultaneously the work of Aguiar et al.\ and Ehrenborg. Among our results are specific categorical results for higher level quasisymmetric functions, several basis change formulas, and a generalization of the descents-to-peaks map. 
\end{abstract}

{\fontsize{12pt}{12pt}\selectfont

\parskip=1pt

\setcounter{tocdepth}{2}
\tableofcontents

}

\section{Introduction}\label{S:intro}

Quasisymmetric functions appear throughout algebraic combinatorics, in contexts that often seem unrelated. Within the framework of combinatorial Hopf algebras developed in \cite{ABS06}, this can be explained in category-theoretic terms: The combinatorial Hopf algebra $\QSym$ of quasisymmetric functions is a terminal object of the category of combinatorial Hopf algebras. In this framework, a special role is claimed by the odd subalgebra of  $\QSym$, the algebra of peak functions, as the terminal object of the category of odd combinatorial Hopf algebras. 

In this paper, we work with a natural generalization of $\QSym$. Specifically we focus on $\QSyml$, the multigraded (i.e., $\NNl$-graded) algebra of quasisymmetric functions of level $\lev$, and show that in the corresponding category of multigraded combinatorial Hopf algebras, an analogous universal property still holds (Theorem \ref{T:QSymlUniversal}). Along the way we obtain several more general results about the objects of this category, which are natural multigraded analogues of results in \cite{ABS06}. 

Another goal of ours is to obtain a refinement of the notions of odd and even subalgebras of \cite{ABS06} by developing ``$\vect{k}$-analogues" of the relevant constructions. We attain this goal in Sections \ref{S:defkoddkeven} and \ref{S:koddkevenQSym}. There we show that given a multigraded Hopf algebra $\hopf = \bigoplus_{\vect{n} \in \NNl} \hopf_{\vect{n}}$ and a character on $\hopf$, for every $\vect{k} \in \NNinfl$ we have two canonically defined Hopf subalgebras $\kodd^{\vect{k}}(\hopf)$ and $\keven^{\vect{k}}(\hopf)$. As expected our $\vect{k}$-analogues generalize the original notions: When $\lev = 1$ and $\vect{k} = \infty$, $\kodd^{\vect{k}}(\hopf)$ and $\keven^{\vect{k}}(\hopf)$ are precisely the odd and even subalgebras of \cite{ABS06}. Moreover in the $\lev = 1$ case our construction refines what is known about $\QSym$ by providing a sequence of Hopf subalgebras 
\[ 
\QSym = \kodd^0(\QSym) \supsetneq \kodd^2(\QSym) \supsetneq \kodd^4(\QSym) \supsetneq \cdots \supsetneq \kodd^\infty(\QSym) 
\]
with certain universal properties. This sequence includes the algebra of Billey-Haiman shifted quasisymmetric functions (our $\kodd^2(\QSym)$) and Stembridge's peak algebra (our $\kodd^\infty(\QSym)$). When $\lev > 1$ and $\vect{k} = (\infty, \ldots, \infty)$, our $\kodd^{\vect{k}}(\QSyml)$ may be regarded as a higher level analogue of the peak algebra.

This work stems from several related but distinct threads of earlier research. In the following we summarize a few key points related to quasisymmetric functions, Eulerian posets, and colored and multigraded Hopf algebras. For more details we refer the reader to our bibliography. 

Quasisymmetric functions were first defined explicitly by Gessel \cite{Gessel83}, who introduced them as generating functions for weights of $P$-partitions and gave applications to permutation enumeration. The Hopf algebra structure on $\QSym$ was studied in detail by Malvenuto and Reutenauer \cite{MR95}. The Hopf algebraic approach allows us to interpret a multitude of constructions related to quasisymmetric functions in one uniform manner. A particularly relevant construction in this flavor is Ehrenborg's $F$-homomorphism, which associates each graded poset $P$ with a quasisymmetric function $F(P)$ that encodes the flag $f$-vector of $P$. When $P$ is restricted to the class of Eulerian posets, the $F(P)$ span a Hopf subalgebra of $\QSym$. Viewed within the combinatorial Hopf algebra framework of \cite{ABS06}, this Hopf subalgebra is precisely the odd subalgebra of $\QSym$, which also happens to be Stembridge's peak algebra \cite{Stembridge97}. We study multigraded analogues of these ideas in this paper. For example we define a Hopf algebra of multigraded posets, and then consider a generalization of Ehrenborg's $F$-homomorphism from this Hopf algebra to $\QSyml$ (see Example~\ref{Ex:GradedPosets}). As one would expect, the image of an Eulerian multigraded poset under this homomorphism always lies in the odd subalgebra $\kodd^{(\infty, \ldots, \infty)}(\QSyml)$ (cf.\ Example~\ref{Ex:EulerianPoset}). This in turn implies that the natural level $\lev$ versions of the generalized Dehn-Sommerville relations hold for multigraded Eulerian posets (see Example~\ref{Ex:GeneralizedDH} and Remark~\ref{Re:GeneralizedDH}).

We actually develop the multigraded version of the story relating graded posets and Hopf algebras even further. In a somewhat different context,  Ehrenborg \cite{Ehrenborg01} introduced a refinement of the notion of an Eulerian poset, called a $k$-Eulerian poset ($k\in \NN \cup \{ \infty \}$), and proposed that one could define canonical algebras corresponding to $k$-Eulerian posets in a way that would generalize the notion of an odd subalgebra\footnote{Actually Ehrenborg's work takes place in the setting of Newtonian coalgebras, or infinitesimal Hopf algebras \cite{Aguiar02}, where instead of ``odd subalgebra" one has the analogous notion of ``Eulerian subalgebra," but it is easy to translate his idea into the language of combinatorial Hopf algebras.}. We take this idea one step further by providing multigraded versions of the constructions suggested by Ehrenborg. In particular our algebras $\kodd^{\vect{k}}(\hopf)$ are in some sense multigraded generalizations of the $\mathcal{E}_k(A)$ mentioned in \cite[\S5]{Ehrenborg01}. 

The level $\lev$ quasisymmetric functions $\QSyml$ considered in this paper were introduced by Poirier \cite{Poirier98} to handle enumeration problems involving colored permutations (i.e., elements of wreath products $\ZZ_\lev\wr \perm_n$). The term ``level $\lev$" was coined by Novelli and Thibon \cite{NT04} and refers to the larger of the two algebras appearing in Poirier's work\footnote{The reader may like to refer to \cite[\S6.2]{BauHoh08} for an analysis of the two different colored generalizations of quasisymmetric functions in \cite{Poirier98}}. The terminology and basic Hopf algebraic properties of $\QSyml$ that we build on, as well as two other examples of multigraded Hopf algebras that we discuss--the higher level noncommutative symmetric functions and free quasisymmetric functions--are due to Novelli and Thibon \cite{NT04, NT08}. Our paper adds the category-theoretic perspective and gives concrete results about new combinatorially interesting bases, subalgebras, and maps within $\QSyml$, as well as maps from other Hopf algebras to $\QSyml$. 

We should mention that the smaller of Poirier's two algebras, usually referred to as the algebra of {\it colored quasisymmetric functions}, denoted here by $\QSymll$, has been studied by many authors. Baumann and Hohlweg \cite{BauHoh08} proved that $\QSymll$ is a Hopf algebra and related it to the larger Hopf algebra $\QSyml$. Their work places $\QSymll$ within a general descent theory for wreath products and reveals the functorial nature of many related colored constructions, notably the colored descent algebras of Mantaci and Reutenauer \cite{ManReu95}. In subsequent work, Bergeron and Hohlweg \cite{BerHoh06} continued to explain and unify various colored constructions, and they proposed a theory of colored combinatorial Hopf algebras analogous to the theory developed in \cite{ABS06}; this theory is developed further in \cite{HP09}. Some connections between our work and the Bergeron-Hohlweg theory are discussed briefly in \S\S\ref{SS:EarlyColoredHAs}. One key point that distinguishes our viewpoint is our emphasis on the $\NNl$-graded structure. It seems, however, that many of the $\vect{k}$-refinements that we consider here should translate into analogous constructions within the framework suggested by Bergeron and Hohlweg.

One final topic that we consider in this paper is the notion of a $\vect{k}$-analogue of the classic descents-to-peaks map on quasisymmetric functions, which shows up naturally in several different settings. When viewing $\QSym$ and the peak algebra as arising from ordinary and enriched $P$-partitions, the descents-to-peaks map is Stembridge's $\theta$-map \cite{Stembridge97}, which sends the $P$-partition weight enumerator of a labeled poset to the enriched $P$-partition weight enumerator of the same poset. In the setting of noncommutative symmetric functions, the dual of the descents-to-peaks map appears in the work of Krob, Leclerc and Thibon \cite{KLT97} as the specialization at $q=-1$ of the $A\to (1-q)A$ transform. Yet another interpretation of this map involving flag enumeration in oriented matroids was given in \cite{BHW03}. We do not attempt to develop level $\lev$ versions or $\vect{k}$-analogues of these general frameworks which give rise to the descents-to-peaks map (although see \cite{HP09} for a related story on colored $P$-partitions). Instead we employ the character-theoretic approach of \cite{ABS06} to define the appropriate $\vect{k}$-analogues $\thetamap^{\vect{k}}:\QSyml \to \QSyml$ and give explicit formulas to help compute these maps.

The rest of this paper is organized as follows. In Section \ref{S:background} we provide the necessary background on $\lev$-partite numbers and vector compositions, and introduce several examples of multigraded Hopf algebras. In Section \ref{S:TheoryMCHA} we study the category of multigraded combinatorial Hopf algebras and formulate the universal property of $\QSyml$. We also describe how to relate our constructions to earlier work. Section \ref{S:defkoddkeven} contains our basic results on $\vect{k}$-odd and $\vect{k}$-even Hopf subalgebras. In particular we show that a substantial part of the standard theory of combinatorial Hopf algebras developed in \cite{ABS06} goes through for these $\vect{k}$-analogues. Section~\ref{S:koddkevenQSym} contains more explicit descriptions of the $\vect{k}$-odd and $\vect{k}$-even Hopf subalgebras of $\QSyml$. We focus mostly on the $\vect{k}$-odd algebras, for which we describe various bases and compute Hilbert series. In Section \ref{S:ktheta} we define the $\vect{k}$-analogue of the descents-to-peaks map and introduce $\vect{k}$-analogues of the basis of peak functions.

We thank Marcelo Aguiar, Nantel Bergeron, Christophe Hohlweg, Jean-Yves Thibon and the two anonymous referees for helpful comments.


\section{Background and examples of multigraded Hopf algebras}
\label{S:background}


Throughout this paper, let $\NN=\{\, 0,1,2,\dots \, \}$ denote the set of nonnegative integers and $\PP=\{1,2,\ldots\}$ denote the set of positive integers. If $m, n\in \NN$ then  $[n]=\{1,2,\ldots,n\},$ and $[n,m]=\{n,n+1,n+2,\ldots,m\}$. In particular, $[0]=\emptyset$ and $[n,m]=\emptyset$ if $n>m.$
In this paper, we use the term \newword{multigraded} to mean \newword{$\NNl$-graded}, where $\lev$ is a fixed positive integer. We call the elements of $\NNl$ \newword{$\lev$-partite number}s and treat them as column vectors. Thus $\vect{0}\in \NNl$ denotes the zero vector, and for each $i\in [0,\lev-1]$, $\vect{e}_i \in \NNl$ denotes the coordinate vector with a $1$ in position $i$ and $0$ everywhere else. When needed, we add vectors componentwise, and we sometimes use a partial order on $\NNl$ given by $\vect{i} \le \vect{j}$ if each coordinate of $\vect{i}$ is less than or equal to the corresponding coordinate of $\vect{j}$. The $\lev$-partite number
\[
\vect{n} = \begin{pmatrix} n_0 \\ \vdots \\ n_{l-1}\end{pmatrix} \in \NN^{\lev}
\] 
is said to have \newword{weight} $|\vect{n}| = n_1 + \cdots + n_{\lev}$ and  \newword{support} $\supp(\vect{n}) = \{ i \in [0, \lev-1] \mid n_i \ne 0 \}$. 


This section provides the necessary background on $\lev$-partite numbers and vector compositions, and introduces several examples of multigraded Hopf algebras that appear throughout the paper. These algebras are
\begin{center}
\begin{tabular}{cl}
$\GPosl$ & multigraded posets \\
$\GPoslk$ & multigraded $\vect{k}$-Eulerian posets \\
$\NPl$ &  colored posets \\
$\FQSyml$ & free quasisymmetric functions of level $\lev$ \\
$\NSyml$ &  noncommutative symmetric functions of level $\lev$ \\
$\QSyml$ & quasisymmetric functions of level $\lev$ \\
$\Syml$ & MacMahon's multi-symmetric functions
\end{tabular}
\end{center}
Generally speaking, the Hopf algebras that we work with are $\NNl$-graded bialgebras $\hopf = \bigoplus_{\vect{n} \in \NNl} \hopf_{\vect{n}}$ that are connected, meaning $\hopf_{\vect{0}}$ is the one-dimensional vector space spanned by the unit element; the counit is always defined to be the projection onto $\hopf_{\vect{0}}$. Such a bialgebra always has a recursively defined antipode (see, e.g., \cite[Lemma~2.1]{Ehrenborg96}), so it automatically becomes a Hopf algebra.

Our discussion of vector compositions and the Hopf algebras $\QSyml$, $\NSyml$, and $\FQSyml$ follows Novelli and Thibon \cite{NT08}, where most of the relevant definitions and results can be found. 

\subsection{The Hopf algebra $\GPosl$}
\label{SS:GPosl}

A finite poset $P$ is said to be \newword{graded} if it has a unique minimum element  $\hat{0}$ and a unique maximum element $\hat{1}$, and it is equipped with a rank function $\mathrm{rk}:P \to \NN$ with the property that if $y$ covers $x$ in $P$ then $\mathrm{rk}(y) = \mathrm{rk}(x) + 1$.

\begin{definition}
A \newword{multigraded poset} is a (finite) graded poset $P$ together with a function $\rk_P:P \to \NNl$ such that $\rk_P(\hat{0}) = \vect{0}$ and if $y$ covers $x$ in $P$ then $\rk_P(y) = \rk_P(x) + \vect{e}_i$ for some coordinate vector $\vect{e}_i \in \NNl$. We call $\rk_P$ the \newword{multirank function} of $P$ and say that the \newword{multirank of $P$} is $\rk_P(\hat{1})$. 
\end{definition}

Each interval $[x,y] = \{ z \in P \mid x \le z \le y\}$ in a multigraded poset $P$ is again a multigraded poset with multirank function $\rk_{[x,y]}(z) = \rk_P(z) - \rk_P(x)$. Two multigraded posets $P$ and $Q$ are said to be isomorphic if there is an isomorphism of posets $\phi:P\to Q$ such that $\rk_Q \circ \phi =  \rk_P$. 

Let $\GPosl$ be the vector space with basis the set of isomorphism classes of multigraded posets. We get a multigrading $\GPosl = \bigoplus_{\vect{n} \in \NNl} \GPosl_{\vect{n}}$ by taking $\GPosl_{\vect{n}}$ to be the linear span of all (isomorphism classes of) posets with multirank $\vect{n}$. 

Multiplication in $\GPosl$ is defined to be the usual Cartesian product $P \times Q$ with $\rk_{P\times Q}(x,y) = \rk_P(x) + \rk_Q(y)$ for $(x,y) \in P\times Q$. The coproduct is defined by 
\[
\Delta(P) = \sum_{\hat{0} \le x \le \hat{1}} [\hat{0}, x] \otimes [x, \hat{1}].
\]
The unit is the one-element poset. When $\lev = 1$, the Hopf algebra $\GPosl$ specializes to Rota's Hopf algebra (see, e.g., \cite[Example~2.2]{ABS06}).

\subsection{The Hopf subalgebra $\GPoslk \subseteq \GPosl$}
\label{SS:GPoslk}

Recall that a graded poset $P$ is said to be \newword{Eulerian} if whenever $x\le_P y$, the M\"obius function satisfies $\mu([x,y]) = (-1)^{\mathrm{rk}(y) - \mathrm{rk}(x)}$.  Generalizing \cite[Definition~4.1]{Ehrenborg01}, we make the following definition.

\begin{definition}
Let $P$ be a multigraded poset and let $\vect{k} \in \NNinfl$. We say that $P$ is \newword{$\vect{k}$-Eulerian} if every interval of rank $\vect{n} \le \vect{k}$ is \newword{Eulerian}. 
\end{definition}

For example, the multigraded poset
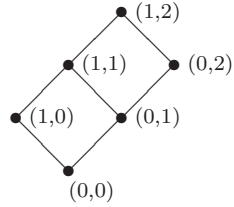
\begin{center}
\setlength{\unitlength}{1pt}
\begin{picture}(60,60)
\put(20,0){\circle*{4}} \put(20,0){\line(-1,1){20}} \put(20,0){\line(1,1){20}} 
\put(0,20){\circle*{4}} \put(0,20){\line(1,1){20}}
\put(40,20){\circle*{4}} \put(40,20){\line(-1,1){20}} \put(40,20){\line(1,1){20}}
\put(20,40){\circle*{4}} \put(20,40){\line(1,1){20}}
\put(60,40){\circle*{4}} \put(60,40){\line(-1,1){20}}
\put(40,60){\circle*{4}} 
\put(20,-10){$\scriptstyle{(0,0)}$} \put(5,18){$\scriptstyle{(1,0)}$}  
\put(45,18){$\scriptstyle{(0,1)}$} \put(25,38){$\scriptstyle{(1,1)}$}
\put(65,38){$\scriptstyle{(0,2)}$}
\put(45,58){$\scriptstyle{(1,2)}$}
\end{picture}
\end{center}
is $(1,1)$-Eulerian but not $(0,2)$-Eulerian. 

Let $\GPoslk$ be the subspace of $\GPosl$ spanned by all $\vect{k}$-Eulerian posets. Using the fact that the M\"obius function is multiplicative, i.e., $\mu(P\times Q) = \mu(P) \cdot \mu(Q)$, it is easy to show that the Cartesian product of two $\vect{k}$-Eulerian posets is a $\vect{k}$-Eulerian poset. It is also clear that every interval in a $\vect{k}$-Eulerian poset is again $\vect{k}$-Eulerian. Therefore $\GPoslk$ is a Hopf subalgebra of $\GPosl$.

\subsection{The Hopf algebra $\NPl$}
\label{SS:ColoredPosets}

Let $\vect{A} =  \PP \times [0, \lev-1]$ be the set of $\lev$-colored positive integers. If $x = (i, j)\in \vect{A}$ then we call $|x|= i$ the \newword{absolute value} of $x$ and $\clr(x)=j$  the \newword{color} of $x$. 

\begin{definition}
A finite subset $P\subseteq \vect{A}$ together with a partial order $<_P$ will be called a \newword{colored poset} if $\{ |x| \mid x\in P\} = \{1, 2, \ldots, |P|\}$.
\end{definition}

Let $\NPl$ denote the vector space of formal $\field$-linear combinations of colored posets. If we define the multidegree of a colored poset $P$ to be $(n_0, \ldots, n_{\lev-1})^T\in \NNl$, where $n_i$ is the number of occurrences of color $i$ in the multiset $\{ \clr(x) \mid x\in P\}$, and denote by $\NPl_{\vect{n}}$ the subspace of colored posets of multidegree $\vect{n}$, then $\NPl$ becomes a multigraded vector space $\NPl = \bigoplus_{\vect{n} \in \NNl} \NPl_{\vect{n}}$. 
 
Given two colored posets $P$ and $Q$, define a new colored poset $P \sqcup Q$ to be the disjoint union of $P$ and $Q$ with the absolute values of $Q$ shifted up by $n = |P|$. Thus as a set $P \, \sqcup \, Q = \{ (i, j) \in \vect{A} \mid (i,j) \in P \text{ or } (i-n, j) \in Q\}$, and the new partial order satisfies $(i, j) <_{P\sqcup Q} (i', j')$ if and only if $(i, j) <_P (i', j')$, or $(i-n, j) <_Q (i'-n, j')$. For example,
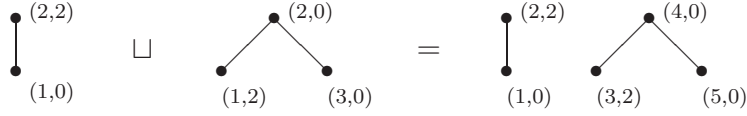
\begin{center}
\setlength{\unitlength}{1pt}
\begin{picture}(30,40)(0,-10)
 \put(0,0){\line(0,1){20}} \put(0,20){\circle*{4}} \put(0,0){\circle*{4}}
 \put(5,-10){$\scriptstyle (1,0)$} \put(5,20){$\scriptstyle (2,2)$} 
\end{picture}
\begin{picture}(40,40)
\put(10,15){$\sqcup$}
\end{picture}
\begin{picture}(50,40)(0,-10)
\put(0,0){\line(1,1){20}} \put(20,20){\line(1,-1){20}}
\put(0,0){\circle*{4}} \put(20,20){\circle*{4}} \put(40,0){\circle*{4}}
\put(0,-12){$\scriptstyle (1,2)$} \put(25,20){$\scriptstyle (2,0)$}  \put(40,-12){$\scriptstyle (3,0)$}
\end{picture}
\begin{picture}(40,40)
\put(20,15){$=$}
\end{picture}
\begin{picture}(40,40)(-10,-10)
 \put(0,0){\line(0,1){20}} \put(0,20){\circle*{4}} \put(0,0){\circle*{4}}
 \put(0,-12){$\scriptstyle (1,0)$} \put(5,20){$\scriptstyle (2,2)$} 
\end{picture}
\begin{picture}(50,40)(0,-10)
\put(0,0){\line(1,1){20}}  \put(20,20){\line(1,-1){20}}
\put(0,0){\circle*{4}} \put(20,20){\circle*{4}}  \put(40,0){\circle*{4}}
\put(0,-12){$\scriptstyle (3,2)$} \put(25,20){$\scriptstyle (4,0)$}  \put(40,-12){$\scriptstyle (5,0)$}
\end{picture}
\end{center}
With this product $\NPl$ becomes a noncommutative multigraded algebra whose unit element is the empty colored poset. When $\lev = 1$ we recover the construction discussed in \cite[\S3.8]{DHT02}.

To define the coproduct, first recall that a \newword{(lower) order ideal} of a poset $P$ is a subset $I\subseteq P$ such that if $x\in I$ and $y <_P x$ then  $y\in I$. Every order ideal $I$ and its complement $P-I$ can be thought of as a subposet of $P$. The set of order ideals of $P$ is denoted by $\mathcal{I}(P)$. Now suppose that $P$ is a colored poset and $Q = \{ (i_1, j_1) \ldots (i_m, j_m)\}$ is any subposet (not necessarily a colored poset). The \newword{standardization} of $Q$ is the colored poset $\std(Q)$ obtained by relabeling $(i_r, j_r)$ by $(i'_r, j_r)$ for each $(i_r, j_r) \in Q$, where $i'_1 i'_2 \ldots i'_m$ is the standardization of $i_1 i_2 \ldots i_m$. In other words $i'_1 i'_2 \ldots i'_m$ is the unique permutation of $[m]$ such that $i'_r < i'_s$ if and only if $i_r < i_s$ for every $r, s \in [m]$. Now we can define the coproduct of a colored poset $P$ by
\[ 
\Delta(P) = \sum_{I \in \mathcal{I}(P)} \std(I) \otimes \std (P-I).
\]
For example, 
\[
\begin{aligned}
\Delta\left(
\begin{picture}(45,20)(0,0)
\put(0,0){\line(1,1){15}} \put(15,15){\line(1,-1){15}}
\put(0,0){\circle*{4}} \put(15,15){\circle*{4}} \put(30,0){\circle*{4}}
\put(0,-10){$\scriptstyle{(1,2)}$} \put(18,15){$\scriptstyle{(2,0)}$}  \put(30,-10){$\scriptstyle{(3,0)}$}
\end{picture}
\right)
=
\emptyset \otimes 
\begin{picture}(50,20)(0,0)
\put(0,0){\line(1,1){15}} \put(15,15){\line(1,-1){15}}
\put(0,0){\circle*{4}} \put(15,15){\circle*{4}} \put(30,0){\circle*{4}}
\put(0,-10){$\scriptstyle{(1,2)}$} \put(18,15){$\scriptstyle{(2,0)}$}  \put(30,-10){$\scriptstyle{(3,0)}$}
\end{picture}
+ &
\begin{picture}(20,20)(0,0) 
\put(10,0){\circle*{4}} \put(5,-10){$\scriptstyle{(1,2)}$}
\end{picture}
\otimes
\begin{picture}(30,20)(0,0) 
\put(10,0){\circle*{4}} \put(10,15){\circle*{4}} \put(10,0){\line(0,1){15}}
\put(12,-10){$\scriptstyle{(2,0)}$} \put(12,15){$\scriptstyle{(1,0)}$} 
\end{picture}
+
\begin{picture}(20,20)(0,0) 
\put(10,0){\circle*{4}} \put(5,-10){$\scriptstyle{(1,0)}$}
\end{picture}
\otimes
\begin{picture}(30,20)(0,0) 
\put(10,0){\circle*{4}} \put(10,15){\circle*{4}} \put(10,0){\line(0,1){15}}
\put(12,-10){$\scriptstyle{(1,2)}$} \put(12,15){$\scriptstyle{(2,0)}$} 
\end{picture} \\
& +
\begin{picture}(40,20)(-10,0) 
\put(0,0){\circle*{4}} \put(20,0){\circle*{4}} 
\put(-5,-10){$\scriptstyle{(1,2)}$} \put(15,-10){$\scriptstyle{(2,0)}$}
\end{picture}
\otimes
\begin{picture}(20,20)(0,0) 
\put(10,0){\circle*{4}} \put(5,-10){$\scriptstyle{(1,0)}$}
\end{picture}
+
\begin{picture}(50,20)(-10,0)
\put(0,0){\line(1,1){15}} \put(15,15){\line(1,-1){15}}
\put(0,0){\circle*{4}} \put(15,15){\circle*{4}} \put(30,0){\circle*{4}}
\put(0,-10){$\scriptstyle{(1,2)}$} \put(18,15){$\scriptstyle{(2,0)}$}  \put(30,-10){$\scriptstyle{(3,0)}$}
\end{picture}
\otimes
\emptyset.
\end{aligned}
\]
With this coproduct $\NPl$ becomes a multigraded bialgebra.

Note that our definitions here differ slightly from the ones in \cite{HP09}. However it is still possible to see the connections. For instance the commutative Hopf algebra of colored posets $\mathcal{P}^{(\lev)}$ studied in \cite{HP09} can be thought of as $\NPl$ modulo isomorphism of colored posets.

\subsection{Vector compositions} 
\label{SS:lpartite}

We now give some definitions related to $\lev$-partite numbers and vector compositions which will be used in the remaining examples.

If $\vect{i}_1, \ldots, \vect{i}_m$ are $\lev$-partite numbers of nonzero weight, then the $\lev \times m$ matrix $\vect{I} = (\vect{i}_1, \ldots, \vect{i}_m)$ is called a \newword{vector composition of $\vect{i}_1 + \cdots + \vect{i}_m$} of \newword{length} $\len(\vect{I}) = m$ and \newword{weight} $|\vect{I}| = |\vect{i}_1| + \cdots + |\vect{i}_\lev|.$ The empty composition $\emptycomp$ is said to be the unique vector composition of $\vect{0}$. Let $\Comp(\vect{n})$ denote the set of vector compositions of $\vect{n}$. We often write $\vect{I} \comp \vect{n}$ to mean $\vect{I}\in \Comp(\vect{n})$. Conversely we will write $\Sigma\vect{I}$ to denote the $\lev$-partite number $\vect{n}$ obtained by summing the columns of $\vect{I}$. An example of a vector composition of a $4$-partite number is
\begin{equation} \label{E:lcomp-example}
\vect{I} = \left(\begin{array}{ccccc}1 & 2 & 0 & 0 & 0 \\0 & 0 & 0 & 1 & 1 \\ 0 & 0 & 0 & 0 & 0 \\1 & 0 & 3 & 0 & 4 \end{array}\right) \comp \left(\begin{array}{c}3 \\2 \\ 0 \\8\end{array}\right). 
\end{equation}
This vector composition has length $\len(\vect{I}) = 5$, weight $|\vect{I}| = 13$ and $\Sigma\vect{I}=(3,2,0, 8)^T$.

Next we consider three different partial orders on $\Comp(\vect{n})$. Let $\vect{I} = (\vect{i}_1, \ldots, \vect{i}_m) \comp \vect{n}$ and $\vect{J} \comp \vect{n}$. Write 
\[
\vect{I} \tleq \vect{J}
\] 
if $\vect{J}$ is a \newword{refinement} of $\vect{I}$, meaning there are vector compositions $\vect{J}_1,\ldots, \vect{J}_m$ whose concatenation is $\vect{J} = \vect{J}_1 \cdots \vect{J}_m$ and such that $\vect{J}_k \comp \vect{i}_k$ for each $k$. If additionally, every $\vect{J}_k$ has the property that every element in the support of a column is less than or equal to every element in the support of every column further to the right, then we write
\[
\vect{I} \cleq \vect{J}.
\] 
If instead we require elements in the support of each column to be {\em strictly less than} elements in the support of columns to the right, then we write
\[ 
\vect{I} \cleq_s \vect{J}.
\]
By construction, $\vect{I} \cleq_s \vect{J} \implies \vect{I}\cleq \vect{J} \implies \vect{I}\tleq \vect{J}$. For example, inside $\Comp( \left( \begin{smallmatrix}
3 \\ 5 \\ 7 
\end{smallmatrix} \right) )$ we have
\[
\begin{array}{cccc}
\vect{I} &  & \vect{J}_1 & \vect{J}_2  \\
\left(
\begin{array}{cc}
2 & 1 \\
0 & 5 \\
3 & 4
\end{array}
\right)
&
\begin{array}{c}
\tleq \\
\not\cleq
\end{array}
&
\left( 
 \begin{array}{cc}
 1 & 1 \\
 0 & 0 \\
 3 & 0 
 \end{array}
\right. 
&
\left.
\begin{array}{ccc}
1 & 0 & 0  \\
0 & 2 & 3  \\
0 & 0 & 4 
\end{array}
\right)
\end{array}
\quad
\quad
\begin{array}{cccccc}
\vect{I} &  & \vect{J}_1 & \vect{J}_2 & & \\
\left(
\begin{array}{cc}
2 & 1 \\
0 & 5 \\
3 & 4
\end{array}
\right)
&
\begin{array}{c}
\cleq \\
\not\cleq_s
\end{array}
&
\left( 
 \begin{array}{cc}
 1 & 1 \\
 0 & 0 \\
 0 & 3 
 \end{array}
\right. 
&
\left.
\begin{array}{ccc}
1 & 0 & 0 \\
0 & 2 & 3 \\
0 & 0 & 4
\end{array}
\right)
\end{array}
\]
\[
\begin{array}{cccccc}
\vect{I} &  & \vect{J}_1 & \vect{J}_2 & & \\
\left(
\begin{array}{cc}
2 & 1 \\
0 & 5 \\
3 & 4
\end{array}
\right)
&
\cleq_s
&
\left( 
 \begin{array}{cc}
 2 & 0 \\
 0 & 0 \\
 0 & 3 
 \end{array}
\right. 
&
\left.
\begin{array}{cc}
1 & 0  \\
0 & 5 \\
0 & 4 
\end{array}
\right)
\end{array}
\]
Notice that if $\vect{I} \tleq \vect{J}$ and if $\vect{K}$ is obtained by adding together any two adjacent columns of some $\vect{J}_i$, then clearly $\vect{I} \tleq \vect{K} \tleq \vect{J}$. The same is true of $\cleq$ and $\cleq_s$. Therefore, each interval of the poset $\Comp(\vect{n})$ relative to any of the partial orders $\tleq$, $\cleq$, and $\cleq_s$, is isomorphic to an interval in a poset of (ordinary) compositions ordered by refinement; the isomorphism takes a vector composition $(\vect{i}_1, \ldots, \vect{i}_m)$ to the sequence of weights $(|\vect{i}_1|,\ldots,|\vect{i}_m|)$.  It is well-known that the poset of compositions under refinement is a boolean lattice. Thus the M\"obius function of every interval $[\vect{I},\vect{J}]$ in $\Comp(\vect{n})$ with respect to any one of the three partial orders is given by $\mu([\vect{I}, \vect{J}]) = (-1)^{\len(\vect{I}) - \len(\vect{J})}$.

Novelli and Thibon \cite{NT04} introduced a way to encode a vector composition $\vect{I} = (\vect{i}_1, \ldots, \vect{i}_m)$ via two associated statistics, namely a set $\dofI(\vect{I})\subseteq [\, |\vect{I}| - 1\,]$ and a word $\cofI(\vect{I})$, where
\[
 \dofI(\vect{I}) = \{ |\vect{i}_1|, |\vect{i}_1| + |\vect{i}_2|, \ldots, |\vect{i}_1| + \cdots + |\vect{i}_{m-1}| \},
\]
and $\cofI(\vect{I})$ is formed by recording $\vect{I}_{i,j}$ copies of the color $i$ as one reads the entries of the matrix $\vect{I}$ sequentially down every column, starting from the left-most column. We will refer to $\cofI(\vect{I})$ as the \newword{coloring word} of $\vect{I}$. For example, with $\vect{I}$ as in \eqref{E:lcomp-example}, we have $\dofI(\vect{I}) = \{ 2, 4, 7, 8\}$ and $\cofI(\vect{I}) = 0300333113333$. 

If $w = w_1 \cdots w_n$ is a word in an ordered alphabet then the \newword{descent set of $w$} is
\[
\Des(w) = \{ i\in [n-1] \mid w_i > w_{i+1}\}.
\]
It is always true that $\Des(\cofI(\vect{I})) \subseteq \dofI(\vect{I})$. Conversely, given a positive integer $n$, a subset $S \subseteq [n-1]$, and a color word $w$ of length $n$ (i.e., a word $w = w_1 w_2 \cdots w_{n}$ in the alphabet $[0,\lev-1]$) such that $\Des(w) \subseteq S$, there is a unique vector composition $\vect{I}$ such that $\dofI(\vect{I}) = S$ and $\cofI(\vect{I}) = w$. Note that
\[
\vect{I} \cleq \vect{J} \quad \text{ if and only if } \quad \cofI(\vect{I}) = \cofI(\vect{J}) \text{ and } 
\dofI(\vect{I}) \subseteq \dofI(\vect{J}).
\]

In the following, most notably in \S\S\ref{SS:Syml}, besides the notion of vector compositions, we will need the analogous notion of a {\em vector partition}. For us a vector partition will be a multiset $\svect{\lambda}$ of $\lev$-partite numbers $\svect{\lambda} = \{\svect{\lambda}_1, \ldots, \svect{\lambda}_m\}$. Note for example that the multiset $\mathrm{cols}(\vect{I})$ of the columns of a given vector composition $\vect{I}$ is in fact a vector partition. 

We conclude this subsection with two more definitions. If $u=u_1 u_2 \ldots u_n \in [0, \lev-1]^n$ is a color word and $n_i$ is the number of occurrences of color $i$ for each $i \in [0,\lev-1]$, then we define the \newword{multidegree} of $u$ to be 
\begin{equation}\label{E:mdeg}
\mdeg(u) = \begin{pmatrix} n_0 \\ \vdots \\ n_{\lev-1} \end{pmatrix}.
\end{equation}
We also associate a vector composition $\vect{E}_u$ to each such color word $u$:
\[ \vect{E}_u = (\vect{e}_{u_1}, \cdots, \vect{e}_{u_n}).\]
A vector composition that is either empty or whose columns consist of coordinate vectors (i.e., standard basis vectors) will be called a \newword{coordinate vector composition}.

\subsection{The Hopf algebra $\FQSyml$} 
\label{SS:FQSyml}

This algebra is due to Novelli and Thibon \cite{NT04, NT08}. It can be thought of as a level $\lev$ generalization of the Malvenuto-Reutenauer Hopf algebra of permutations \cite{MR95}.
 
Let $\vect{A} = \PP \times [0, \lev-1]$ be the alphabet of colored positive integers as in \S\ref{SS:ColoredPosets}. A typical word in this alphabet will be written as $(J, u)$ where $J = j_1 \ldots j_n$ is a word of positive integers and $u = u_1 \ldots u_n$ is a word of colors. For each $n\ge 0$ we will call elements of the set $\perm_n \times [0, \lev-1]^n$ \newword{colored permutations of $n$}. The empty colored permutation of $0$ is $(\emptyset, \emptyset)$. 

If $(\sigma, u)$ is a colored permutation of $n$, then we define the noncommutative power series $\vect{F}_{\sigma, u} \in \field\langle \vect{A} \rangle$ by
\begin{equation} \label{E:FFdef}
\vect{F}_{\sigma, u} = 
\sumsub{j_1 \le j_2 \le \cdots \le j_n \\ r \in \Des(\sigma) \implies j_r < j_{r+1}}
(j_{\sigma^{-1}(1)} \, j_{\sigma^{-1}(2)} \ldots j_{\sigma^{-1}(n)}, \;
u_{\sigma^{-1}(1)}  \,  u_{\sigma^{-1}(2)} \ldots u_{\sigma^{-1}(n)}),
\end{equation}
with the convention $\vect{F}_{\emptyset, \emptyset} = 1$. For example, $\Des(251463) = \{2, 5\}$ and
\begin{equation*} 
\vect{F}_{251463, 120101} = \sum_{j_1 \le j_2 < j_3 \le j_4 \le j_5 < j_6} 
(j_3 \, j_1 \, j_6 \, j_4 \, j_2 \, j_5, 011120).
\end{equation*}
This definition is equivalent to the one given in \cite{NT08}, though here it is stated in a different way to emphasize the analogy with the upcoming defining equation~\eqref{E:Fdef} for the quasi-ribbon functions $F_{\vect{I}}$ (cf.\ \cite[Equation~(1.9)]{AS05}). 

Let $\FQSyml$ denote the vector space spanned by the $\vect{F}_{\sigma, u}$. We get a multigrading $\FQSyml = \bigoplus_{\vect{n} \in \NNl} \FQSyml_{\vect{n}}$ by defining $\FQSyml_{\vect{n}}$ to be the span of all $\vect{F}_{\sigma, u}$ with $\mdeg(u) = \vect{n}$ (in the sense of \eqref{E:mdeg}).  Novelli and Thibon proved that $\FQSyml$ is a subalgebra of $\field\langle \vect{A} \rangle$, and moreover that the product of two basis elements $\vect{F}_{\sigma, u} \, \vect{F}_{\tau, v}$ is given by \newword{shifted shuffling}; i.e., to get the result, we simply shift the letters of $\tau$ up by the length of $\sigma$, then shuffle $\sigma$ and $\tau$ together, all the while keeping track of their corresponding colors. For example,
\[
\vect{F}_{21, 02} \, \vect{F}_{12, 10} = \vect{F}_{2134, 0210} + \vect{F}_{2314,0120} + \vect{F}_{2341,0102}
+ \vect{F}_{3241, 1002} + \vect{F}_{3214, 1020} + \vect{F}_{3421, 1002}.
\]

The coproduct on $\FQSyml$ is defined by breaking up $\sigma$ into a concatenation of two (possibly empty) subwords, $\sigma = \tau \, \tau'$ and taking the standardization of each subword $\tau$ and $\tau'$. For example,
\[
\Delta(\vect{F}_{1423, 0021}) = 1 \otimes \vect{F}_{1423, 0021} + \vect{F}_{1, 0} \otimes \vect{F}_{312, 021} 
+ \vect{F}_{12, 00} \otimes \vect{F}_{12, 21} + \vect{F}_{132, 002} \otimes \vect{F}_{1, 1}
+ \vect{F}_{1423, 0021} \otimes 1.
\]
We refer the reader to \cite{NT08} for further details and precise definitions.

\subsection{The Hopf algebra $\NSyml$} 
\label{SS:NSyml}

This algebra is again due to Novelli and Thibon \cite{NT04, NT08}. It is a level $\lev$ version of the Hopf algebra of noncommutative symmetric functions, $\NSym$, introduced by Gelfand et al.\ \cite{GKLLRT95}.

For each $\vect{n} \in \NNl$, define the \newword{complete homogeneous noncommutative symmetric function} $S_{\vect{n}} \in \FQSyml$ by
\[
S_{\vect{n}} = \sum_{\mdeg(u) = \vect{n}} \vect{F}_{12\ldots n, u},
\]
where $\vect{F}_{12\ldots n, u}$ is defined as in Equation \eqref{E:FFdef}.
In particular $S_{\vect{0}} = 1$. For each vector composition $\vect{I} = \vect{i}_1 \ldots \vect{i}_m$, let 
\[
S^{\vect{I}} = S_{\vect{i}_1} \cdots S_{\vect{i}_m}.
\]
Note that $S^{\vect{I}}$ has multidegree $\vect{i}_1 + \cdots +\vect{i}_n$. 

Let $\NSyml$ be the subalgebra of $\FQSyml$ generated by the $S_{\vect{n}}$. Novelli and Thibon in \cite{NT08} proved the following:
\begin{prop}
The complete functions $S_{\vect{n}}$ are algebraically independent, so the $S^{\vect{I}}$ form a linear basis for $\NSyml$. Moreover $\NSyml$ is a multigraded Hopf subalgebra of $\FQSyml$ with coproduct satisfying
\begin{equation*}
 \Delta(S_{\vect{n}} ) = \sum_{\vect{n} = \vect{i} + \vect{j}} S_{\vect{i}} \otimes S_{\vect{j}} .
\end{equation*}
\end{prop}
We illustrate this with an example: 
\[
\Delta(S_{\colvec{2 \\ 1}}) = 1 \otimes S_{\colvec{2 \\ 1}} + S_{\colvec{1 \\ 0}} \otimes S_{\colvec{1 \\ 1}} + S_{\colvec{2 \\ 0}} \otimes S_{\colvec{0 \\ 1}} + S_{\colvec{0 \\ 1}} \otimes S_{\colvec{ 2 \\ 0}}+ S_{\colvec{1 \\ 1}} \otimes S_{\colvec{1 \\ 0}} + S_{\colvec{2 \\ 1}} \otimes 1.  
\]

The basis $S^{\vect{I}}$ is the level $\lev$ analog of the basis $S^I$ defined in \cite{GKLLRT95}. For us it will be useful to consider a level $\lev$ analog of a different basis from \cite{GKLLRT95}, the basis of noncommutative power sum symmetric functions of the second kind. We introduce these functions next. 

Let $t_0, \ldots, t_{\lev-1}$ be commutative variables. For each $\lev$-partite number $\vect{n} = (n_0,\ldots, n_{\lev-1})^T$, let $t^{\vect{n}} = t_0^{n_0} \cdots t_{\lev-1}^{n_{\lev-1}}$, and for each vector composition $\vect{I} = (\vect{i}_1, \dots, \vect{i}_m)$, let $t^{\vect{I}} = t^{\vect{i}_1} \cdots t^{\vect{i}_m}.$ In analogy with \cite[Equation~(26)]{GKLLRT95}, we define the \newword{level $\lev$ noncommutative power sum symmetric functions of the second kind} $\Phi_{\vect{n}}$ implicitly by 
\begin{equation}\label{E:Phidef}
\sum_{\vect{n} \in \NNl \setminus \{ \vect{0} \}} \frac{\Phi_{\vect{n}}}{|\vect{n}|} \; t^{\vect{n}} = \log \bigl(1 + \sum_{\vect{n} \in \NNl \setminus \{ \vect{0} \} } S_{\vect{n}} \; t^{\vect{n}} \bigr) 
\end{equation}
and $\Phi_{\vect{0}} = 1$. Here the logarithm is to be interpreted as a formal operation that satisfies $\log(1+T) = T - T^2/2 + T^3/3 \pm \cdots$. For each vector composition $\vect{I} = (\vect{i}_1, \ldots, \vect{i}_m)$, let
\begin{equation*}
\Phi^{\vect{I}} = \Phi_{\vect{i}_1} \cdots \Phi_{\vect{i}_m}.
\end{equation*}

The $\Phi^{\vect{I}}$ and $S^{\vect{I}}$ are related by triangular matrices. To describe this relationship, we adapt the notation from \cite[\S4.3]{GKLLRT95}.  For any vector composition $\vect{I} = (\vect{i}_1, \ldots, \vect{i}_m)$, let $\pi(\vect{I}) = \prod_{k=1}^m |\vect{i}_k|$ and $\spp(\vect{I}) = \len(\vect{I})! \; \pi(\vect{I})$. Now suppose that $\vect{I} \tleq \vect{J} = \vect{J}_1 \cdots \vect{J}_m$, where $\vect{J}_k \comp \vect{i}_k$ for each $k$. Then define
\[ 
\len(\vect{J}, \vect{I}) = \prod_{k=1}^m \len(\vect{J}_k) \quad \text{ and } \quad \spp(\vect{J}, \vect{I}) = \prod_{k = 1}^m \spp(\vect{J}_k).
\]
The bases $\Phi^{\vect{I}}$ and $S^{\vect{I}}$ are related as follows. 

\begin{prop}
For every vector composition $\vect{I}$, we have
   \begin{equation}\label{E:PhiS}
      \Phi^{\vect{I}} = \sum_{\vect{I}~\tleq~\vect{J}} (-1)^{\len(\vect{J}) - \len(\vect{I})} \; \frac{\pi(\vect{I})}{\len(\vect{J}, \vect{I})} \; S^{\vect{J}}
   \end{equation}
and
   \begin{equation}\label{E:SPhi}
     S^{\vect{I}} = \sum_{\vect{I}~\tleq~\vect{J}} \frac{1}{\spp(\vect{J},\vect{I})} \; \Phi^{\vect{J}}.
   \end{equation}
   Consequently, the $\Phi^{\vect{I}}$ form a basis of $\NSyml$. 
\end{prop}
\begin{proof}
From \eqref{E:Phidef} we compute
\begin{equation*} 
\Phi_{\vect{n}} = \sum_{\vect{I} \comp \vect{n}}  (-1)^{\len(\vect{I}) - 1} \frac{|\vect{n}|}{\len(\vect{I})} \; S^{\vect{I}},
\end{equation*}
from which \eqref{E:PhiS} follows by multiplicativity of the $\Phi^{\vect{I}}$. By exponentiating \eqref{E:Phidef}, we obtain
\begin{equation*}
S_{\vect{n}} = \sum_{\vect{I} \comp \vect{n}} \frac{1}{\spp(\vect{I})} \; \Phi^{\vect{I}},
\end{equation*}
and then \eqref{E:SPhi} follows by multiplicativity of the $S^{\vect{I}}$.  
\end{proof}

We now gather some additional facts about the $\Phi^{\vect{I}}$.

\begin{prop}\label{P:PhiPrimitive} 
   Every $\Phi_{\vect{n}}$ is primitive; that is,
   \begin{equation} \label{E:Phi-coprod}
      \Delta(\Phi_{\vect{n}}) = 1 \otimes \Phi_{\vect{n}} + \Phi_{\vect{n}} \otimes 1.
   \end{equation}
\end{prop}
\begin{proof}
The proof is a straightforward extension of the proof of Proposition~3.10 in \cite{GKLLRT95} and will be omitted. 
\end{proof}

\begin{prop}
   The antipode $\apode$ on $\NSyml$ satisfies
   \begin{equation} \label{E:Phi-antipode}
      \apode(\Phi^{\vect{I}}) = (-1)^{\len(\vect{I})} \, \Phi^{\rev{\vect{I}}}
   \end{equation}
and
   \begin{equation} \label{E:S-antipode}
      \apode(S^{\vect{I}}) = \sum_{\vect{I} \tleq \rev{\vect{J}} } (-1)^{\len(J)} \, S^{\vect{J}} ,
   \end{equation}
where $\overleftarrow{\vect{I}}$ is the vector composition obtained by reading the columns of $\vect{I}$ in reverse order.
\end{prop}
\begin{proof}
It follows from \eqref{E:Phi-coprod} that $\apode(\Phi_{\vect{n}}) = -\Phi_{\vect{n}}$. Since $\apode$ is an anti-homomorphism of algebras, \eqref{E:Phi-antipode} follows. Next, following the same argument leading up to Equation~(2.14) in \cite{MR95}, we have
\[ 
\begin{aligned}
\sum_{\vect{n} \ge \vect{0}} \apode(S_{\vect{n}}) \, t^{\vect{n}} = & 
\apode\left( \sum_{\vect{n} \ge \vect{0}} S_{\vect{n}} \, t^{\vect{n}}\right) = 
\apode\left(\exp\left(\sum_{\vect{n} \ne \vect{0}} \Phi_{\vect{n}} \, t^{\vect{n}} \right)\right) \\
& = \exp\left( \sum_{\vect{n} \ne \vect{0}} \apode(\Phi_{\vect{n}}) \, t^{\vect{n}} \right) 
 = \exp\left(\sum_{\vect{n} \ne \vect{0}} -\Phi_{\vect{n}} \, t^{\vect{n}}\right) = 
\left(\sum_{\vect{n} \ge \vect{0}} S_{\vect{n}} \, t^{\vect{n}} \right)^{-1}
\end{aligned}
\]
which implies
\[
\apode(S_{\vect{n}}) = \sum_{\vect{I} \comp \vect{n}} (-1)^{\len(\vect{I})} S^{\vect{I}}.
\]
Again since $\apode$ is an anti-homomorphism, \eqref{E:S-antipode} follows.
\end{proof}

\subsection{The Hopf algebra $\Syml$} 
\label{SS:Syml}

This algebra was introduced by MacMahon \cite{MacMahon}; our presentation here is adapted from Gessel \cite{Gessel87}. 

Let $\XX = X^0 \sqcup X^1 \sqcup \cdots \sqcup X^{\lev-1}$ be a set of independent colored commutative variables, with $X^i = \{ x_j^{(i)} \mid j \in \PP \}$ being the variables of \newword{color} $i$. For $\vect{n} = (n_0, \ldots, n_{\lev-1})^T \in \NNl$ define 
\[ 
\vect{x}_j^{\vect{n}} = (x_j^{(0)})^{n_0} \cdots (x_j^{(\lev-1)})^{n_{\lev-1}}. 
\]
A formal power series in $\XX$ of finite degree is called a \newword{multi-symmetric function} if the coefficients of the monomials $\vect{x}_{1}^{\boldsymbol{\lambda}_1} \vect{x}_{2}^{\boldsymbol{\lambda}_2} \cdots \vect{x}_{m}^{\boldsymbol{\lambda}_m}$ and $\vect{x}_{j_1}^{\boldsymbol{\lambda}_1} \vect{x}_{j_2}^{\boldsymbol{\lambda}_2} \cdots \vect{x}_{j_m}^{\boldsymbol{\lambda}_m}$ are equal whenever $\boldsymbol{\lambda}_1, \ldots, \boldsymbol{\lambda}_m$ are $\lev$-partite numbers and $j_1, \ldots, j_m$ are  distinct. The set of multi-symmetric functions forms a multigraded algebra denoted by $\Syml = \bigoplus_{\vect{n} \in \NNl} \Syml_{\vect{n}}$, where the \newword{multidegree} of a monomial $\vect{x}^{\boldsymbol{\lambda}_1}_{j_1} \cdots \vect{x}^{\boldsymbol{\lambda}_m}_{j_m}$ is defined to be the vector sum  $\boldsymbol{\lambda}_1 + \cdots + \boldsymbol{\lambda}_m$, and the $\vect{n}$th multigraded component $\Syml_{\vect{n}}$ consists of all multi-symmetric functions of multidegree $\vect{n}$. 

We will give three bases for $\Syml$ that extend well-known bases of symmetric functions. These bases are indexed by \newword{vector partitions} in the sense of \S\S\ref{SS:lpartite}, i.e., multisets of $\lev$-partite numbers $\svect{\lambda} = \{\svect{\lambda}_1, \ldots, \svect{\lambda}_m\}$. 

First, the \newword{monomial function} $m_{\svect{\lambda}}$ is defined to be the sum of all monomials
\[  
\vect{x}^{\svect{\lambda}_1}_{j_1} \cdots \vect{x}^{\svect{\lambda}_m}_{j_m}
\]
where $j_1, \ldots, j_m$ are distinct. For example, using the shorthand $x_i = x_i^{(0)}$, $\bar{x}_i = x_i^{(1)}$ and $\bar{\bar{x}}_i = x_i^{(2)}$, we have
\[ 
m_{\colvec{1 \\ 2 \\ 4} \colvec{3 \\ 0 \\ 1}} = \sum_{i\ne j} x_i^1 \, \bar{x}_i^2 \, \bar{\bar{x}}_i^4 \, x_j^3 \, \bar{\bar{x}}_j \quad \text{and} \quad
m_{\colvec{2 \\ 1} \colvec{3 \\ 3} \colvec{2 \\ 1} } = \sum_{i < k \text{ and } j\notin \{i,k\} } x_i^2 \, \bar{x}_i \, x_j^{3} \, \bar{x}_j^3 \, x_k^2 \, \bar{x}_k.
\]

Recall that for a color word $u=u_1 u_2 \ldots u_n \in [0, \lev-1]^n$, we defined the \newword{multidegree} of $u$ to be the vector of multiplicities of colors appearing in $u$ (see \eqref{E:mdeg}). Now for each $\vect{n} \in \NNl$ with weight $n = |\vect{n}|$, we can define the \newword{complete function} $h_{\vect{n}}$ by
\[
h_{\vect{n}} = \sumsub{u=u_1u_2\ldots u_n \in [0,\lev-1]^n \\ \mdeg(u) = \vect{n}} \sum_{j_1 \le j_2 \le \cdots \le j_n} 
x_{j_1}^{(u_1)} \, x_{j_2}^{(u_2)} \cdots x_{j_n}^{(u_n)}.
\]
For example, there are two color words $u = 01$ and $u=10$ of multidegree $\colvec{1\\1}$, so
\[ 
h_{\colvec{1 \\ 1}} = \sum_{i \le j} x_i \bar{x}_j + \sum_{i \le j} \bar{x}_i x_j = 2 m_{\colvec{1 \\ 1}} + m_{\colvec{1 \\ 0} \colvec{0 \\ 1}}.
\]
The basis $h_{\svect{\lambda}}$ is defined multiplicatively, by $h_{\svect{\lambda}} = h_{\svect{\lambda}_1} h_{\svect{\lambda}_1} \cdots $.

Lastly, we define the \newword{power sum function} $p_{\vect{n}}$ by
\[
p_{\vect{n}} = \sum_{j = 1}^\infty \vect{x}_j^{\vect{n}} = m_{\vect{n}}.
\]
and set $p_{\svect{\lambda}} = p_{\svect{\lambda}_1} p_{\svect{\lambda}_2} \cdots$. Again the multi-symmetric functions $p_{\svect{\lambda}}$ form a basis for $\Syml$.

We can define a coproduct in $\Syml$ by generalizing the coproduct of ordinary symmetric functions:
\[ 
\Delta(p_{\vect{n}}) = 1 \otimes p_{\vect{n}} + p_{\vect{n}} \otimes 1.
\]
Alternatively, the same coproduct can be defined by the usual method of introducing a duplicate set of variables $\YY$ and letting $\Delta(f(\XX)) =  f(\XX+\YY)$.

\subsection{The Hopf algebra $\QSyml$} 
\label{SS:QSyml}

This algebra is due to Poirier \cite{Poirier98}; our presentation is adapted from Novelli and Thibon \cite{NT04, NT08}. 

Let $\XX = X^0 \sqcup X^1 \sqcup \cdots \sqcup X^{\lev-1}$ be commutative colored variables as before. A formal power series in $\XX$ of finite degree is called a \newword{quasisymmetric function of level $\lev$} if the coefficients of the monomials $\vect{x}_1^{\vect{i_1}} \cdots \vect{x}_m^{\vect{i}_m}$ and $\vect{x}_{j_1}^{\vect{i_1}} \cdots \vect{x}_{j_m}^{\vect{i}_m}$ are equal whenever $j_1 < \cdots < j_m$ and $\vect{I} = (\vect{i}_1, \ldots, \vect{i}_m)$ is a vector composition. The set of quasisymmetric functions of level $\lev$ forms a multigraded vector space $\QSyml = \bigoplus_{\vect{n} \in \NNl} \QSyml_{\vect{n}}$ with various natural bases indexed by vector compositions. The $\vect{n}$th multigraded component $\QSyml_{\vect{n}}$ is the vector space of  homogeneous quasisymmetric functions of multidegree $\vect{n}$, where once again we define the multidegree of a monomial $\vect{x}_1^{\vect{i_1}} \cdots \vect{x}_m^{\vect{i}_m}$ as the vector sum ${\vect{i_1}} + \cdots + {\vect{i}_m}$.

Here we will consider two bases. First is the quasisymmetric analog of the monomial multi-symmetric functions $m_{\boldsymbol{\lambda}}$. For a nonempty vector composition $\vect{I} = (\vect{i}_1, \ldots, \vect{i}_m)$, define the \newword{level $\lev$ monomial quasisymmetric function} $M_{\vect{I}}$ by
\[ 
M_{\vect{I}} = \sum_{j_1 < \cdots < j_m} \vect{x}_{j_1}^{\vect{i}_1} \cdots \vect{x}_{j_m}^{\vect{i}_m} 
\]
and let $M_{\emptycomp} = 1$. For example,
\[  
M_{\left( \begin{smallmatrix} 1 & 3 \\2 & 0 \\ 4 & 1\end{smallmatrix}\right)} = \sum_{ i < j} x_i \,\bar{x}_i^2 \, \bar{\bar{x}}_i^4 \, x_j^3 \, \bar{\bar{x}}_j  
\quad \text{ and } \quad
M_{\left(\begin{smallmatrix} 2 & 3 & 2 \\ 1 & 3 & 1 \end{smallmatrix} \right)} = \sum_{i < j < k} x_i^2 \,\bar{x}_i \, x_j^3 \, \bar{x}_j^3 \, x_k^2 \, \bar{x}_k. 
\]
As in the level $1$ case (see, e.g., \cite[Lemma~3.3]{Ehrenborg96}), and as observed by Aval et al.\ \cite{ABB05} in the level-$2$ case, multiplying two monomial functions $M_{\vect{I}} M_{\vect{J}}$ can be described in terms of quasi-shuffling. A \newword{quasi-shuffle} of two vector compositions $\vect{I}$ and $\vect{J}$ is a shuffling of the columns of $\vect{I}$ with the columns of $\vect{J}$ in which two columns may be added together as they are shuffled past each other. An example should make this clear: 
\[ 
\begin{aligned}
  M_{\colvec{1 & 3 \\ 0 & 2}} M_{\colvec{2 & 1 \\ 5 & 0}} = & 
  M_{\colvec{1 & 3 & 2 & 1 \\ 0 & 2 & 5 & 0}} + M_{\colvec{1 & 5 & 1 \\ 0 & 7 & 0}} 
    + M_{\colvec{1 & 2 & 3 & 1 \\ 0 & 5 & 2 & 0}} + M_{\colvec{1 & 2 & 4 \\ 0 & 5 & 2}} 
    + M_{\colvec{1 & 2 & 1 & 3 \\0 & 5 & 0 & 2}} + M_{\colvec{3 & 3 & 1 \\ 5 & 2 & 0}} \\
  & + M_{\colvec{3 & 4 \\5 & 2}} + M_{\colvec{3 & 1 & 3 \\ 5 & 0 & 2}} 
    + M_{\colvec{ 2 & 1 & 3 & 1 \\ 5 & 0 & 2 & 0}} + M_{\colvec{2 & 1 & 4 \\ 5 & 0 & 2}} 
    + M_{\colvec{ 2 & 2 & 3 \\ 5 & 0 & 2}} + 2 M_{\colvec{2 & 1 & 1 & 0 \\ 5 & 0 & 0 & 2}}.
\end{aligned}
\]
Thus $\QSyml$ is a quasi-shuffle algebra. General properties of such algebras are discussed in \cite{Hoffman00, Loday07}. 

The coproduct is given on the monomial basis by
\[ 
\Delta(M_{\vect{I}}) = \sum_{\vect{I} = \vect{J} \vect{K}} M_{\vect{J}} \otimes M_{\vect{K}} ,
\]
the sum being over all ways of writing $\vect{I}$ as a concatenation of two (possibly empty) vector compositions $\vect{J}$ and $\vect{K}$. For example,
\[ 
\Delta \bigl(M_{\left( \begin{smallmatrix} 1 & 0 & 2 \\4 & 3 & 0 \end{smallmatrix} \right)} \bigr) = 
1 \otimes M_{\left( \begin{smallmatrix} 1 & 0 & 2 \\4 & 3 & 0 \end{smallmatrix} \right)} + 
M_{ \left( \begin{smallmatrix} 1 \\4 \end{smallmatrix} \right)} \otimes
M_{ \left( \begin{smallmatrix}  0 & 2 \\ 3 & 0 \end{smallmatrix} \right) } + 
M_{ \left( \begin{smallmatrix} 1 & 0 \\4 & 3  \end{smallmatrix} \right) } \otimes
M_{ \left( \begin{smallmatrix}  2 \\ 0 \end{smallmatrix} \right) } + 
M_{ \left( \begin{smallmatrix} 1 & 0 & 2 \\4 & 3 & 0 \end{smallmatrix} \right) } \otimes
1 .
\]
This coproduct can be equivalently defined by the usual method of introducing a duplicate set of variables $\YY$ and letting $\Delta(f(\XX)) =  f(\XX+\YY)$. 

It is clear from the definitions that $\Syml$ is a Hopf subalgebra of $\QSyml$, and that the monomial bases of these two algebras are related by
\[ 
m_{\svect{\lambda}} = \sum_{\mathrm{cols}(\vect{I}) = \svect{\lambda}} M_{\vect{I}}
\]
where $\mathrm{cols}(\vect{I})$ denotes the multiset of columns of $\vect{I}$.

Next we consider the basis of \newword{level $\lev$ quasi-ribbon functions} $F_{\vect{I}}$, defined by
\begin{equation} 
\label{E:Fdef}
F_{\vect{I}} 
= \sumsub{j_1 \le j_2 \le \cdots \le j_n \\ r \in \dofI(\vect{I}) \implies j_r < j_{r+1}}
 x_{j_1}^{(u_1)} \cdots x_{j_n}^{(u_n)}
 = \sum_{\vect{I} \cleq \vect{J}} M_{\vect{J}}
\end{equation}
where $\vect{I}$ is a nonempty vector composition of weight $n$, and $u_1 \cdots u_n = \cofI(I)$. For example, 
\begin{equation*} 
F_{\colvec{ 2 & 0 \\ 0 & 1 \\ 1 & 0}} 
= \sum_{i \le j \le k < \ell } x_i x_j \bar{\bar{x}}_k \bar{x}_\ell 
=
M_{\colvec{ 2 & 0 \\ 0 & 1 \\ 1 & 0}} + M_{\colvec{ 2 & 0 & 0 \\ 0 & 0 & 1 \\ 0 & 1 & 0}}
+ M_{\colvec{ 1 & 1 & 0 \\ 0 & 0 & 1 \\ 0 & 1 & 0}} + M_{\colvec{ 1 & 1 & 0 & 0 \\ 0 & 0 & 0 & 1 \\ 0 & 0 & 1 & 0}}.
\end{equation*}
This basis specializes to Gessel's fundamental basis \cite{Gessel83} when $\lev = 1$. Unlike the level $1$ case, however, the product $F_{\vect{I}} \, F_{\vect{J}}$ is not always $F$-positive. For example, 
\[
F_{\colvec{1 & 0 \\ 0 & 0 \\ 0 & 1}} F_{\colvec{0 \\ 1 \\ 0}} = 
F_{\colvec{0 & 1 & 0 \\ 1 & 0 & 0 \\ 0 & 0 & 1}} + F_{\colvec{1 & 0 \\ 1 & 0 \\ 0 & 1}} + 
F_{\colvec{1 & 0 \\ 0 & 1 \\ 0 & 1}} - F_{\colvec{1 & 0 & 0 \\ 0 & 1 & 0 \\ 0 & 0 & 1}} + 
F_{\colvec{1 & 0 & 0 \\ 0 & 0 & 1 \\ 0 & 1 & 0}}.
\]

\subsection{Duality between $\QSyml$ and $\NSyml$} 
\label{SS:Duality}
As Novelli and Thibon observed \cite{NT04, NT08}, the multigraded dual Hopf algebra of $\NSyml$ may be identified with $\QSyml$ by making $S^{\vect{I}}$ the dual basis of $M_{\vect{J}}$. More precisely, for each $\vect{n}$ we have $(\QSyml_{\vect{n}})^* = \NSyml_{\vect{n}}$, for any two vector compositions $\vect{I}$ and $\vect{J}$ we have $S^{\vect{I}}(M_{\vect{J}}) = \delta_{\vect{I}, \vect{J}}$, and for any $T, U \in \NSyml$ and $G, H \in \QSyml$, we have
\[ TU(G) = \sum T(G_1) U(G_2) \quad\text{and}\quad T(GH) = \sum T_1(G) T_2(H), \]
where $\Delta(G) = \sum G_1 \otimes G_2$ and $\Delta(T) = \sum T_1 \otimes T_2$ in Sweedler notation.

We can use this duality in several ways. As an example, if we
let $\apodeQ$ denote the antipode on $\QSyml$, then by taking the dual of \eqref{E:S-antipode}, we get
\begin{equation} \label{E:M-antipode}
\apodeQ(M_{\vect{I}}) = (-1)^{\len(\vect{I})} \sum_{\vect{J} \tleq \rev{\vect{I}}} M_{\vect{J}}.
\end{equation}
As before, $\rev{\vect{I}}$ is the vector composition obtained by reversing the order of the columns of $\vect{I}$.

There is yet another interesting pair of dual bases for these two Hopf algebras. Let us define $P_{\vect{I}}$ to be the basis of $\QSyml$ that is dual to the power sum basis $\Phi^{\vect{I}}$ of $\NSyml$. Thus, 
\[ 
\Phi^{\vect{I}}(P_{\vect{J}})  = \delta_{\vect{I}, \vect{J}}. 
\]
Then by \eqref{E:SPhi} we get
\begin{equation*}
P_{\vect{I}} = \sumsub{\vect{J} \tleq \vect{I}} \frac{1}{\spp(\vect{I}, \vect{J})} \, M_{\vect{J}}.
\end{equation*}
Since the $\Phi_{\vect{n}}$ are primitive, the $P_{\vect{I}}$ make up a shuffle basis for $\QSyml$. It is well-known that a shuffle algebra is freely generated by its Lyndon words; see, for example, \cite[Theorem~6.1]{Reutenauer93}. We will come back to these ideas in \S\S\ref{SS:koddQSym}.


\section{The category of multigraded combinatorial Hopf algebras}
\label{S:TheoryMCHA}



In this section we define the category of multigraded combinatorial Hopf algebras and derive a universal property satisfied by $\QSyml$ similar to, and inspired by, \cite[Theorem~4.1]{ABS06}. This leads to a systematic way to investigate morphisms $\hopf \to \QSyml$. The closely related category of colored combinatorial Hopf algebras was described by Bergeron and Hohlweg \cite{BerHoh06} and studied further in \cite{HP09}. We discuss how to relate our work to these earlier versions in the second half of this section. 

We will work over the rationals $\field$, but results in this section are valid over any field.

\subsection{Definitions}
\label{SS:DefMCHA} 

Let $\hopf = \bigoplus_{\vect{n} \in \NNl} \hopf_{\vect{n}}$ be a multigraded connected Hopf algebra over $\field$. Let $\phi : \hopf \to \field$ be a multiplicative invertible linear functional, or a {\em character}, on $\hopf$ (see \S\S\ref{SS:Invertible} for more on (convolution) invertible linear functionals). Then we say that the ordered pair $(\hopf, \phi)$ (or simply $\hopf$ if $\phi$ is unambiguous from the context) is a \newword{multigraded combinatorial Hopf algebra}. Morphisms in the category of multigraded combinatorial Hopf algebras are homomorphisms $\Psi: \hopf \to \hopf'$ of $\NNl$-graded Hopf algebras for the multigraded combinatorial Hopf algebras $(\hopf,\phi)$ and $(\hopf',\phi^{\prime})$ such that $\phi = \phi' \circ \Psi$.

When $\lev=1$ we recover the original construction of a combinatorial Hopf algebra in \cite{ABS06}.

Each of the algebras discussed in \S\ref{S:background} can be naturally made into a multigraded combinatorial Hopf algebra by pairing it with a character. In the context of this paper, the most important example is the Hopf algebra $\QSyml$ paired with the \newword{universal character $\zetaQ : \QSyml \to \field$}, defined by setting each of the variables $x_1^{(0)}, x_1^{(1)}, \ldots, x_1^{(\lev-1)}$ to $1$ and all other variables $x_j^{(i)}$, $j\ne 1$, to $0$. Alternatively, $\zetaQ$ can be defined by
\begin{equation} \label{E:zetaQSyml}
\zetaQ(M_{\vect{I}}) = \zetaQ(F_{\vect{I}}) = \begin{cases}
1 & \text{ if $\vect{I} = \emptycomp$ or $\len(\vect{I}) = 1$} \\
0 & \text{ otherwise.}
\end{cases}
\end{equation}
Clearly $\zetaQ$ is a character since it is an evaluation map. The significance of this character will be discussed next.

\subsection{Universality of the Hopf algebra $\QSyml$}
\label{SS:QSymUniversal}

The pair $(\QSyml, \zetaQ)$ is terminal in the category of multigraded combinatorial Hopf algebras; this generalizes the $\lev = 1$ case \cite[Theorem~4.1]{ABS06}. More specifically we have:

\begin{theorem}\label{T:QSymlUniversal}
For every multigraded combinatorial Hopf algebra $(\hopf, \zeta)$, there exists a unique morphism $\Psi:(\hopf, \zeta) \to (\QSyml, \zetaQ)$. Explicitly, if $\vect{n} \in \NNl$ and $h\in \hopf_{\vect{n}}$ then
\begin{equation}\label{E:InducedMap}
\Psi(h) = \sum_{\vect{I} \comp \vect{n}} \zeta_{\vect{I}}(h) M_{\vect{I}}
\end{equation}
where if $\vect{I} = (\vect{i}_1, \ldots, \vect{i}_p)$ then $\zeta_{\vect{I}}$ is the composite map 
\[
\hopf \xrightarrow{\Delta^{(p-1)}} \hopf^{\otimes p} \xrightarrow{\text{projection}} \hopf_{\vect{i}_1} \otimes \cdots \otimes \hopf_{\vect{i}_p} \xrightarrow{\zeta^{\otimes p}} \field^{\otimes p} \xrightarrow{\text{multiplication}} \field.
\]
\end{theorem}

This theorem is actually just the $\vect{k} = \vect{0}$ case of Theorem~\ref{T:kOddInducedMap} (and also of Theorem~\ref{T:kEvenInducedMap}), so we will defer the proof until \S\ref{S:defkoddkeven} and focus on examples for now.


\begin{example} \label{Ex:GradedPosets}

Recall that we introduced the Hopf algebra $\GPosl$ of multigraded posets in \S\S\ref{SS:GPosl}. Now we define $\zeta: \GPosl \to \field$ to be 
\[
\zeta(P) = 1 \text{ for every multigraded poset $P$.}
\]
Clearly $\zeta$ is a character, so $(\GPosl, \zeta)$ is a multigraded combinatorial Hopf algebra. Theorem \ref{T:QSymlUniversal} then implies that we have a morphism $\mathcal{F}: \GPosl \to \QSyml$ such that $\zetaQ \circ \mathcal{F} = \zeta$. As observed in \cite[Example~4.4]{ABS06}, the $\lev = 1$ version of this morphism is the $F$-homomorphism introduced by Ehrenborg \cite{Ehrenborg96}.

By \eqref{E:InducedMap}, for every multigraded poset $P$ of rank $\vect{n}$, we have
\begin{equation}\label{E:FPM}
\mathcal{F}(P) = \sum_{\vect{I} \comp \vect{n}} f_{\vect{I}}(P) \, M_{\vect{I}},
\end{equation} 
where
\[
\begin{aligned}
f_{\vect{i}_1 \vect{i}_2 \ldots \vect{i}_m}(P) = & \text{ number of chains } \hat{0} = t_0 < t_1 < \cdots < t_{m+1} = \hat{1} \text{ in $P$} \\
& \text{ such that } \rk(t_{j+1}) - \rk(t_{j}) = \vect{i}_j \text{ for $j=1,\ldots, m$}.
\end{aligned}
\]
Thus $(f_{\vect{I}}(P) \mid \vect{I} \comp \vect{n})$ is a refinement of the usual flag $f$-vector of a graded poset.

For example, if
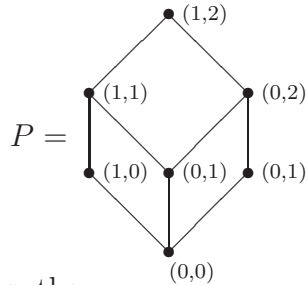
\[
\begin{picture}(30,90)(0,-40)
$P =$ 
\end{picture}
\setlength{\unitlength}{1pt}
\begin{picture}(60,90)
\put(30,0){\circle*{4}} \put(30,0){\line(-1,1){30}} \put(30,0){\line(0,1){30}} \put(30,0){\line(1,1){30}} \put(0,30){\circle*{4}} \put(0,30){\line(0,1){30}}
\put(30,30){\circle*{4}} \put(30,30){\line(-1,1){30}} \put(30,30){\line(1,1){30}}
\put(60,30){\circle*{4}} \put(60,30){\line(0,1){30}}
\put(0,60){\circle*{4}} \put(0,60){\line(1,1){30}}
\put(60,60){\circle*{4}} \put(60,60){\line(-1,1){30}}
\put(30,90){\circle*{4}}
\put(30,-10){$\scriptstyle{(0,0)}$} \put(5,28){$\scriptstyle{(1,0)}$}  
\put(35,28){$\scriptstyle{(0,1)}$} \put(65,28){$\scriptstyle{(0,1)}$}
\put(5,58){$\scriptstyle{(1,1)}$} \put(65,58){$\scriptstyle{(0,2)}$}
\put(35,88){$\scriptstyle{(1,2)}$}
\end{picture}
\]
where the labels represent multiranks, then
\[
\mathcal{F}(P) = 
M_{\colvec{1 \\ 2}} +  M_{\colvec{1 & 0 \\ 0 & 2}} + 2 M_{\colvec{0 & 1 \\ 1 & 1}} + M_{\colvec{1 & 0 \\ 1 & 1}} + M_{\colvec{0 & 1 \\ 2 & 0}} + M_{\colvec{1 & 0 & 0 \\ 0 & 1 & 1}} + M_{\colvec{0 & 1 & 0 \\ 1 & 0 & 1}} + 2 M_{\colvec{0 & 0 & 1 \\ 1 & 1 & 0}}.
\] 

Another way to express \eqref{E:FPM}, analogous to \cite[Equation~(1)]{Stanley96}, is
\[
\mathcal{F}(P) = 
\sum_{\hat{0} = t_0 \le t_1 \le \cdots \le t_{k-1} < t_k = \hat{1}} \vect{x}_1^{\rk(t_1) - \rk(t_0)} 
\vect{x}_2^{\rk(t_2) - \rk(t_1)} \cdots \vect{x}_k^{\rk(t_k) - \rk(t_{k-1})}
\]
where the sum is over all multichains in $P$ from $\hat{0}$ to $\hat{1}$ in which $\hat{1}$ occurs exactly once.
\end{example}


\begin{example}\label{Ex:EulerianPoset}
Recall that we defined the Hopf subalgebra $\GPoslk$ of $\GPosl$ in \S\S\ref{SS:GPoslk}. 
Now we consider the restriction of the character $\zeta$ from the above example to $\GPoslk$. For simplicity we also denote the restricted character by $\zeta$, so 
\[
\zeta: \GPoslk \to \field
\]
is given by $\zeta(P) = 1$ for every $\vect{k}$-Eulerian poset $P$.  Note that $\zeta^{-1}$ is just the M\"obius function $\mu$. Thus if $P$ is $\vect{k}$-Eulerian of multirank $\vect{n} \le \vect{k}$ (hence $P$ is Eulerian), then 
$\zeta^{-1}(P) = \mu(P) = (-1)^{|\vect{n}|} = \bar{\zeta}(P).$ We can therefore assert that
\begin{equation} \label{E:kEulerianCharacter}
\bar{\zeta}(h) = \zeta^{-1}(h) \quad \text{ for every $h \in \bigoplus_{\vect{n} \le \vect{k}} \GPoslk_{\vect{n}}$}.
\end{equation}
If every entry of $\vect{k}$ is $\infty$ (hence there are no restrictions on $h$), then \eqref{E:kEulerianCharacter} amounts to saying that $\zeta$ is an \newword{odd character}, in the sense of \cite{ABS06}. Thus \eqref{E:kEulerianCharacter} suggests a way to refine the definition of odd character. This is the motivation for our definition of a $\vect{k}$-odd character in \S\ref{SS:kOddFunctional}.
\end{example}


\begin{example} \label{Ex:ColoredPosets}
Recall that we defined the Hopf algebra $\NPl$ of colored posets in \S\S\ref{SS:ColoredPosets}. We will say that a colored poset $P$ is \newword{naturally labeled} if it has a linear extension of the form $(12\cdots n, u)$. Define $\zeta: \NPl \to \field$ by
\[
\zeta(P) = \begin{cases}
1 & \text{ if $P = \emptyset$ or $P$ is naturally labeled} \\
0 & \text{ otherwise.}
\end{cases}
\]
Clearly $P$ and $Q$ are naturally labeled if and only if $P \sqcup Q$ is. Thus $\zeta$ is a character, and $(\NPl, \zeta)$ is a multigraded combinatorial Hopf algebra. Let $\Gamma: \NPl \to \QSyml$ be the unique morphism of multigraded combinatorial Hopf algebras satisfying $\zetaQ \circ \Gamma = \zeta$. We will describe this map $\Gamma$ in a bit more detail shortly, but first we need a new notation. Recall that colored permutations were defined at the beginning of \S\S\ref{SS:FQSyml}. If $(\sigma, u)$ is a colored permutation of $n$ and $\Des(\sigma) = \{s_1 < s_2 < \cdots < s_p\}$, then set $s_0 = 0$ and $s_{p+1} = n$, and define $\wDes(\sigma,u)$ to be the vector composition of length $p+1$ whose $i$th column is the sum of columns $s_{i-1} + 1$ through $s_{i}$ in the vector composition $\vect{E}_{u} = (\vect{e}_{u_1}, \cdots, \vect{e}_{u_n})$. For example,
\[
\vect{E}_{120101} = \begin{pmatrix}
0 & 0 & 1 & 0 & 1 & 0 \\
1 & 0 & 0 & 1 & 0 & 1 \\
0 & 1 & 0 & 0 & 0 & 0 
\end{pmatrix}
\quad
\text{ and }
\quad
\wDes(251463, 120101) =
\begin{pmatrix}
0 & 2 & 0 \\
1 & 1 & 1 \\
1 & 0 & 0 
\end{pmatrix}.
\]
In short, $\wDes(\sigma, u)$ keeps track of the descent set of $\sigma$ along with the multiset of colors appearing in each run of ascents.

\begin{prop}
Let $P$ be a colored poset, and let $L(P)$ be the set of linear extensions of $P$. We have
\begin{equation} \label{E:GammaPM}
\Gamma(P) = \sum_{(\sigma, u) \in L(P)} \sum_{\wDes(\sigma, u)  \tleq  \vect{I} \tleq \vect{E}_u} M_{\vect{I}}.
\end{equation}
\end{prop}
\begin{proof}
For each subset $I\subseteq P$, let $\mdeg(I) = (n_0, \ldots, n_{\lev-1})^T \in \NNl$ where $n_i$ is the number of elements in $I$ with color $i$. By \eqref{E:InducedMap} we have
\begin{multline*}
\Gamma(P) = \sum_{\emptyset = I_0 \subsetneq I_1 \subsetneq I_2 \subsetneq \cdots \subsetneq I_m = P}
\zeta(\std(I_1\setminus I_0)) \zeta(\std(I_2\setminus I_1)) \cdots \zeta(\std(I_m \setminus I_{m-1})) 
\\ 
\cdot M_{(\mdeg(I_1\setminus I_0), \mdeg(I_2 \setminus I_1), \ldots, \mdeg(I_m \setminus I_{m-1}))}  
\end{multline*}
where the sum is over all chains of order ideals in $P$. By the definition of $\zeta$ we only need to sum over chains $\emptyset = I_0 \subsetneq \cdots \subsetneq I_m = P$ such that $\std(I_i \setminus I_{i-1})$ is naturally labeled for each $i$. For every such chain $C$, let $[I_i\setminus I_{i-1}]$ denote the word obtained by reading the elements of $I_i\setminus I_{i-1}$ in order of increasing absolute value, and let $\pi(C) = [I_1 \setminus I_0]\cdots [I_m\setminus I_{m-1}]$. Thus $\pi(C)$ is a linear extension of $P$, written as a concatenation of subwords with no descents. This linear extension (and its decomposition into subwords) contributes a term of the form $M_{\vect{I}}$, $\wDes(\pi(C)) \tleq \vect{I} \tleq \vect{E}_u$ to our expression for $\Gamma(P)$, where $u = \cofI(\pi(C))$. Conversely, for any linear extension $(\sigma, u) \in L(P)$ and any $\vect{I} = (\vect{i}_1, \ldots, \vect{i}_m)$ such that $\wDes(\sigma, u) \tleq \vect{I} \tleq \vect{E}_u$, we can find a unique chain $C = \{\emptyset = I_0 \subsetneq I_1 \subsetneq \cdots \subsetneq I_m = P\}$ with $\pi(C) = (\sigma, u)$ and $\mdeg(I_j \setminus I_{j-1}) = \vect{i}_j$ for every $j$  by letting $I_1$ be the first $|\vect{i}_1|$ elements of $(\sigma, u)$, $I_2 \setminus I_1$ be the next $|\vect{i}_2|$ elements of $(\sigma, u)$, and so forth.
\end{proof}

Another way to get from $\NPl$ to $\QSyml$ is through the Hopf algebra $\GPosl$ of multigraded posets. If $P$ is a colored poset, let $J(P)$ as usual denote the poset (distributive lattice) of order ideals of $P$ ordered by inclusion. We make $J(P)$ into a multigraded poset by defining the multirank of an order ideal $I \in J(P)$ to be the multidegree of $I$ as a colored poset (i.e., the vector of multiplicities of the colors appearing in $I$). This gives us a map 
\[J: \NPl \to \GPosl\] 
which is easily shown to be a morphism of Hopf algebras using elementary properties of order ideals. Notice that $J(P)$ depends only on the colors of the elements of $P$ and not on their absolute values. When $\lev = 1$ this is essentially the map considered in \cite[Example~2.4]{ABS06}. 
\end{example}


\begin{example} \label{Ex:FQSyml}

Recall that we defined the Hopf algebra $\FQSyml$ in \S\S\ref{SS:FQSyml}. Now we define $\zeta: \FQSyml \to \field$ by 
\[
\zeta(\vect{F}_{\sigma, u}) = \begin{cases}
1 & \text{ if $\sigma$ is the identity permutation in $\perm_n$ for some $n$}\\
0 & \text{ otherwise.}
\end{cases}
\]
It is easy to see that $\zeta$ is a character, so by Theorem \ref{T:QSymlUniversal}, there is a unique morphism of multigraded combinatorial Hopf algebras $\mathcal{D} : \FQSyml \to \QSyml$ such that $\zetaQ \circ \mathcal{D} = \zeta$.

Now using \eqref{E:InducedMap}, we can see that
\begin{equation}\label{E:DFM}
\mathcal{D}(\vect{F}_{\sigma, u}) 
= \sum_{\wDes(\sigma, u) \; \tleq \; \vect{I} \; \tleq \; (\vect{e}_{u_1}, \ldots, \vect{e}_{u_n})} M_{\vect{I}} 
= \sumsub{j_1 \le j_2 \le \cdots \le j_n \\ r \in \Des(\sigma) \implies j_r < j_{r+1}}
x_{j_1}^{(u_1)} \, x_{j_2}^{(u_2)} \cdots x_{j_n}^{(u_n)}.
\end{equation}
As a special case of \eqref{E:DFM}, if $\Des(u) \subseteq \Des(\sigma)$ then 
\begin{equation}\label{E:DofF}
\mathcal{D}(\vect{F}_{\sigma, u}) = F_{\wDes(\sigma, u)}.
\end{equation}
It also follows from \eqref{E:DFM} that for every $\vect{n} \in \NNl$,
\[ 
\mathcal{D}(\vect{S}_{\vect{n}}) = h_{\vect{n}}. 
\]
Thus $\Syml$ is the commutative image of $\NSyml$ under $\mathcal{D}$.

Novelli and Thibon \cite{NT08} noted that each $F_{\vect{I}}$ arises as the commutative image of certain $\vect{F}_{\sigma, u}$ under the abelianization map $\mathrm{ab}: \field \langle\vect{A}\rangle \onto \field [\XX]$,  $(j, i) \mapsto x_j^{(i)}$. Comparing \eqref{E:FFdef} with \eqref{E:DFM}, we see that the morphism $\mathcal{D}$ is in fact just the restriction of $\mathrm{ab}$. To summarize, the diagram
\[
\xymatrix{
\NSyml \ar@{^{(}->}[r] \ar@{->>}[d]_{\mathcal{D}} &
\FQSyml \ar@{^{(}->}[r] \ar@{->>}[d]_{\mathcal{D}} & 
\field\langle\langle \vect{A} \rangle\rangle \ar@{->>}[d]_{\mathrm{ab}} \\
\Syml \ar@{^{(}->}[r] & \QSyml \ar@{^{(}->}[r] & \field [[ \XX ]]
}
\]
is commutative  (cf.\ \cite[\S1.3]{AS05}). 
\end{example}


To see the more general picture, we consider once again the setup from
\S\S\ref{Ex:ColoredPosets}.
The map $\Gamma: \NPl \to \QSyml$ defined explicitly in Equation \eqref{E:GammaPM} factors through $\FQSyml$ via the morphism $\widehat{\Gamma}: \NPl \onto \FQSyml$ given by 
\[
\widehat{\Gamma}(P) = \sum_{(\sigma, u) \in L(P)} \vect{F}_{\sigma, u}.
\]
When $\lev = 1$, $\widehat{\Gamma}(P)$ is the \newword{free quasisymmetric generating function of $P$} introduced by Duchamp et al.\ \cite[Definition~3.15]{DHT02}. Combining \eqref{E:DFM} and \eqref{E:GammaPM}, we see that $\Gamma = \mathcal{D} \circ \widehat{\Gamma}$. 

We can summarize the relationship between the various examples discussed so far in the following commutative diagram:
\[
\xymatrix{
& & \NPl \ar@{->>}[dl]_{\textstyle \widehat{\Gamma}} \ar@{->}[dd]^{\textstyle \Gamma}   & \NPl \ar@{->}[d]^{\textstyle J} \\
 & \FQSyml \ar@{->>}[dr]_{\textstyle \mathcal{D}} & & \GPosl \ar@{->}[dl]^{\textstyle \mathcal{F}} \\
   \NSyml \ar@{^{(}->}[ur] \ar@{->>}[dr]_{\textstyle \mathcal{D}} & & \QSyml &  \\
 & \Syml \ar@{^{(}->}[ur] & &
}
\]
Note that $\Gamma \ne \mathcal{F} \circ J$ because $\Gamma(P)$ depends on the absolute values of the elements of $P$ while $J(P)$ does not.

\subsection{The Bergeron-Hohlweg theory of colored combinatorial Hopf algebras}
\label{SS:EarlyColoredHAs}

We now briefly review the theory of $m$-colored combinatorial Hopf algebras \cite{BerHoh06}. Our presentation follows \cite{HP09}. Then we describe how to relate this theory to our work here. 

An $\lev$-colored combinatorial Hopf algebra is a pair $(\hopf, \dot{\phi})$ where $\hopf= \bigoplus_{n \in \NN} \hopf_n$ is a graded connected Hopf algebra and $\dot{\phi} = (\phi^{(0)}, \phi^{(1)}, \cdots, \phi^{(\lev-1)})$ is an $\lev$-tuple of characters ${\phi^{(i)} : \hopf \rightarrow \Q}$, $i=0, 1, \cdots, \lev-1$. Morphisms in the category of $\lev$-colored combinatorial Hopf algebras are maps $\Psi:\hopf \to \hopf'$ such that ${\phi^{(i)}}'\circ \Psi = {\phi^{(i)}}$ for $i=0,1,\ldots, \lev-1$.

\begin{remark}
To see explicitly how the $\lev$-colored combinatorial Hopf algebras $(\hopf, \phi)$ of \cite{BerHoh06} correspond to the $\lev$-colored combinatorial Hopf algebras $(\hopf, \dot{\phi})$ of \cite{HP09}, recall first that for \cite{BerHoh06}, an $\lev$-colored combinatorial Hopf algebra is a pair $(\hopf, \phi)$ where $\hopf= \bigoplus_{n \in \NN} \hopf_n$ and $\phi: \hopf \rightarrow \Q[\mathcal{C}_{\lev}]$. Here $\mathcal{C}_{\lev}$ is the cyclic group generated by $w$, a primitive $\lev$th root of unity. Elements of $\Q[\mathcal{C}_{\lev}]$ are polynomial expressions of the form $q_0+q_1w +q_2w^2 + \cdots + q_{\lev-1}w^{\lev-1}$ where $q_i \in \Q$ for each $i=0, 1, \cdots, \lev-1$. To make these two definitions of $\lev$-colored combinatorial Hopf algebras compatible, the multiplication in $\Q[\mathcal{C}_{\lev}]$ should be defined on the color components via: 
\[ w^i w^j = \begin{cases}
w^i & \textmd{ if } i=j;\\
0 & \textmd{ otherwise}
\end{cases} \]
and extended linearly \cite{HP09}. Then we can identify any particular map $\phi : \hopf \rightarrow \Q[\mathcal{C}_{\lev}]$ with an $\lev$-tuple $\dot{\phi} = (\phi^{(0)}, \phi^{(1)}, \cdots, \phi^{(\lev-1)})$ by first defining the functions $q_i : \hopf \rightarrow \Q$ such that for any $h \in \hopf$ we have:
\[ 
\phi(h) = q_0(h) +q_1(h) w +q_2(h) w^2 + \cdots + q_{\lev-1}(h) w^{\lev-1} . 
\]
Then we simply set $\phi^{(i)} = q_i$ for $i=0,1,\cdots,\lev-1$. We will have:
\[ 
\phi(h) = \phi^{(0)}(h) + \phi^{(1)}(h)w + \phi^{(2)}(h)w^2 + \cdots + \phi^{(\lev-1)}(h)w^{\lev-1} \textmd{ for all } h \in \hopf. 
\]
\end{remark}

The terminal object in this category is the Hopf algebra $\QSymll$ of $\lev$-colored quasisymmetric functions, first studied by Poirier \cite{Poirier98}. The bases of $\QSymll$ are typically indexed by \newword{$\lev$-colored compositions} $\alpha = ((\alpha_1, u_1), \ldots, (\alpha_m, u_m))$, where $(\alpha_1, \ldots, \alpha_m)$ is an integer composition (vector of positive integers) and $u_1 \ldots u_m \in [0,\lev-1]^m$ is a color vector. In the following we will write $\alpha \comp_l \, n$ when $\alpha$ is an $\lev$-colored composition of $n$.
A natural basis for $\QSymll$ is the set of \newword{colored monomial quasisymmetric functions} $M^{(\lev)}_{\alpha}$, defined by
\[
M^{(\lev)}_\alpha = \sum_{(j_1, u_1) <_{\text{lex}} \cdots <_{\text{lex}} (j_m, u_m)} (x_{j_1}^{(u_1)})^{\alpha_1} (x_{j_2}^{(u_2)})^{\alpha_2} \cdots (x_{j_m}^{(u_m)})^{\alpha_m}
\]  
where $<_{\text{lex}}$ refers to the lexicographic order on $\PP \times [0, \lev-1]$. 

Baumann and Hohlweg \cite{BauHoh08} proved that $\QSymll$ is a Hopf subalgebra of $\QSyml$. A different but isomorphic realization of $\QSymll$ was introduced by Novelli and Thibon \cite{NT04}. To see that $\QSymll \subseteq \QSyml$, note that the monomial functions $M^{(\lev)}_\alpha$ can be written in the monomial basis $M_{\vect{I}}$ as follows (recall that $\vect{e}_i$ denotes the $i$th coordinate vector in $\NNl$ and that the partial order $\cleq_s$ is defined in \S\ref{SS:lpartite}):
\begin{equation}\label{E:monomialColor}
M^{(\lev)}_\alpha = \sum_{\vect{I} \cleq_s (\alpha_1 \cdot \vect{e}_{u_1}, \ldots, \alpha_m \cdot \vect{e}_{u_m})} M_{\vect{I}}
\end{equation}
with $\alpha = ((\alpha_1, u_1), \ldots, (\alpha_m, u_m))$. 
For example,
\begin{eqnarray*}
M_{((2,1),(1,0),(1,2),(3,0))}^{(3)} &=& \sum_{i_1 < i_2 < i_3 < i_4}
(x_{i_1}^{(1)})^{2}(x_{i_2}^{(0)})(x_{i_3}^{(2)})(x_{i_4}^{(0)})^{3} + \sum_{i_1 < i_2  < i_3}
(x_{i_1}^{(1)})^{2}(x_{i_2}^{(0)})(x_{i_2}^{(2)})(x_{i_3}^{(0)})^{3} \\
&=& 
M_{\left (\begin{smallmatrix}
0&1&0&3\\
2&0&0&0\\
0&0&1&0\end{smallmatrix}\right )} +
M_{\left (\begin{smallmatrix}
0&1&3\\
2&0&0\\
0&1&0
\end{smallmatrix}\right )}.
\end{eqnarray*}
Details of the Hopf algebra structure of $\QSymll$ can be found in \cite{BerHoh06, HP09, NT04, NT08}. 

The universal character $\dot{\psi} = (\psi^{(0)}, \ldots, \psi^{(\lev-1)})$ of $\QSymll$ is defined as a tuple of evaluation characters: each $\psi^{(i)}$ takes the variable $x^{(i)}_1$ to $1$ and all other variables to $0$. Equivalently,
\[
\psi^{(i)}(M^{(\lev)}_\alpha) = 
\begin{cases}
1 & \text{ if $\alpha = \emptycomp$ or $\alpha = ((a, i))$} \\
0 & \text{ otherwise.}
\end{cases}
\]
In \cite[Theorem 13]{HP09} it is shown that for any colored combinatorial Hopf algebra $(\hopf, \dot{\phi})$, there is a unique morphism $\Psi_c : (\hopf, \dot{\phi}) \rightarrow (\QSymll, \dot{\psi})$ of colored combinatorial Hopf algebras given explicitly by 
\begin{equation} \label{E:InducedMapColor}
\Psi_c(h) = \sum_{\alpha \comp_l n} \phi_{\alpha}(h) M_{\alpha}^{(l)},
\end{equation}
for any $h \in \hopf_n$, where for $\alpha = (\omega^{j_1}\alpha_1,\dots, \omega^{j_k}\alpha_k)$, $\phi_{\alpha}$ is the composite map:
\[ \hopf \xrightarrow{\Delta^{(k-1)}} \hopf^{\otimes k} \xrightarrow{\text{projection}}  \hopf_{\alpha_1} \otimes \cdots \otimes \hopf_{\alpha_k} 
\xrightarrow{\phi^{(j_1)} \otimes \cdots \phi^{(j_k)}} \field^{\otimes k} \xrightarrow{m} \field. \]

We will now see how $\QSyml$ fits into this picture.
Let $\zeta^{(i)}$, $i = 0, 1, \ldots, \lev-1$ be the character on $\QSyml$ that takes $x_1^{(i)}$ to $1$ and all other variables to $0$. Let $\dot{\zeta} = (\zeta^{(0)}, \ldots, \zeta^{(\lev-1)})$, so $(\QSyml, \dot{\zeta})$ becomes an $\lev$-colored combinatorial Hopf algebra. Note that we are ignoring the multigrading on $\QSyml$ and thinking of it as an ordinary graded Hopf algebra $\QSyml = \bigoplus_{n=0}^\infty \QSyml_n$, where $\QSyml_n = \bigoplus_{|\vect{n}| = n} \QSyml_{\vect{n}}$. 

Since $\psi^{(i)}$ defined above is just the restriction of $\zeta^{(i)}$ to the subalgebra $\QSymll$, it follows that $(\QSymll, \dot{\psi})$ is a combinatorial Hopf subalgebra of $(\QSyml, \dot{\zeta})$.

\begin{prop} 
In the category of colored combinatorial Hopf algebras, the morphism
\[
(\QSymll, \dot{\psi}) \hookrightarrow (\QSyml, \dot{\zeta})
\]
is the inclusion map.
\end{prop}

Conversely, the result quoted above \cite[Theorem~13]{HP09} implies that there is a morphism going in the other direction, $(\QSyml, \dot{\zetaQ}) \to (\QSymll,\dot{\psi})$. We now describe this map more explicitly. A vector composition $\vect{I}$ will be called \newword{monochromatic} if it is of the form $\vect{I} = (\alpha_1 \cdot \vect{e}_{u_1}, \alpha_2 \cdot \vect{e}_{u_2}, \ldots, \alpha_m \cdot \vect{e}_{u_m})$ for some $u_1, \ldots, u_m \in [0,\lev-1]$. In this case we define $w(\vect{I})$ to be the colored composition $w(\vect{I}) = ( (\alpha_1, u_1), \ldots, (\alpha_m, u_m))$. The following can be proved by a straightforward application of \eqref{E:InducedMapColor}.

\begin{prop} \label{P:QSymltoQSymll}
Let $\Psi_c:\QSyml \to \QSym^{[\lev]}$ be the unique morphism of graded Hopf algebras satisfying $\psi^{(i)}\circ \Psi_c = \zetaQ^{(i)}$ for all $i$. For every vector composition $\vect{I}$, we have
\[
\Psi_c(M_{\vect{I}}) = \begin{cases}
M_{w(\vect{I})} & \text{if $\vect{I}$ is monochromatic} \\
0 & \text{otherwise.}
\end{cases}
\]
\end{prop}

A consequence of Proposition~\ref{P:QSymltoQSymll} is that $\QSymll$ can be identified with $\QSyml$ modulo the relations $x_i^{(p)} x_i^{(q)}$ for $p\ne q$ and $i=1,2,\ldots$. This realization of $\QSymll$  was first given by Novelli and Thibon \cite{NT04}.

Lastly, one can view $\QSymll$ as a multigraded combinatorial Hopf algebra by pairing it with the character $\psi = \psi^{(0)} \psi^{(1)} \cdots \psi^{(\lev-1)}$. Then Theorem \ref{T:QSymlUniversal} implies the existence of a morphism $\Psi : \QSymll \rightarrow \QSyml$ of multigraded combinatorial Hopf algebras satisfying $\zetaQ \circ \Psi = \psi$. It is clear from the definitions that
\[
\zetaQ(M_{\alpha}^{(\lev)}) = \psi(M_{\alpha}^{(\lev)}) =
\begin{cases}
1 & \text{ if $\alpha = \emptycomp$ or $\alpha = ((\alpha_1, u_1), \ldots, (\alpha_m, u_m))$, $u_1 < u_2 < \cdots < u_m$} \\
0 & \text{ otherwise.}
\end{cases}
\]
Therefore $\Psi$ is the inclusion map. In other words, the following holds:

\begin{prop}
In the category of multigraded combinatorial Hopf algebras, the morphism
\[
(\QSymll, \psi) \hookrightarrow (\QSyml, \zetaQ)
\]
is the inclusion map.
\end{prop}


\section{Definitions and basic properties of $\vect{k}$-odd and $\vect{k}$-even Hopf algebras}
\label{S:defkoddkeven}


\def\field{\mathbb{Q}}

In this section we develop the notions of $\vect{k}$-odd and $\vect{k}$-even subalgebras of multigraded combinatorial Hopf algebras. Our constructions and results directly generalize the definitions and basic properties of odd and even Hopf subalgebras developed by Aguiar et al.\ \cite{ABS06}.

\subsection{Basic constructions}

Let $\hopf = \bigoplus_{\vect{n} \in \NNl} \hopf_{\vect{n}}$ be a multigraded connected Hopf algebra. For a linear functional $\phi : \hopf \to \field$, let $\phi_{\vect{n}}$ denote the restriction of $\phi$ to $\hopf_{\vect{n}}$. This is an element of degree $\vect{n}$ of the (multi)graded dual $\hopf^*$. We also define $\overline{\phi}$ to be the functional given by $\overline{\phi}(h) = (-1)^{|\vect{n}|} \phi(h)$ for $h \in \hopf_{\vect{n}}$. 

Let $\phi, \psi : \hopf \to \field$ be characters. The canonical Hopf subalgebra $\Ssubalg(\phi,\psi)$ and its orthogonal Hopf ideal $\ideal(\phi,\psi)$ are defined in \cite[Section 5]{ABS06} for $\NN$-graded Hopf algebras. Generalizing to the $\NNl$-graded case, we define $\Ssubalg(\phi,\psi)$ to be the largest graded subcoalgebra of $\hopf$ such that 
\[ \forall h \in \Ssubalg(\phi,\psi), \quad \phi(h) = \psi(h),\]
\noindent
and $\ideal(\phi,\psi)$ to be the ideal of $\hopf^*$ generated by $\phi_{\vect{n}} - \psi_{\vect{n}}$ for each $\vect{n} \in \NNl$.

We now define $\vect{k}$-analogs of these. 

Let $\lev$ be a positive integer, to be fixed for the rest of this section. It will be convenient to work with the ``extended" $\lev$-partite numbers $\NNinfl$, where the symbol $\infty$ is understood to be larger than every natural number. 

For $\vect{k} \in \NNinfl$, let $\Ssubalg_{\vect{k}}(\phi,\psi)$ denote the largest graded subcoalgebra of $\hopf$ with the property that 
\begin{equation}
\label{SamDef}
\phi(h) = \psi(h) \textmd{ for all } h \in \Ssubalg_{\vect{k}}(\phi,\psi) \cap \bigoplus_{\vect{n} \le \vect{k}} \hopf_{\vect{n}}.
\end{equation}
(When forming direct sums it should be understood that our indices lie in $\NNl$.) Equation \eqref{SamDef} defines $\Ssubalg_{\vect{k}}(\phi,\psi)$ as the largest graded subcoalgebra of $\hopf$ whose graded pieces up to degree $\vect{k}$ are all contained in $\ker{(\phi-\psi)}$. In other words, $\Ssubalg_{\vect{k}}(\phi,\psi)$ is the largest graded subcoalgebra of $\hopf$ whose intersection with $\bigoplus_{\vect{n} \le \vect{k}} \hopf_{\vect{n}}$, the $\vect{k}$-initial graded piece of $\hopf$, lies in the kernel of $\phi-\psi$. 

Note that this last statement is equivalent to asserting that the graded pieces up to degree $\vect{k}$ of $\Ssubalg_{\vect{k}}(\phi,\psi)$ lie in  $\bigoplus_{\vect{n} \le \vect{k}} (\ker{\phi_{\vect{n}} - \psi_{\vect{n}}})$. Thus, we can alternatively define $\Ssubalg_{\vect{k}}(\phi,\psi)$ as the largest graded subcoalgebra of $\hopf$ with the property that 
\begin{equation}
\label{AltDef}
 \forall h \in \Ssubalg_{\vect{k}}(\phi,\psi), \quad \phi_{\vect{n}}(h) = \psi_{\vect{n}}(h) \textmd{ for all } \vect{n} \le \vect{k}.
\end{equation}
\noindent

Next we let $\ideal^{\vect{k}}(\phi,\psi)$ denote the ideal of the graded dual $\hopf^*$ generated by $\phi_{\vect{n}} - \psi_{\vect{n}}$ for each $\vect{n} \le \vect{k}$.  Each $\ideal^{\vect{k}}(\phi,\psi)$ is generated by homogeneous elements and so is a (multi)graded ideal of $\hopf^*$.

Recall that $\vect{0} \in \NNl$ denotes the zero vector. Let $\vectinf \in \NNinfl$ denote the vector whose entries are all $\infty$. For any $\vect{k} \in \NNinfl$, 
\[ \Ssubalg_{\vect{0}}(\phi,\psi) \supset \Ssubalg_{\vect{k}}(\phi,\psi) \supset \Ssubalg_{\vectinf}(\phi,\psi) = \Ssubalg(\phi,\psi)\]
\noindent
and
\[ \ideal^{\vect{0}}(\phi,\psi) \subset \ideal^{\vect{k}}(\phi,\psi) \subset \ideal^{\vectinf}(\phi,\psi) = \ideal(\phi,\psi).\]

The properties of $\Ssubalg(\phi,\psi)$ and $\ideal(\phi,\psi)$ stated in \cite[Theorem 5.3]{ABS06}, along with their proofs, extend without difficulty to their $\vect{k}$-analogs:

\begin{theorem}
\label{theoremgeneralproperties}
Let $\hopf = \bigoplus_{\vect{n} \in \NNl} \hopf_{\vect{n}}$ be a multigraded connected Hopf algebra and let ${\phi, \psi : \hopf \to \field}$ be characters on $\hopf$. Define $\Ssubalg_{\vect{k}}(\phi,\psi)$ and $\ideal^{\vect{k}}(\phi,\psi)$ as above. 
For $\vect{k} \in \NNinfl$ the following properties hold:
\begin{enumerate}
\item[(a)] $\Ssubalg_{\vect{k}}(\phi,\psi) = (\ideal^{\vect{k}}(\phi,\psi))^{\perp}$;
\item[(b)] $\ideal^{\vect{k}}(\phi,\psi)$ is a graded Hopf ideal of $\hopf^*$;
\item[(c)] $\Ssubalg_{\vect{k}}(\phi,\psi)$ is a graded Hopf subalgebra of $\hopf$.
\end{enumerate}
\noindent
Here, $\ideal^{\vect{k}}(\phi,\psi)^{\perp}$ is the set $\{ h \in \hopf : f(h) = 0 \textmd{ for all } f \in \ideal^{\vect{k}}(\phi,\psi)\}$.
\end{theorem}

The proof of this result follows the main lines of the proof of \cite[Theorem 5.3]{ABS06}. We include it here in order to demonstrate the general feel of the proofs of some of the more straightforward extensions of the results of \cite{ABS06} to their $\vect{k}$-analogs:

\begin{proof}
We begin with part (a). 
Since $\ideal^{\vect{k}}(\phi,\psi)$ is a (multi)graded ideal of $\hopf^*$, $\ideal^{\vect{k}}(\phi,\psi)^{\perp}$ is a (multi)graded subcoalgebra of $\hopf$. Let $h = \sum_{\vect{i}\in \NNl} h_{\vect{i}} \in \ideal^{\vect{k}}(\phi,\psi)^{\perp}$, where all but finitely many of the $h_{\vect{i}}$ are zero and hence the sum is finite. Because $\phi_{\vect{i}} - \psi_{\vect{i}} \in \ideal^{\vect{k}}(\phi,\psi)$ for all $\vect{i} \le \vect{k}$, we have $\phi_\vect{i}(h) = \phi_{\vect{i}}(h_{\vect{i}}) = \psi_{\vect{i}}(h_{\vect{i}}) = \psi_{\vect{i}}(h)$ for all $\vect{i} \le \vect{k}$. Then it follows from the definition of $\Ssubalg_{\vect{k}}(\phi,\psi)$ (as the greatest subcoalgebra of $\hopf$ satisfying Equation \eqref{AltDef}) that $h \in \Ssubalg_{\vect{k}}(\phi,\psi)$. 

Next let $C$ be a graded subcoalgebra of $\hopf$ such that for all $\vect{i} \le \vect{k}$, $\phi_{\vect{i}}(h) = \psi_{\vect{i}}(h)$ for all $h = \sum_{\vect{i}\in \NNl} h_{\vect{i}} \in C$; here once again we assume that all but finitely many of the terms $h_{\vect{i}}$ are zero. Since $C$ is a coalgebra and $\ideal^{\vect{k}}(\phi,\psi)$ is the ideal generated by $\phi_{\vect{i}} - \psi_{\vect{i}}$ for $\vect{i} \le \vect{k}$, it follows that $f(C) = 0$ for all $f \in \ideal^{\vect{k}}(\phi,\psi)$. This shows that $\Ssubalg_{\vect{k}}(\phi,\psi) \subset \ideal^{\vect{k}}(\phi,\psi)^{\perp}$ and concludes the proof of part (a).

To prove parts (b) and (c), we begin by noting that the product $C \cdot D$ of two graded subcoalgebras of $\hopf$ is again a graded subcoalgebra. This makes $\Ssubalg_{\vect{k}}(\phi,\psi) \cdot \Ssubalg_{\vect{k}}(\phi,\psi)$ a graded subcoalgebra of $\hopf$. By the multiplicativity of 
$\phi$ and $\psi$ we have:
\begin{equation}
\label{ResMultChar} 
(\phi-\psi)(xy) = (\phi-\psi)(x)\phi(y) + \psi(x)(\phi-\psi)(y)åÊ
\end{equation}
\noindent
Now if $x, y \in \Ssubalg_{\vect{k}}(\phi,\psi)$ such that $xy \in \bigoplus_{\vect{n} \le \vect{k}} \hopf_{\vect{n}}$, we also have $x,y \in \bigoplus_{\vect{n} \le \vect{k}} \hopf_{\vect{n}}$ and so by Equation \eqref{SamDef}, we have $(\phi-\psi)(x) = (\phi-\psi)(y)=0$. 
Equation \eqref{ResMultChar} then gives us $(\phi-\psi)(xy)=0$. This proves that $\Ssubalg_{\vect{k}}(\phi,\psi) \cdot \Ssubalg_{\vect{k}}(\phi,\psi) \subset \Ssubalg_{\vect{k}}(\phi,\psi)$.

Finally we note that $\hopf_{\vect{0}} = \field \cdot 1$ is a graded subcoalgebra of $\hopf$
and $\phi(1) = \psi(1) = 1$ so we can conclude that $\hopf_{\vect{0}}$ is included in $\Ssubalg_{\vect{k}}(\phi,\psi)$. This proves part (c), or in other words, that $\Ssubalg_{\vect{k}}(\phi,\psi)$ is indeed a Hopf subalgebra of $\hopf$. Together with part (a), this implies that $\ideal^{\vect{k}}(\phi,\psi)$ is a coideal of $\hopf^*$, and thus we are also done with part (b). 
\end{proof}

\begin{remark}
\label{RemDeltaSquared}
A natural $\vect{k}$-analog of part (d) of \cite[Thm.5.3]{ABS06} is also valid: More specifically one can easily show that a homogeneous element $h \in \hopf$ belongs to $\Ssubalg_{\vect{k}}(\phi,\psi)$ if and only if 
\[ (\textmd{id} \otimes (\phi_{\vect{n}}-\psi_{\vect{n}})\otimes \textmd{id} ) \circ \Delta^{(2)} (h) = 0\]
for all $\vect{n} \in \NNl$ such that $\vect{n} \le \vect{k}$.
\end{remark}

In the next proposition we list a few properties of $\Ssubalg_{\vect{k}}(\phi,\psi)$ and the associated ideals $\ideal^{\vect{k}}(\phi,\psi)$. 
\begin{prop}
Let $\hopf = \oplus_{\vect{n} \in \NNl} \hopf_{\vect{n}}$ be a multigraded connected Hopf algebra and let $\phi,\phi^{\prime},\psi,\psi^{\prime}$ be characters on $\hopf$. The following hold for all $\vect{k} \in \NNinfl$:
\begin{enumerate}
\item[(a)] There is an isomorphism of graded Hopf algebras 
\[ \Ssubalg_{\vect{k}}(\phi,\psi) \cong \hopf^* / \ideal^{\vect{k}}(\phi,\psi). \] 
\item[(b)] Suppose that
\[ \psi^{-1}\phi = (\psi^{\prime})^{-1}\phi^{\prime} \textmd{ or }  \phi \psi^{-1} = \phi^{\prime} (\psi^{\prime})^{-1}.\]
Then 
\[ \Ssubalg_{\vect{k}}(\phi,\psi) = \Ssubalg_{\vect{k}}(\phi^{\prime}, \psi^{\prime}) \textmd{ and } \ideal^{\vect{k}}(\phi,\psi) = \ideal^{\vect{k}}(\phi^{\prime}, \psi^{\prime}). \]
\item[(c)] $\Ssubalg_{\vect{k}}(\phi,\psi) = \Ssubalg_{\vect{k}}(\psi, \phi)$ and $\ideal^{\vect{k}}(\phi,\psi) = \ideal^{\vect{k}}(\psi, \phi)$. 
\end{enumerate}
\end{prop}

\begin{proof}
Parts (a) and (b) are the $\vect{k}$-analogs of Corollary 5.4 and Proposition 5.5 of \cite{ABS06}, respectively, and their proofs follow the proofs of those results in a straightforward manner. Part (c) follows directly from the definitions.  
\end{proof}

\subsection{$\vect{k}$-analogs of odd and even subalgebras}

Theorem \ref{theoremgeneralproperties} justifies the following generalization of the definition of the {\it odd subalgebra $\Ssubalg_-(\hopf,\phi) = \Ssubalg(\overline{\phi}, \phi^{-1})$} of a combinatorial Hopf algebra $(\hopf,\phi)$ \cite[Def.5.7]{ABS06}:
\begin{definition}
Given a multigraded combinatorial Hopf algebra $(\hopf,\phi)$ and $\vect{k} \in \NNinfl$, we call $\Ssubalg_{\vect{k}}(\overline{\phi}, \phi^{-1})$ the \emph{$\vect{k}$-odd Hopf subalgebra} of $(\hopf,\phi)$ and denote it by $\kodd^{\vect{k}}(\hopf,\phi)$, or simply by $\kodd^{\vect{k}}(\hopf)$ when $\phi$ is obvious from context.
\end{definition}

Similarly we can make the following definition (cf.\ \cite[Def.5.7]{ABS06}):
\begin{definition}
Given a multigraded combinatorial Hopf algebra $(\hopf,\phi)$ and $\vect{k} \in \NNinfl$, we call $\Ssubalg_{\vect{k}}(\overline{\phi}, \phi)$ the \emph{$\vect{k}$-even Hopf subalgebra} of $(\hopf,\phi)$ and denote it by $\keven^{\vect{k}}(\hopf, \phi)$, or simply by $\keven^{\vect{k}}(\hopf)$ when $\phi$ is obvious from context.
\end{definition}

We introduce a special notation for the associated ideals:
\begin{definition}
Given a multigraded combinatorial Hopf algebra $(\hopf,\phi)$ and $\vect{k} \in \NNinfl$, we call $\ideal^{\vect{k}}(\overline{\phi}, \phi^{-1})$ the \emph{$\vect{k}$-odd Hopf ideal} of $(\hopf,\phi)$ and denote it by $\ideal\kodd^{\vect{k}}(\hopf, \phi)$, or simply by $\ideal\kodd^{\vect{k}}(\hopf)$ when $\phi$ is obvious from context.
Similarly, we call $\ideal^{\vect{k}}(\overline{\phi}, \phi)$ the \emph{$\vect{k}$-even ideal} of $(\hopf,\phi)$ and denote it by $\ideal\keven^{\vect{k}}(\hopf, \phi)$, or simply by $\ideal\keven^{\vect{k}}(\hopf)$ when $\phi$ is obvious from context.
\end{definition}

For future reference, we collect together a few basic properties of $\vect{k}$-odd and $\vect{k}$-even subalgebras in the next proposition:
\begin{prop}
Let $(\hopf,\phi)$ be a multigraded combinatorial Hopf algebra. 
\begin{enumerate}
\item[(a)] $\kodd^{\vect{k}}(\hopf)$ and $\keven^{\vect{k}}(\hopf)$ are multigraded Hopf subalgebras of $\hopf$, for each $\vect{k} \in \NNinfl$.
\item[(b)] $\ideal\kodd^{\vect{k}}(\hopf)$ and $\ideal\keven^{\vect{k}}(\hopf)$ are multigraded Hopf ideals of $\hopf^*$, for each $\vect{k} \in \NNinfl$.
\item[(c)] There are isomorphisms of multigraded Hopf algebras 
\[ \kodd^{\vect{k}}(\hopf) \cong \hopf^* / \ideal\kodd^{\vect{k}}(\hopf) \textmd{ and } 
\keven^{\vect{k}}(\hopf) \cong \hopf^* / \ideal\keven^{\vect{k}}(\hopf) \]
for each $\vect{k} \in \NN$.
\item[(d)] Let $\vect{k} \in \NNinfl$. $\kodd^{\vect{k}}(\hopf)$ is the largest subcoalgebra of $\hopf$ with the property that 
\[ \phi^{-1}(h) = (-1)^{|\vect{n}|}\phi(h) \]
for every $h \in \kodd^{\vect{k}}(\hopf)$ of degree $\vect{n} \le \vect{k}$. Similarly $\keven^{\vect{k}}(\hopf)$ is the largest subcoalgebra of $\hopf$ with the property that 
\[ \phi(h) = (-1)^{|\vect{n}|}\phi(h) \]
for every $h \in \kodd^{\vect{k}}(\hopf)$ of degree $\vect{n} \le \vect{k}$.
\end{enumerate}
\end{prop}

\begin{proof}
Parts (a), (b) and (c) are simple specializations of earlier results. Part (d) follows from Equation \eqref{AltDef}, which allows us to describe $\kodd^{\vect{k}}(\hopf)$ as the largest subcoalgebra of $\hopf$ with the property that 
\[ (\phi^{-1})_{\vect{n}}(h) = (-1)^{|\,\vect{i}\,|}\phi_{\vect{n}}(h) \textmd{ for all } \vect{n} \le \vect{k}\]
for every $h \in \kodd^{\vect{k}}(\hopf)$ of degree $\vect{i}$, and $\keven^{\vect{k}}(\hopf)$ as the largest subcoalgebra of $\hopf$ with the property that 
\[ \phi_{\vect{n}}(h) = (-1)^{|\,\vect{i}\, |}\phi_{\vect{n}}(h) \textmd{ for all } \vect{n} \le \vect{k}\]
for every $h \in \kodd^{\vect{k}}(\hopf)$ of degree $\vect{i}$.
\end{proof}

\subsection{Invertible linear functionals}
\label{SS:Invertible}

Let $\hopf$ be a multigraded connected Hopf algebra over $\field$. Recall that the convolution product of two linear functionals $\phi, \psi : \hopf \to \field$ is given as:
\[ \hopf \to \hopf \otimes \hopf \to \field \otimes \field \to \field\]
\noindent
 where the arrows are, respectively, $\Delta_{\hopf}$, $\phi \otimes \psi$ and $m_{\field}$. In the following, we choose simplicity and write convolution by concatenation; in other words, we denote the convolution of $\phi$ and $\psi$ simply by $\phi \psi$. Moreover we simply say \emph{invertible} when we mean \emph{convolution invertible}.

We know that the set $\mathbb{X}(\hopf)$ of characters of an arbitrary Hopf algebra $\hopf$ is a group under the convolution product, where the unit element is given by the counit $\epsilon_{\hopf}$ of $\hopf$ and the inverse of a given element $\phi$ of $\mathbb{X}(\hopf)$ is $\phi^{-1} = \phi \circ \apode_{\hopf}$. Here $\apode_{\hopf}$ is the antipode of $\hopf$. It is easy to see that $\phi(\phi \circ \apode_{\hopf}) = (\phi \circ \apode_{\hopf}) \phi = \epsilon_{\hopf}$. 

In this paper, invertible linear functionals play an important role. Therefore we now focus on the notion of invertibility and collect together some facts about invertible linear functionals on a combinatorial Hopf algebra. Here is a basic characterization of invertibility, which is noted in \cite{AH04}:

\begin{lemma}\label{L:invertibility}
Let $\phi:\hopf \to \field$ be a linear functional. Then $\phi$ is invertible if and only if $\phi(1) \ne 0$. 
\end{lemma}

Because of this lemma, we will make the reasonable assumption that all of our linear functionals satisfy $\phi(1) \ne 0.$ In fact, all the linear functionals we will be considering will satisfy $\phi(1) = 1$. 

Here are  two more simple observations about invertible linear functionals.

\begin{lemma}\label{L:inverse-equal}
Let $\phi$ and $\rho$ be invertible linear functionals on $\hopf$. Assume that $\phi(1) = \rho(1) = 1$. Let $\vect{k} \ge \vect{0}$. If $\phi_{\vect{n}} = \rho_{\vect{n}}$ for all $\vect{n} \le \vect{k}$, then $(\phi^{-1})_{\vect{n}} = (\rho^{-1})_{\vect{n}}$ for all $\vect{n} \le \vect{k}$.
\end{lemma}
\begin{proof}
We prove this by induction on $\vect{n}$. For the base case we have $(\phi^{-1})_{\vect{0}}=(\rho^{-1})_{\vect{0}} = \epsilon$. Let $\vect{n}$ be such that $\vect{0} < \vect{n} \le \vect{k}$. Then by induction we have
\[ (\phi^{-1})_{\vect{n}} = - \sum_{\vect{0}~\ne~\vect{i}~\le~\vect{n} } \phi_{\vect{i}} (\phi^{-1})_{\vect{n} - \vect{i}} = - \sum_{\vect{0}~\ne~\vect{i}~\le~\vect{n} } \rho_{\vect{i}} (\rho^{-1})_{\vect{n} - \vect{i}} = (\rho^{-1})_{\vect{n}}.\qedhere\]
\end{proof}

\begin{lemma}
Let $\phi$ be an invertible linear functional on $\hopf$ such that $\phi(1) = 1$. Then $\overline{\phi}^{-1} = \overline{\phi^{-1}}$.
\end{lemma}
\begin{proof}
It is easy to verify that $(\overline{\phi}^{-1})_{\vect{0}} = (\overline{\phi^{-1}})_{\vect{0}} = \epsilon$. For $\vect{n} > \vect{0}$, by induction we have
\begin{eqnarray*}
(\overline{\phi}^{-1})_{\vect{n}} &=& -\sum_{\vect{0}~\ne~\vect{i}~\le~\vect{n} } \overline{\phi}_{\vect{i}} (\overline{\phi}^{-1})_{\vect{n-i}}  = -\sum_{\vect{0}~\ne~\vect{i}~\le~\vect{n} } \overline{\phi_{\vect{i}}} (\overline{\phi^{-1}})_{\vect{n-i}} \\ &=& - \sum_{\vect{0}~\ne~\vect{i}~\le~\vect{n} } \overline{\phi_{\vect{i}} (\phi^{-1})_{\vect{n-i}}} = \overline{(\phi^{-1})_{\vect{n}}} = (\overline{\phi^{-1}})_{\vect{n}}. \qedhere 
\end{eqnarray*}
\end{proof}

\subsection{$\vect{k}$-odd linear functionals} \label{SS:kOddFunctional}

We next recall the definition of an odd character from \cite{ABS06}: A character $\phi$ of an $\NN$-graded Hopf algebra $\hopf$ is \textit{odd} if $\overline{\phi} = \phi^{-1}$. Here the bar denotes the involution $\phi \mapsto \overline{\phi}$ on the characters of $\hopf$ defined by $\overline{\phi}(h) = (-1)^n \phi(h)$ for $h \in \hopf_n$. Recall also that at the beginning of this section, we introduced the analogous involution for the multigraded case: $\overline{\phi}(h) = (-1)^{|\vect{n}|} \phi(h)$ for $h \in \hopf_{\vect{n}}$. 

In order to define the $\vect{k}$-analogs of odd characters, we once again focus on invertibility first. We begin with the following:
\begin{definition}
A (convolution) invertible linear functional $\phi$ is called {\em $\vect{k}$-odd} if $(\overline{\phi})_{\vect{n}} = (\phi^{-1})_{\vect{n}}$ for all $\vect{n} \le \vect{k}.$
\end{definition}

Note that for $l =1$, $\vect{k}$ is simply a natural number $k$, and 
an odd character in the sense of \cite{ABS06} is $k$-odd for all $k \in \NN$. Thus the following is a most natural notion to introduce:
\begin{definition}
A (convolution) invertible linear functional $\phi$ on $\hopf$ is called \emph{odd} if it is $\vect{k}$-odd for all $\vect{k} \in \NNl$ (equivalently, for all $\vect{k} \in \NNinfl$).
\end{definition}

In the rest of this section, we will use the notation  $\kodd^{\vect{k}}$ for $\kodd^{\vect{k}}(\QSyml)$, the $\vect{k}$-odd subalgebra $\Ssubalg_{\vect{k}}(\overline{\zetaQ}, (\zetaQ)^{-1})$ of $\QSyml$. Similarly we will use the notation $\keven^{\vect{k}}$ for $\keven^{\vect{k}}(\QSyml)$, the $\vect{k}$-even subalgebra $\Ssubalg_{\vect{k}}(\overline{\zetaQ}, \zetaQ)$ of $\QSyml$.

Here is the main result of this subsection:

\begin{theorem}
\label{T:kOddInducedMap}
Let $\hopf$ be a multigraded Hopf algebra $\hopf$ and let $\vect{k} \in \NNinfl$.
\begin{enumerate}
\item If $\phi : \hopf \to \field$ is a $\vect{k}$-odd linear functional on $\hopf$, then there exists a unique morphism $\Psi : \hopf \to \QSyml$ of $\NNl$-graded coalgebras such that $\zetaQ \circ \Psi = \phi$. The image of $\Psi$ lies in $\kodd^{\vect{k}}$. 
\item Explicitly,  if $\vect{n} \in \NNl$ and $h\in \hopf_{\vect{n}}$ then
\begin{equation}\label{E:OddInducedMap}
\Psi(h) = \sum_{\vect{I} \comp \vect{n}} \phi^{\vect{I}}(h)M_{\vect{I}}, 
\end{equation}
where if $\vect{I} = (\vect{i}_1, \ldots, \vect{i}_m) \comp \vect{n}$ then $\phi^{\vect{I}}$ is the composite map:
\[ 
\hopf \xrightarrow{\Delta^{(m-1)}} \hopf^{\otimes m} \xrightarrow{\text{projection}}  \hopf_{\vect{i}_1} \otimes \hopf_{\vect{i}_2} \otimes \cdots \otimes \hopf_{\vect{i}_m} \xrightarrow{\phi^{\otimes m}} \field^{\otimes m} \xrightarrow{\text{multiplication}} \field.
\]
\item If $\phi$ is a character then $\Psi$ is a homomorphism of multigraded (combinatorial) Hopf algebras. In other words, $(\kodd^{\vect{k}}, \zetaQ)$ is the terminal object of the category of multigraded (combinatorial) Hopf algebras with $\vect{k}$-odd characters.
\end{enumerate}
\end{theorem}

\begin{remark}
For $\vect{k} = \vect{0}$, $\kodd^{\vect{k}}$ is equal to $\QSyml$, and the above theorem reduces to Theorem~\ref{T:QSymlUniversal}.
\end{remark}

\begin{proof}
For a linear functional $\phi  : \hopf \to \field$, Theorem~4.1 of \cite{ABS06} provides us with a unique graded coalgebra map $\Psi$ between $\hopf$ and $\QSyml$ satisfying $\zetaQ \circ \Psi = \phi$ provided $\lev = 1$. Here we are interested in general $\lev$. Moreover, the statement we are making applies to a certain class of linear functionals on $\NNl$-graded connected Hopf algebras, the $\vect{k}$-odd ones. In this case, our theorem asserts that the image of the relevant morphism lies in $\kodd^{\vect{k}}$. Below we follow the construction in the proof of Theorem~4.1 of \cite{ABS06} carefully and modify as necessary to make sure that we get what we want. 

We first construct a map $\Phi : \NSyml \to \hopf^*$. Recall from \S\S\ref{SS:NSyml} that $\NSyml$ is freely generated as an algebra by $\{S_{\vect{n}} \mid \vect{n} \in \NNl\}.$ We let $\Phi$ be the algebra homomorphism that maps $S_{\vect{n}}$ to $\phi_{\vect{n}}$, the restriction  of $\phi$ to $\hopf_{\vect{n}}$. Clearly $\phi_{\vect{n}}$ is in $(\hopf_{\vect{n}})^* = (\hopf^*)_{\vect{n}}$, so $\Phi$ preserves the $\NNl$-grading. 

We next set $\Psi = \Phi^*$ be the dual map from $\hopf$ into ${\NSyml}^* = \QSyml$. In particular these two maps ought to satisfy
\[ 
\Phi(S_{\vect{n}})(h) = (S_{\vect{n}} \circ \Psi)(h) \textmd{ or equivalently, } (\phi_{\vect{n}}(h) =S_{\vect{n}} \circ \Psi)(h).
\]
Then $\Psi$ is a graded coalgebra map. 

Now since as an element of ${\QSyml}^* = \NSyml$, the $\vect{n}$th graded piece of $\zetaQ$ is $S_{\vect{n}}$, we have:
\[ \left . \zetaQ \circ \Psi \right \vert_{{\hopf}_{\vect{n}}} = \left . \zetaQ \right \vert_{\QSyml_{\vect{n}}} \circ \left . \Psi \right \vert_{{\hopf}_{\vect{n}}} = S_{\vect{n}} \circ \left . \Psi \right \vert_{\hopf_{\vect{n}}}  =  \Phi(S_{\vect{n}}) = \phi_{\vect{n}},\] 
and therefore $\zetaQ \circ \Psi = \phi$. This shows that $\Psi : \hopf \to \QSyml$ is a morphism of combinatorial coalgebras.

Next for any vector composition $\vect{I} = (\vect{i}_1, \ldots, \vect{i}_m)\comp \vect{n}$ define $\phi^{\vect{I}}$ to be the composition: 
\[ \hopf \to \hopf^{\otimes m} \to \hopf_{\vect{i}_1} \otimes \hopf_{\vect{i}_2} \otimes \cdots \otimes \hopf_{\vect{i}_m} \to \field^{\otimes m} \xrightarrow{\text{multiplication}} \field\]
where the unlabeled arrows stand for $\Delta^{(m-1)}$, the tensor product of the canonical projections onto the appropriate homogeneous components, and $\phi^{\otimes m}$, respectively. Since $S^{\vect{I}} = S_{\vect{i}_1} \cdots S_{\vect{i}_m}$, we can see that $\Phi(S^{\vect{I}}) = \phi^{\vect{I}}$, and so $\Psi$ is given by:
\[ \Psi(h) = \sum_{\vect{I} \comp \vect{n}} \phi^{\vect{I}}(h)M_{\vect{I}}, \]
where $h \in \hopf_{\vect{n}}$.
Uniqueness of $\Psi$ follows from the uniqueness of $\Phi$ by duality. 

Finally we need to show that the image of $\Psi$ lies in $\kodd^{\vect{k}}$.
Now, if $\phi$ is $\vect{k}$-odd, then by definition, we have: $(\overline{\phi})_{\vect{n}} = (\phi^{-1})_{\vect{n}}$ for all $\vect{n}\le \vect{k}$. But then, since $\Ssubalg_{\vect{k}}(\overline{\phi}, \phi^{-1})$ is the largest subcoalgebra of $\hopf$ satisfying 
\[ \forall h \in \Ssubalg_{\vect{k}}(\overline{\phi},\phi^{-1}), \quad \overline{\phi}_{\vect{n}}(h) = (\phi^{-1})_{\vect{n}}(h) \textmd{ for all } \vect{n} \le \vect{k}, \]
(cf. Equation \eqref{AltDef}), we can easily see that
\[ 
\Ssubalg_{\vect{k}}(\overline{\phi}, (\phi)^{-1}) = \hopf.
\]
But since we have $\kodd^{\vect{k}} = \Ssubalg_{\vect{k}}(\overline{\zetaQ}, (\zetaQ)^{-1})$, our statement reduces to:
\[ \Psi(\Ssubalg_{\vect{k}}(\overline{\phi}, (\phi)^{-1})) \subset \Ssubalg_{\vect{k}}(\overline{\zetaQ}, (\zetaQ)^{-1}).\]
This will follow readily from a modification of Prop.5.6(a) of \cite{ABS06} (also see \cite[Prop.5.8(e)]{ABS06}):
\begin{lemma}
\label{OddsMap2OddsLemma}
Let $\phi$ and $\psi$ be linear functionals on the multigraded coalgebra $\hopf$, and let $\phi^{\prime}$ and $\psi^{\prime}$ be linear functionals on the multigraded coalgebra $\hopf^{\prime}$. Let $\Psi : \hopf\to \hopf^{\prime}$ be a morphism of multigraded coalgebras with $\phi = \phi^{\prime} \circ \Psi$ and $\psi = \psi^{\prime} \circ \Psi$. Then $\Psi(\Ssubalg_{\vect{k}}({\phi}, \psi)) \subset \Ssubalg_{\vect{k}}({\phi^{\prime}}, \psi^{\prime}))$ for each $\vect{k} \in \NNl$.
\end{lemma}
\noindent
Modulo the proof of this proposition (which can be obtained by a simple modification of the relevant arguments in \cite{ABS06}), we are done with the proof of part (a). 

Part (b) follows from the observation that $\zetaQ \circ \Psi = \phi$ now implies that $\Psi$ is in fact a morphism of multigraded combinatorial Hopf algebras. The argument follows the same route as that in the proof of Theorem 4.1 in \cite{ABS06}. In particular we consider the two commutative diagrams:
\[
\xymatrix{
 \hopf^{\otimes2} \ar@{->}[r]^{m} \ar@{->}[dr]_{\phi^{\otimes 2}} & \hopf \ar@{->}[r] \ar@{->}[d]_{\phi} & \QSyml \ar@{->}[dl]^{\zetaQ} \\
 & \field & 
}
\qquad \text{and} \qquad 
\xymatrix{
 \hopf^{\otimes2} \ar@{->}[r] \ar@{->}[dr]_{\phi^{\otimes 2}} & (\QSyml)^{\otimes 2} \ar@{->}[r]^{m} \ar@{->}[d]_{\zetaQ^{\otimes 2}} & \QSyml \ar@{->}[dl]^{\zetaQ} \\
 & \field & 
}
\]
where the unlabeled arrows represent $\Psi$ and $\Psi^{\otimes 2}$ respectively, and $m$ stands for the multiplication in the appropriate space.  
The fact that all the arrows in both diagrams are graded coalgebra maps, together with the universal property of $\QSyml$ as a combinatorial coalgebra which has already been established, implies that the two diagrams can be glued together to obtain $\Psi \circ m = m \circ \Psi^{\otimes 2}$. From this we can conclude that $\Psi$ indeed is a morphism of (combinatorial Hopf) algebras. 

Note that the above implies that $\Psi$ is multiplicative if $\phi$ is. This gives us the following formula which is not obvious from basic definitions: Given $h_1 \in \hopf_{\vect{n}_1}$, $h_2 \in \hopf_{\vect{n}_2}$, (and so $h_1h_2 \in \hopf_{\vect{n}_1 +\vect{n}_2}$), we have:
\[ \Psi(h_1h_2) = \sum_{\vect{I} \comp \vect{n}_1+\vect{n}_2} \phi^{\vect{I}}(h_1h_2)M_{\vect{I}} = \sum_{\vect{I}_1 \comp \vect{n}_1} \phi^{\vect{I}_1}(h_1)M_{\vect{I}_1}  \sum_{\vect{I}_2 \comp \vect{n}_2} \phi^{\vect{I}_2}(h_2)M_{\vect{I}_2} = \Psi(h_1)\Psi(h_2). \qedhere  \]
\end{proof}

\begin{remark}
A $\vect{k}$-analog of part (b) of Proposition 5.6 from \cite{ABS06} can also be proved, but we will not need it in this paper.
\end{remark}

\begin{example} \label{Ex:GeneralizedDH}
Recall that we defined the Hopf algebra $\GPoslk$ of $\vect{k}$-Eulerian posets in \S\S\ref{SS:GPoslk} and described a suitable character $\zeta$ on it in Example~\ref{Ex:EulerianPoset}; also see Equation~\eqref{E:FPM}. 
Now we can see from Equation~\eqref{E:kEulerianCharacter}  that this $\zeta$ is indeed a $\vect{k}$-odd character. Thus if $P$ is $\vect{k}$-Eulerian, then $\mathcal{F}(P) \in \kodd^{\vect{k}}$. This in turn implies that the flag numbers $f_{\vect{I}}(P)$ must satisfy certain linear relations. To understand what these relations are, first notice that
\[
S^{\vect{I}}(\mathcal{F}(P)) = f_{\vect{I}}(P)
\]
for every vector composition $\vect{I}$ and every multigraded poset $P$ of multirank $\vect{n} = \Sigma\vect{I}$. Therefore, given scalars $a_{\vect{I}} \in \field$, $\vect{I} \comp \vect{n}$, we have
\[
\sum_{\vect{I} \comp \vect{n}} a_{\vect{I}} \, S^{\vect{I}} \in \ideal\kodd^{\vect{k}}(\QSyml) \implies
\sum_{\vect{I} \comp \vect{n}} a_{\vect{I}} \, f_{\vect{I}}(P) = 0 
\text{ for all $P$ such that } \mathcal{F}(P) \in \kodd^{\vect{k}}.
\]
The ideal $\ideal\kodd^{\vect{k}}(\QSyml)$ is described explicitly in \S\ref{S:koddkevenQSym}; see in particular Theorems~\ref{T:koddQSym}, \ref{T:kEulerIdeal} and \ref{T:koddQSym-eta} and Corollary~\ref{C:EulerCharacter}. We continue with this example in Remark~\ref{Re:GeneralizedDH} where we explicitly describe the linear equations the flag numbers $f_{\vect{I}}(P)$ must satisfy. More specifically we show there that the natural $\vect{k}$-analogues of the generalized Dehn-Sommerville equations hold for all $\vect{k}$-Eulerian posets.
\end{example}

Here is a characterization of $\vect{k}$-odd characters in terms of the $\vect{k}$-odd subalgebra of a multigraded combinatorial Hopf algebra which follows easily from definitions:
\begin{prop}
Let $(\hopf, \phi)$ be a multigraded combinatorial Hopf algebra. Then $\phi$ is $\vect{k}$-odd if and only if $\kodd^{\vect{k}}(\hopf) = \hopf$.
\end{prop}

One can also relate $\vect{k}$-odd characters on a combinatorial Hopf algebra to $\vect{k}$-odd subalgebras of the dual Hopf algebra:
\begin{prop}
Let $\hopf$ be a multigraded connected Hopf algebra with a character $\phi: \hopf \to \field$ and let $\eta : \hopf^* \to \field$ be any character on the dual Hopf algebra $\hopf^*$. If $\phi$ is $\vect{k}$-odd, then for $\vect{n} \le \vect{k}$, the homogeneous component $\phi_{\vect{n}}$ belongs to $\kodd^{\vect{k}}(\hopf^*, \eta)$.
\end{prop}

\begin{proof}
This is a straightforward $\vect{k}$-analog of Proposition 5.9 of \cite{ABS06}, and the proof follows similarly. Also see Remark \ref{RemDeltaSquared}.
\end{proof}

We will 
study a most fundamental example of $\vect{k}$-odd functionals in Section \ref{S:ktheta}.

\subsection{$\vect{k}$-even linear functionals} 

We will now briefly ponder the question of what we can say about the $\vect{k}$-analogs of even characters.
Recall from  \cite{ABS06} that a character $\phi$ of a graded Hopf algebra $\hopf$ is \textit{even} if $\overline{\phi} = \phi$.
Therefore, we will define $\vect{k}$-even functionals as follows:  

\begin{definition}
A (convolution) invertible linear functional $\phi$ is called {\em $\vect{k}$-even} if $(\overline{\phi})_{\vect{n}} = (\phi)_{\vect{n}}$ for all $\vect{n} \le \vect{k}.$
\end{definition}

Note that when $l =1$, $\vect{k}$ is simply a natural number $k$, and 
an even character in the sense of \cite{ABS06} is $k$-even for all $k \in \NN$. Thus the following is a most natural notion to introduce:
\begin{definition}
A (convolution) invertible linear functional $\phi$ on $\hopf$ is called \textbf{even} if it is $\vect{k}$-even for all $\vect{k} \in \NNl$ (equivalently for all $\vect{k} \in \NNinfl$).
\end{definition}

Recall that we use the notation $\keven^{\vect{k}}$ for $\keven^{\vect{k}}(\QSym)$, the $\vect{k}$-even subalgebra $\Ssubalg_{\vect{k}}(\overline{\zetaQ}, \zetaQ)$ of $\QSym$. 
With these definitions one can prove the following result analogous to Theorem~\ref{T:kOddInducedMap}:
\begin{theorem}
\label{T:kEvenInducedMap}
Let $\hopf$ be a multigraded Hopf algebra $\hopf$ and let $\vect{k} \in \NNinfl$.
\begin{enumerate}
\item If $\phi : \hopf \to \field$ is a $\vect{k}$-even linear functional on $\hopf$, then there exists a unique morphism $\Psi : \hopf \to \QSyml$ of $\NNl$-graded coalgebras such that $\zetaQ \circ \Psi = \phi$. The image of $\Psi$ lies in $\keven^{\vect{k}}$.
\item  Explicitly,  if $\vect{n} \in \NNl$ and $h\in \hopf_{\vect{n}}$ then
\begin{equation}\label{E:EvenInducedMap}
\Psi(h) = \sum_{\vect{I} \comp \vect{n}} \phi^{\vect{I}}(h)M_{\vect{I}}, 
\end{equation}
where if $\vect{I} = (\vect{i}_1, \ldots, \vect{i}_m) \comp \vect{n}$ then $\phi^{\vect{I}}$ is the composite map:
\[ \hopf \xrightarrow{\Delta^{(m-1)}} \hopf^{\otimes m} \xrightarrow{\text{projection}}  \hopf_{\vect{i}_1} \otimes \hopf_{\vect{i}_2} \otimes \cdots \otimes \hopf_{\vect{i}_m} \xrightarrow{\phi^{\otimes m}} \field^{\otimes m} \xrightarrow{\text{multiplication}} \field.\]
\item If $\phi$ is a character then $\Psi$ is a homomorphism of multigraded (combinatorial) Hopf algebras. In other words, $(\keven^{\vect{k}}, \zetaQ)$ is the terminal object of the category of multigraded (combinatorial) Hopf algebras with $\vect{k}$-even characters.
\end{enumerate}
\end{theorem}
The proof is similar to that of Theorem~\ref{T:kOddInducedMap} and will be skipped.

\begin{remark}
For $\vect{k} = \vect{0}$, $\keven^{\vect{k}}$ is equal to $\QSyml$, and the above theorem reduces to Theorem~\ref{T:QSymlUniversal}.
\end{remark}

Here is a characterization of $\vect{k}$-even characters in terms of the $\vect{k}$-even subalgebra of a multigraded combinatorial Hopf algebra which follows easily from definitions:
\begin{prop}
Let $(\hopf, \phi)$ be a multigraded combinatorial Hopf algebra. $\phi$ is $\vect{k}$-even if and only if $\keven^{\vect{k}}(\hopf) = \hopf$.
\end{prop}

One can also relate $\vect{k}$-even characters on a combinatorial Hopf algebra to $\vect{k}$-even subalgebras of the dual Hopf algebra:
\begin{prop}
Let $\hopf$ be a multigraded connected Hopf algebra with a character $\phi: \hopf \to \field$ and let $\eta : \hopf^* \to \field$ be any character on the dual Hopf algebra $\hopf^*$. If $\phi$ is $\vect{k}$-even, then for $\vect{n} \le \vect{k}$, the homogeneous component $\phi_{\vect{n}}$ belongs to $\keven^{\vect{k}}(\hopf^*, \eta)$.
\end{prop}

\begin{proof}
This is a straightforward $\vect{k}$-analog of Proposition 5.9 of \cite{ABS06}, and the proof follows similarly. Also see Remark \ref{RemDeltaSquared}.
\end{proof}

\section{The $\vect{k}$-odd and $\vect{k}$-even Hopf subalgebras of $\QSyml$} 
\label{S:koddkevenQSym}

Throughout this section, fix $\lev > 0$ and $\vect{k} = (k_0, \ldots, k_{\lev-1})^T \in \NNinfl$, and let $\zetaQ$ denote the universal character on $\QSyml$ as defined in \S\S\ref{SS:DefMCHA}. In this section we give two bases for the $\vect{k}$-odd Hopf algebra
\[ 
\kodd^{\vect{k}} = \kodd^{\vect{k}}(\QSyml, \zetaQ)
\]
and compute its Hilbert series. We also describe explicitly the ideal
\[
  \ideal\kodd^{\vect{k}} =   \ideal\kodd^{\vect{k}}(\QSyml, \zetaQ).
  \]
At the end of the section, we discuss very briefly the $\vect{k}$-even Hopf algebra
\[ 
\keven^{\vect{k}} = \keven^{\vect{k}}(\QSyml, \zetaQ).
\]

\subsection{The $\vect{k}$-odd Hopf subalgebra of $\QSyml$} 
\label{SS:koddQSym}

Recall that the level $\lev$ noncommutative power sum symmetric functions $\Phi^{\vect{I}}$ of $\NSyml$ were introduced in \S\S\ref{SS:NSyml}, and in \S\S\ref{SS:Duality} we defined the functions $P_{\vect{I}}$ which form the dual basis in $\QSyml$.

\begin{theorem}\label{T:koddQSym}
The canonical $\vect{k}$-odd Hopf ideal $\ideal\kodd^{\vect{k}} \subseteq \NSyml$ is given by 
\begin{equation*} 
  \ideal\kodd^{\vect{k}} 
  = \left< \, \Phi_{\vect{n}} \mid 
  \vect{0} < \vect{n} \le \vect{k} \text{ and } |\vect{n}| \text{ even} \, \right>.
\end{equation*}
On the dual side, the canonical $\vect{k}$-odd Hopf algebra $\kodd^{\vect{k}} \subseteq \QSyml$ is given by
\begin{equation} \label{E:OddAlgBasis}
  \kodd^{\vect{k}} 
  = \mathrm{span} \{\, P_{(\vect{i}_1, \ldots, \vect{i}_m)} \mid 
  \vect{i}_r \le \vect{k} \implies |\vect{i}_r| \text{ odd} \, \}.
\end{equation} 
\end{theorem}

\begin{proof}
Recall that $\ideal\kodd^{\vect{k}} = \ideal^{\vect{k}}(\bar\zetaQ,\zetaQ^{-1})$ is the ideal generated by $\bar{(\zetaQ)}_{\vect{n}} - (\zetaQ^{-1})_{\vect{n}}$ such that $\vect{n} \le \vect{k}$. Let $\vect{n} \in \NNl$ and assume $\vect{n} \le \vect{k}$. Since as an element of ${\QSyml}^* = \NSyml$, the $\vect{n}$th graded piece of $\zetaQ$ is $S_{\vect{n}}$, we get $\bar{(\zetaQ)}_{\vect{n}} = (-1)^{|\vect{n}|}S_{\vect{n}}$ and $(\zetaQ^{-1})_{\vect{n}} = \apode(S_{\vect{n}})$, so by \eqref{E:SPhi} and \eqref{E:S-antipode},
\begin{equation}\label{E:zeta-Phi}
\bar{(\zetaQ)}_{\vect{n}} - (\zetaQ^{-1})_{\vect{n}} = \sum_{\vect{I} \comp \vect{n}} \left[ \frac{(-1)^{|\vect{n}|}}{\spp(\vect{I})} - \frac{(-1)^{\len(\vect{I})}}{\spp(\vect{I})} \right] \, \Phi^{\vect{I}} = 
\sum_{ \begin{subarray}{c} \vect{I}\comp \vect{n} \\ \len(\vect{I}) \, \not\equiv \, |\vect{n}| \, (\mathrm{mod} 2) \end{subarray}} \frac{2 \, (-1)^{|\vect{n}|}}{\spp(\vect{I})} \; \Phi^{\vect{I}}. \end{equation}
The condition $\len(\vect{I}) \not\equiv |\vect{n}| \pmod{2}$ holds if and only if $\vect{I}$ has an odd number of columns of even weight. If $\vect{I} = (\vect{i}_1, \ldots, \vect{i}_m)$ is such a vector composition, then some column $\vect{i}_r$ has even weight, and moreover $\vect{i}_r \le \vect{k}$. Therefore $\Phi^{\vect{I}} = \Phi_{\vect{i}_1} \cdots \Phi_{\vect{i}_m}$ is in the ideal generated by those $\Phi_{\vect{i}}$ such that $\vect{i} \le \vect{k}$ and $|\vect{i}|$ is even. Consequently $\bar{(\zetaQ)}_{\vect{n}} - (\zetaQ^{-1})_{\vect{n}}$ is also in this ideal, and therefore $\ideal\kodd^{\vect{k}}$ is a subset of this ideal. 

To prove the reverse inclusion, first note that every $\lev$-partite number of weight $2$ is of the form $\vect{e}_i + \vect{e}_j$ for some $i, j \in [0, \lev-1]$. 
By \eqref{E:zeta-Phi} we have
\[
\Phi_{\vect{e}_i + \vect{e}_j} =  (\bar{(\zetaQ)})_{\vect{e}_i + \vect{e}_j} - (\zetaQ^{-1})_{\vect{e}_i + \vect{e}_j}.
\]
Now suppose that $\vect{n} \in \NNl$ has even weight and $|\vect{n}| > 2$. By  \eqref{E:zeta-Phi}, 
\begin{equation} \label{E:Phin-zeta} 
\frac{2}{|\vect{n}|} \, \Phi_{\vect{n}} =  \bar{(\zetaQ)}_{\vect{n}} - (\zetaQ^{-1})_{\vect{n}} - \sumsub{\vect{I} \comp \vect{n} \\ \vect{I} \ne \vect{n} \\ \len(\vect{I}) \, \not\equiv \, |\vect{n}| \, (\mathrm{mod} 2) }  \frac{2}{\spp(\vect{I})} \,\Phi^{\vect{I}}. 
\end{equation}
As we already observed, if a vector composition $\vect{I} = (\vect{i}_1,\ldots, \vect{i}_m)$ appears in the previous sum, then one of its columns, say $\vect{i}_r$, has even weight, and moreover $|\vect{i}_r| < |\vect{n}|$. By induction $\Phi_{\vect{i}_r}$ is in the ideal $\left< ({\zetabarQ)}_{\vect{i}} - (\zetaQ^{-1})_{\vect{i}} \mid \vect{i} \le \vect{i}_r \right>$, so $\Phi^{\vect{I}}$ is also in this ideal. In light of \eqref{E:Phin-zeta}, we conclude that $\Phi_{\vect{n}} \in \left< ({\zetabarQ)}_{\vect{i}} - (\zetaQ^{-1})_{\vect{i}} \mid \vect{i} \le \vect{n} \right>\subseteq \ideal\kodd^{\vect{k}}$ provided $\vect{n} \le \vect{k}$.
\end{proof}

\begin{example}
With $\lev = 3$ we have
\[ 
\ideal\kodd^{\colvec{4 \\ 0 \\ 3}} =  
\bigl< \Phi_{\colvec{4 \\ 0 \\ 2}}, \Phi_{\colvec{4 \\ 0 \\ 0}}, \Phi_{\colvec{3 \\ 0 \\ 3}}, \Phi_{\colvec{3 \\ 0 \\ 1}}, \Phi_{\colvec{2 \\ 0 \\ 2}}, \Phi_{\colvec{2 \\ 0 \\ 0}}, \Phi_{\colvec{1 \\ 0 \\ 1}} 
\bigr>
\]
and 
\[
\kodd^{\colvec{4 \\ 0 \\ 3}} =
\mathrm{span}
\left\{ P_{(\vect{i}_1, \ldots, \vect{i}_m)} \mid \vect{i}_r \notin \left\{{\colvec{4 \\ 0 \\ 2}}, {\colvec{4 \\ 0 \\ 0}}, {\colvec{3 \\ 0 \\ 3}}, {\colvec{3 \\ 0 \\ 1}}, {\colvec{2 \\ 0 \\ 2}}, {\colvec{2 \\ 0 \\ 0}}, {\colvec{1 \\ 0 \\ 1}} 
\right\} 
\right\}.
\]
\end{example}

\begin{example}
If $\lev = 1$ and $k \in \NN$ then
\[ 
\ideal\kodd^{k} = \left< \, \Phi_2, \Phi_4, \ldots, \Phi_{2 \lfloor k/2 \rfloor} \, \right> 
\]
and
\[ 
\kodd^{k} = \mathrm{span} \{ P_{(i_1, \ldots, i_m)} \mid i_r \notin \{2, 4, \ldots, 2 \lfloor k/2 \rfloor \}\text{ for every } r  \, \}.
\]
In the limit case, $\ideal\kodd^{\infty}$ is generated by the set of all even-degree power sums $\Phi_2, \Phi_4, \ldots$, and the odd subalgebra of $\QSym$ is spanned by the quasisymmetric functions $P_I$ where $I$ ranges over all (ordinary) compositions with only odd parts. 
\end{example}

\begin{remark}
Aguiar et al.\ \cite[Proposition~6.5]{ABS06} showed that the odd subalgebra of $\QSym$ is Stembridge's peak Hopf algebra \cite{Stembridge97}, which is the image of the descents-to-peaks homomorphism $\thetamap:\QSym \to \QSym$ of Stembridge \cite{Stembridge97}. On the dual side, it is implicit in the work of Krob et al.\ \cite{KLT97} (as explained in \cite{BHT04}; see also and \cite[Prop.~7.6]{Schocker05}) that the noncommutative power sums $\Phi_2, \Phi_4, \ldots \in \NSym$ of even degree generate the kernel of the dual map $\Theta^*:\NSym \to \NSym$. This kernel, being orthogonal to the peak Hopf algebra, coincides with the canonical ideal $\ideal\kodd^\infty$. Thus the special case of Theorem~\ref{T:koddQSym} in which $\lev = 1$ and $k=\infty$ follows from these results. However, our proof of Theorem~\ref{T:koddQSym} is self-contained and in particular does not use any properties of the peak algebra or the descents-to-peaks map. 
\end{remark}

We now discuss some corollaries of Theorem~\ref{T:koddQSym}. First is an immediate consequence.

\begin{corollary}
If all coordinates of $\vect{k}$ are even, then $\ideal\kodd^{\vect{k}} = \ideal\kodd^{\vect{k} + \vect{e}_i}$ for each $i\in [0, \lev-1]$.
\end{corollary}

\begin{example} 
Given any $\lev$ we have
\[
\kodd^{\vect{0}} = \kodd^{\vect{e}_i} = \QSyml 
\]
for every coordinate vector $\vect{e}_i\in \NNl$. If $\lev = 1$ then
\[ 
\kodd^{0} = \kodd^1 \supsetneq \kodd^{2} = \kodd^3 \supsetneq \kodd^{4} = \kodd^5 \supsetneq \cdots \supsetneq \kodd^{\infty} .
\]
\end{example}

For the next result, first observe that since $\Phi^{\vect{I}}$ and $P_{\vect{I}}$ are dual bases (\S\S\ref{SS:Duality}) and the $\Phi_{\vect{n}}$ are primitive (Proposition~\ref{P:PhiPrimitive}), multiplying two $P$-basis elements $P_{\vect{I}} \cdot P_{\vect{J}}$ is given by shuffling the columns of $\vect{I}$ with the columns of $\vect{J}$. For example, 
\[
P_{\left( \begin{smallmatrix} a & b \\ c & d \end{smallmatrix} \right)} \cdot P_{\left( \begin{smallmatrix} e & f \\ g & h \end{smallmatrix} \right)}
= 
P_{\left( \begin{smallmatrix} a & b & e & f \\ c & d & g & h \end{smallmatrix} \right)} + 
P_{\left( \begin{smallmatrix} a & e & b & f \\ c & g & d & h \end{smallmatrix} \right)} + 
P_{\left( \begin{smallmatrix} a & e & f & b \\ c & g & h & d \end{smallmatrix} \right)} + 
P_{\left( \begin{smallmatrix} e & a & b & g \\ f & c & d & h \end{smallmatrix} \right)} +
P_{\left( \begin{smallmatrix} e & a & f & b \\ g & c & h & d \end{smallmatrix} \right)} + 
P_{\left( \begin{smallmatrix} e & f & a & b \\ g & h & c & d \end{smallmatrix} \right)}.
\]
Thus from \eqref{E:OddAlgBasis} it is clear that {\it $\kodd^{\vect{k}}$ is a shuffle algebra} in the sense of \cite{Reutenauer93}. 

Now let us endow the set of nonzero $\lev$-partite numbers $\NNl \setminus \{\vect{0}\}$ with an arbitrary linear order. A vector composition (thought of as a word in the alphabet $\NNl \setminus \{\vect{0}\}$) that is lexicographically smaller than all of its nontrivial cyclic permutations is called a \newword{Lyndon vector composition}. For a subset $A\subseteq \NNl \setminus \{\vect{0}\}$, let $\lyndon(A)$ denote the set of Lyndon vector compositions in $A$.  For example, if we order the $2$-partite numbers lexicographically by $\colvec{ 0 \\ 1 } <_{lex} \colvec{ 0 \\ 2 } <_{lex} \cdots <_{lex} \colvec{ 1 \\ 0 } <_{lex} \colvec{ 1 \\ 1 } <_{lex} \cdots$, then the Lyndon vector compositions of weight $3$ are $\colvec{ 0 \\ 3 }$, $\colvec{ 1 \\ 2 }$, $\colvec{ 2 \\ 1 }$, $\colvec{ 3 \\ 0 }$, $\colvec{ 0 & 0 \\ 1 & 2 }$, $\colvec{ 0 & 1 \\ 1 & 1 }$, $\colvec{ 0 & 2 \\ 1 & 0 }$, $\colvec{ 0 & 1 \\ 2 & 0 }$, $\colvec{ 0 & 0 & 1 \\ 1 & 1 & 0 }$, and $\colvec{ 0 & 1 & 1 \\ 1 & 0 & 0 }$. 
It is well-known that a shuffle algebra is freely generated by its Lyndon words. We refer to \cite{Reutenauer93} for further details.

In this context the following is now obvious:
\begin{corollary} \label{C:kodd-shuffle}
As an algebra $\kodd^{\vect{k}}$ is the  polynomial algebra 
\[
\kodd^{\vect{k}} = \field[P_{\vect{I}} \mid \vect{I} \in \lyndon(\{\, \vect{n} \in \NNl \setminus \{\vect{0}\} \mid \vect{n} \le \vect{k} \implies |\vect{n}| \text{ odd} \, \})].
\]
\end{corollary}
\noindent
In the case $\lev=1$ and $k=\infty$, this 
follows from results of Schocker \cite[\S7]{Schocker05}. 

A further consequence of Theorem~\ref{T:koddQSym} is the following description of the multi-symmetric part $\kodd^{\vect{k}} \cap \Syml$ of $\kodd^{\vect{k}}$, extending results of Malvenuto and Reutenauer \cite[Corollary~2.2]{MR95} (the $\lev = 1$, $\vect{k} = 0$ case) and Schocker \cite[\S7]{Schocker05} (the $\lev = 1$, $\vect{k} = \infty$ case). Recall that we introduced the Hopf algebra $\Syml$ of MacMahon's multi-symmetric functions in \S\S\ref{SS:Syml}.

\begin{corollary}\label{C:koddSym}
The multi-symmetric part of $\kodd^{\vect{k}}$ is freely generated as an algebra by a subset of the $\lev$-partite power sum symmetric functions $p_{\vect{n}} = \sum_{j  = 1}^\infty \vect{x}_j^{\vect{n}}$. More explicitly we have 
\[
\kodd^{\vect{k}} \cap \Syml = \field[ \, p_{\vect{n}} \mid \vect{n} \le \vect{k} \implies |\vect{n}| \text{ odd} \, ].
\]
Consequently, $\kodd^{\vect{k}}$ is free as a module over $\kodd^{\vect{k}} \cap \Syml.$ 
\end{corollary}

\begin{proof}
Let $\svect{\lambda}$ be an $\lev$-partition in the sense of \S\S\ref{SS:lpartite}. 
For each nonzero $\vect{i} \in \NNl$, let $m_\vect{i}$ be the multiplicity of $\vect{i}$ in $\svect{\lambda}$. Using the shuffling interpretation of multiplication of the $P_{\vect{I}}$ that we mentioned just before Corollary~\ref{C:kodd-shuffle}, we deduce that
\begin{equation}\label{E:p-P}
 p_{\svect{\lambda}} = 
 \prod_{\vect{i} \in \NNl -\{\vect{0}\} } P_{\vect{i}}^{m_{\vect{i}}} = 
 z_{\svect{\lambda}}^{-1} \,
 \sum_{\mathrm{cols}(\vect{I}) = \lambda} P_{\vect{I}} ,
 \end{equation}
where $z_{\svect{\lambda}} = \prod_{\vect{i} \in \NNl - \{\vect{0} \}} \, m_{\vect{i}}! \; |\vect{i}|^{m_{\vect{i}}}$ and $\mathrm{cols}(\vect{I})$ denotes the multiset of columns of $\vect{I}$. It follows from Theorem~\ref{T:koddQSym} that if every element of $\svect{\lambda}$ belongs to the set 
\[
A = \{\, \vect{n} \in \NNl \setminus \{\vect{0}\} \mid \vect{n} \le \vect{k} \implies |\vect{n}| \text{ odd} \, \},
\] 
then $p_{\svect{\lambda}} \in \kodd^{\vect{k}} \cap \Syml$. This shows that $\kodd^{\vect{k}} \cap \Syml \supseteq \field[ \, p_{\vect{n}} \mid \vect{n} \in A \, ].$ Conversely, given $\sum_{\svect{\lambda}} a_{\svect{\lambda}} \, p_{\svect{\lambda}} \in  \kodd^{\vect{k}} \cap \Syml$ where each $a_{\svect{\lambda}}$ is a scalar, Theorem~\ref{T:koddQSym} together with \eqref{E:p-P} ensure that $a_{\svect{\lambda}}$ vanishes if the elements of $\svect{\lambda}$ do not all belong to $A$.

Finally, since $A \subseteq \lyndon(A)$ and $p_{\vect{i}} = \frac{1}{|\vect{i}|} \, P_{\vect{i}}$ for every nonzero $\vect{i} \in \NNl$, we have a free generating set $\{ p_{\vect{i}} \mid \vect{i} \in A \}$ for $\kodd^{\vect{k}} \cap \Syml$ contained in a free generating set $\{ \frac{1}{|\vect{I}|}P_{\vect{I}} \mid \vect{I} \in \lyndon(A)\}$ for $\QSyml$. It follows that $\kodd^{\vect{k}}$ is a free module over $\kodd^{\vect{k}} \cap \QSyml$.
\end{proof}


\begin{example}
With $\lev = 2$ we have
\[ 
\kodd^{\colvec{ 3 \\ 2}} \cap \Sym^{(2)} = \field[ \, p_{\vect{n}} \mid \vect{n} \in \NNl \setminus \{ \colvec{0 \\ 0}, \colvec{ 1 \\ 1}, \colvec{2 \\ 0}, \colvec{0 \\2}, \colvec{2 \\ 2}, \colvec{3 \\ 1} \} \,].
\]
\end{example}
\begin{example}
If $\lev = 1$ and $k \in \NN$ then
\[
\kodd^{2k+1} \cap \Sym  = \field [ \, p_i \mid i \in \NN \setminus \{0, 2, 4, \ldots, 2k\} \,].
\]
\end{example}

\subsection{The Euler character and generators for $\ideal\kodd^{\vect{k}}$}

In this subsection we give another description of $\ideal\kodd^{\vect{k}}$ in terms of the \newword{Euler character} $\chi$ on $\QSyml$, defined as in \cite{ABS06} by
\[ 
\chi = \zetabarQ \zetaQ. 
\]
For any $\lev$-partite number $\vect{n} \in \NNl$, we have
\begin{equation} \label{E:chin}
\chi_{\vect{n}} = \sum_{\vect{j} \le \vect{n}} (-1)^{|\vect{j}|} S^{(\vect{j},\, \vect{n} - \vect{j})}.
\end{equation}
When $\lev = 1$ these are essentially the $\chi_n$ of \cite{BL00} and \cite{Ehrenborg01}. Billera and Liu proved that only the even Euler forms $\chi_2, \chi_4, \ldots$ are needed to generate the ideal $\left< \chi_1, \chi_2, \chi_3, \ldots \right>$. This is generalized in the following result.

\begin{theorem} \label{T:kEulerIdeal}
We have
\[
\ideal\kodd^{\vect{k}} = \left< \chi_{\vect{n}} \mid \vect{0} < \vect{n} \le \vect{k} \right>
= \left< \chi_{\vect{n}} \mid \vect{0} < \vect{n} \le \vect{k}, \, |\vect{n}| \text{ even}\right>.
\]
\end{theorem}
\begin{proof}
From the definition of $\chi$, we get $\zetabarQ - \zetaQ^{-1} = (\chi - \epsilon)\zetaQ^{-1}$. Equating homogeneous components gives:
\[ (\zetabarQ)_{\vect{n}} - (\zetaQ^{-1})_{\vect{n}} = (\zetabarQ - \zetaQ^{-1})_{\vect{n}} = \sum_{\vect{i}\le \vect{n}} (\chi - \epsilon)_{\vect{i}}(\zetaQ^{-1})_{\vect{n}-\vect{i}} = \sum_{\vect{i} \le \vect{n}} (\chi)_{\vect{i}}(\zetaQ^{-1})_{\vect{n} - \vect{i}} \]
for each $\vect{n} \ne \vect{0}$. (Note that here we use the fact that the restriction of any character, and so in particular $\chi$, to the zeroth graded piece is equal to $\epsilon$).  Thus $(\zetabarQ)_{\vect{i}} - (\zetaQ^{-1})_{\vect{i}} \in \left < \chi_{\vect{n}} \mid \vect{0} < \vect{n} \le \vect{k} \right >$ for each $\vect{i} \le \vect{k}$, and hence $\ideal\kodd^{\vect{k}} \subseteq \left < \chi_{\vect{n}} \mid \vect{0} < \vect{n} \le \vect{k} \right >$.
The opposite inclusion, hence equality, is a consequence of
\[ 
\chi_{\vect{n}} = (\chi - \epsilon)_{\vect{n}} = \sum_{\vect{i} \le \vect{n}} (\zetabarQ - \zetaQ^{-1})_{\vect{i}} (\zetaQ)_{\vect{n}-\vect{i}} = \sum_{\vect{i} \le \vect{n}} ((\zetabarQ)_{\vect{i}} - (\zetaQ^{-1})_{\vect{i}}) (\zetaQ)_{\vect{n}-\vect{i}}
\]
for $\vect{n} \ne \vect{0}$, which is derived from the identity $\chi - \epsilon = (\zetabarQ - \zetaQ^{-1}) \zetaQ$. 

Now we give a dimension argument to see that only the $\chi_{\vect{n}}$ where $|\vect{n}|$ is even are needed to generate the ideal $\ideal\kodd^{\vect{k}}$. Given vector compositions $\vect{I}_1, \ldots, \vect{I}_m$ and $\lev$-partite numbers $\vect{n}_1, \ldots, \vect{n}_{m-1}$ with $\vect{n}_i \le \vect{k}$ and $|\vect{n}_i|$ even for all $i$, let $\vect{J} = (\vect{I}_1, \vect{n}_1, \vect{I}_2, \vect{n}_2, \ldots, \vect{n}_{m-1}, \vect{I}_m)$ and define $T^{\vect{J}} = S^{\vect{I}_1} \chi_{\vect{n}_1} S^{\vect{I}_2} \chi_{\vect{n}_1} \cdots \chi_{\vect{n}_{m-1}}  S^{\vect{I}_m}$, which is an element of $\left< \chi_{\vect{n}} \mid \vect{0} < \vect{n} \le \vect{k}, |\vect{n}| \text{ even} \right>$. By \eqref{E:chin}, $T^{\vect{J}}$ has leading term $2 S^{\vect{J} }$ followed by a linear combination of $S^{\vect{I}}$ where $\vect{I}$ is strictly larger than $\vect{J}$ relative to $\tleq$. This shows that the $T^{\vect{J}}$ are upper triangular in the $S^{\vect{I}}$ relative to $\tleq$, and hence they are linearly independent. The possible indices $\vect{J}$ (after removing columns that are zero or empty) are exactly the vector compositions in which every column $\vect{i}$ such that $\vect{i} \le \vect{k}$ has even weight. From Theorem~\ref{T:koddQSym} we know that the number of such compositions is equal to $\dim_{\field} \ideal\kodd^{\vect{k}}$.
\end{proof}

The next result simply describes the elements $S^{\vect{I}} \chi_{\vect{n}} S^{\vect{J}} \in \ideal\kodd^{\vect{k}}$ more explicitly. 
\begin{corollary} \label{C:EulerCharacter}
Let $(\vect{i}_1, \ldots, \vect{i}_m)$ be a vector composition, and let $r\in \{1,\ldots, m\}$ be such that $\vect{i}_r \le \vect{k}$. We have
\begin{equation} \label{E:GeneralizedDH}
\sum_{ \vect{j} \le \vect{i}_r} (-1)^{|\vect{j}|} S^{(\vect{i}_1, \ldots, \vect{i}_{r-1}, \, \vect{j}, \, \vect{i}_r - \vect{j}, \, \vect{i}_{r+1}, \ldots, \vect{i}_m)} \in \ideal\kodd^{\vect{k}},
\end{equation}
where columns of zero weight are omitted.
\end{corollary}

\begin{remark}\label{Re:GeneralizedDH}
{\bf (Example~\ref{Ex:GeneralizedDH} continued) }
Recall that we studied the Hopf algebra $\GPoslk$ of $\vect{k}$-Eulerian posets in \S\S\ref{SS:GPoslk} and in Example~\ref{Ex:EulerianPoset}. Now we derive the linear equations announced in Example~\ref{Ex:GeneralizedDH}.
We deduce, from Equation \eqref{E:GeneralizedDH}, 
that if $(\vect{i}_1, \ldots, \vect{i}_m) \comp \vect{n}$ and $\vect{i}_r$ is a column such that $\vect{i}_r \le \vect{k}$, then for every $\vect{k}$-Eulerian poset $P$ of multirank $\vect{n}$ we have
\begin{equation} \label{E:GeneralizedDH2}
\sum_{\vect{j} \le \vect{i}_r} (-1)^{|\vect{j}|} f_{(\vect{i}_1, \ldots, \vect{i}_{r-1}, \, \vect{j}, \vect{i}_r - \vect{j}, \vect{i}_{r+1}, \ldots, \vect{i}_m)}(P) = 0
\end{equation}
where columns of zero weight are omitted. When $\lev = 1$ and $k=\infty$, the relations of the form \eqref{E:GeneralizedDH2} are precisely the generalized Dehn-Sommerville relations discovered by Bayer and Billera \cite{BB85}; see also \cite[Example~5.10]{ABS06}. Furthermore Ehrenborg proved that for all $k\in \NN$, the linear combinations of these relations give all possible linear relations satisfied by $k$-Eulerian posets \cite{Ehrenborg01}.
\end{remark}

With the help of Theorem~\ref{T:kEulerIdeal} we can provide the following interpretations of the algebras $\kodd^{k}$ when $\lev = 1$.

\begin{prop} \label{P:kodd-kposet}
When $\lev = 1$, we have
\[
\kodd^{k} = \mathrm{span}\{\mathcal{F}(P) \mid P \text{ is a $k$-Eulerian poset}\}
\]
for every $k\in \NN$.
\end{prop}
\begin{proof}
Ehrenborg proved that the ideal $\left< \chi_1, \ldots, \chi_k \right>$ is orthogonal to the vector space spanned by the $\mathcal{F}(P)$ where $P$ is $k$-Eulerian poset \cite[Theorem~4.2]{Ehrenborg01}. Since this ideal is orthogonal to $\kodd^{k}$, the result follows.
\end{proof}

\begin{remark}
For arbitrary $\lev$ and $\vect{k} \in \NNinfl$, we already observed in Example~\ref{Ex:GeneralizedDH} that $\mathcal{F}(P) \in \kodd^{\vect{k}}$ for every $\vect{k}$-Eulerian poset $P$. It should be possible to generalize Ehrenborg's constructions  in the $\lev =1$ case (i.e., \cite[Lemma~4.3]{Ehrenborg01}) to prove that the span of the $\mathcal{F}(P)$, where $P$ is $\vect{k}$-Eulerian, is all of $\kodd^{\vect{k}}$, but we will not pursue this point here.
\end{remark}

Lastly, observe that when $\lev = 1$, we have
\[ 
\ideal\kodd^{2} = \left< \chi_2 \right> = \left< 2 S^{(2)} - S^{(1,1)} \right>.
\]
It is known from the work of Bergeron et al.\ \cite[Theorem~5.7]{BMSW00} that this ideal is orthogonal to the Hopf subalgebra of $\QSym$ spanned by Billey and Haiman's shifted quasisymmetric functions \cite{BMSW02, BH95}. Thus we have the following:

\begin{prop} \label{P:shifted}
When $\lev=1$, the $2$-odd subalgebra $\kodd^{2}$ of $\QSyml = \QSym$ is the Hopf algebra spanned by the shifted quasisymmetric functions.
\end{prop}

\begin{remark}
Noting that the peak Hopf algebra is dual to $\NSym/\left< \chi_2, \chi_4, \ldots \right>$ and that the Hopf algebra of shifted quasisymmetric functions is dual to $\NSym/\left< \chi_2 \right>$, Bergeron et al.\ long ago suggested \cite{BMSW-PC} that the ``intermediate" Hopf algebras $\NSym/\left< \chi_2, \chi_4, \ldots, \chi_{2k}\right>$, where $k=1,2,\ldots $, might be of interest. As we have now shown, these Hopf algebras are precisely the graded duals of our canonically defined $(2k)$-odd subalgebras of $\QSym$. 
\end{remark}

We will revisit shifted quasisymmetric functions briefly in \S\ref{S:ktheta}.

\subsection{The $\etabasis$-basis for $\kodd^{\vect{k}}$}
\label{SS:etabasis}

We now describe a basis for $\kodd^{\vect{k}}$ which directly generalizes the $\eta$-basis studied by Aguiar et al.\ \cite[\S6]{ABS06} and is closely related to Ehrenborg's encoding of the flag $f$-vector of a $k$-Eulerian poset via certain noncommutative polynomials.

Let $\vect{I}$ be a vector composition. Define
\begin{equation*}
\etabasis_{\vect{I}} = \sum_{\vect{J}\, \tleq\, \vect{I}} 2^{\len(\vect{J})} M_{\vect{J}}.
\end{equation*}
By M\"obius inversion we get
\begin{equation*} 
M_{\vect{I}} = \frac{1}{2^{\len(\vect{I})}}  \sum_{\vect{J}\, \tleq \, \vect{I}} (-1)^{\len(\vect{I}) - \len(\vect{J})} \; \etabasis_{\vect{J}}.
\end{equation*}

Let $(\Upsilon^{\vect{I}})$ denote the basis for $\NSyml$ that is dual to $(\etabasis_{\vect{I}})$, so $\Upsilon^{\vect{I}}(\etabasis_{\vect{J}}) = \delta_{\vect{I}, \vect{J}}$. By duality, the two equations above 
become
\begin{equation} \label{E:S-Upsilon}
S^{\vect{I}} = 2^{\len(\vect{I})} \sum_{\vect{I} \, \tleq \, \vect{J}} \Upsilon_{\vect{J}}
\end{equation}
and
\begin{equation} \label{E:Upsilon-S}
\Upsilon^{\vect{I}} = \sum_{\vect{I}\, \tleq \, \vect{J}} (-1)^{\len(\vect{J}) - \len(\vect{I})}\frac{S^{\vect{J}}}{2^{\len(\vect{J})}}.
\end{equation}
Formula \eqref{E:Upsilon-S} makes it clear that the basis $(\Upsilon^{\vect{I}})$ is multiplicative, in the sense that 
\[\Upsilon^{(\vect{i}_1, \ldots, \vect{i}_m)} = \Upsilon_{\vect{i}_1} \cdots \Upsilon_{\vect{i}_m}\]
where $\Upsilon_{\vect{i}} = \Upsilon^{\vect{i}}$.

Next we determine the relationship between the bases $(\Upsilon^{\vect{I}})$ and $(\Phi^{\vect{I}})$. Let $E_n$ denote the $n$th Euler number, as defined by $\tan(t) + \sec(t) = \sum_{n \ge 0} \frac{E_n}{n!} \, t^{n}$ (see, e.g., \cite[p.\ 149]{EC1}).

\begin{prop}
Let $\vect{n} \in \NNl$ and $\vect{I}$ be a vector composition of $\vect{n}$. Then
\begin{equation}\label{E:eta-P}
  \etabasis_{\vect{I}} = 
  \sumsub{\vect{I} = \vect{I}_1\cdots \vect{I}_m \\ \forall r, \, \len(\vect{I}_r) \text{ is odd}}
  2^m \, \frac{|\vect{I}_1|}{\len(\vect{I}_1)} \cdots \frac{|\vect{I}_m|}{\len (\vect{I}_m)} \, P_{(\Sigma\vect{I}_1,\ldots, \Sigma\vect{I}_m)}
\end{equation}
and
\begin{equation}\label{E:P-eta}
  P_{\vect{I}} = 
  \sumsub{\vect{I} = \vect{I}_1\cdots \vect{I}_m \\ \forall i, \, \len(\vect{I}_i) \text{ is odd}}
  \frac{(-1)^{(\len(\vect{I})-m)/2}}{2^m} \, \frac{E_{\len(\vect{I}_1)}}{\len(\vect{I}_1)!}
  \cdots 
  \frac{E_{\len(\vect{I}_m)}}{\len(\vect{I}_m)!} \, \etabasis_{(\Sigma\vect{I}_1,\ldots, \Sigma\vect{I}_m)}.
\end{equation}
\end{prop}

\begin{proof}
It follows from \eqref{E:S-Upsilon} that
\[ 
1 + \sum_{\vect{n} > \vect{0} } S_{\vect{n}} \, t^{\vect{n}} 
  = 1 + 2  \sum_{m\ge 1} 
    \left(\sum_{\vect{n} > \vect{0} } \Upsilon_{\vect{n}} \, t^{\vect{n}} \right)^m 
  = \frac{1 + \sum_{\vect{n} > \vect{0} } \Upsilon_{\vect{n}} \, t^{\vect{n}}}
    {1 - \sum_{n > \vect{0} } \Upsilon_{\vect{n}} \, t^{\vect{n}} }
\]
where $t$ is a new commutative variable. By \eqref{E:Phidef},
\begin{equation} \label{E:Phi-Upsilon-series}
  \sum_{\vect{n} > \vect{0}} \frac{\Phi_{\vect{n}}}{|\vect{n}|} \, t^{\vect{n}} 
    = \log \left(\frac{1 + \sum_{\vect{n} > \vect{0} } \Upsilon_{\vect{n}} \, t^{\vect{n}}}
    {1 - \sum_{n > \vect{0} } \Upsilon_{\vect{n}} \, t^{\vect{n}} }\right)
    = 2 \, \tanh^{-1}\left( 
    \sum_{n > \vect{0} } \Upsilon_{\vect{n}} \, t^{\vect{n}}
    \right)
\end{equation}
where  $\tanh^{-1}(T) = \sum_{m \ge 0} \frac{1}{2m+1} T^{2m+1}$. 
Thus, for any vector composition $\vect{I} = (\vect{i}_1, \ldots, \vect{i}_m)$, 
\begin{equation*} 
\frac{\Phi^{\vect{I}}}{\pi(\vect{I})} = \sumsub{\vect{J} = \vect{J}_1 \cdots \vect{J}_m   \\
\forall r, \, \vect{J}_r \comp {\vect{i}}_r \text{ and } \len(\vect{J}_r) \text{ is odd}}
\frac{2^m}{\len(\vect{J}_1) \cdots \len(\vect{J}_m)} \cdot \Upsilon^{\vect{J}}
\end{equation*}
which is equivalent to \eqref{E:eta-P} by duality.

The formal inverse of the series $\tanh^{-1}(T)$ is $\tanh(T) = \sum_{m \ge 0} (-1)^m \frac{E_{2m+1}}{(2m+1)!} T^{2m+1}$. By \eqref{E:Phi-Upsilon-series},
\begin{equation*}
\Upsilon_{\vect{n}} = \sum_{\vect{I} \comp \vect{n}: \, \len(\vect{J}) \text{ is odd}} 
\frac{(-1)^{(\len(\vect{J})-1)/2}}{2^{\len(\vect{J})}} \, \frac{E_{\len(\vect{J})}}{\spp(\vect{J})} \, \Phi^{\vect{J}}.
\end{equation*}
Thus, for $\vect{I} = (\vect{i}_1, \ldots, \vect{i}_m)$, 
\begin{equation} \label{E:Upsilon-Phi}
\Upsilon^{\vect{I}} = \sumsub{\vect{J} = \vect{J}_1 \cdots \vect{J}_m   \\
\forall r, \, \vect{J}_r \comp {\vect{i}}_r \text{ and } \len(\vect{J}_r) \text{ is odd}}
\frac{(-1)^{(\len(\vect{J})-m)/2}}{2^{\len(\vect{J})}} \frac{E_{\len(\vect{J}_1)}}{\spp(\vect{J}_1)} \cdots 
\frac{E_{\len(\vect{J}_m)}}{\spp(\vect{J}_m)} \, \Phi^{\vect{J}}.
\end{equation}
which is equivalent to \eqref{E:P-eta} by duality.
\end{proof}

The canonical $\vect{k}$-odd ideal and $\vect{k}$-odd Hopf algebra can be expressed in terms of the $\etabasis$-basis in a manner similar to Theorem~\ref{T:koddQSym}.

\begin{theorem}\label{T:koddQSym-eta}
The canonical $\vect{k}$-odd Hopf ideal $\ideal\kodd^{\vect{k}} \subseteq \NSyml$ is given by 
\begin{equation*} 
  \ideal\kodd^{\vect{k}} 
  = \left< \, \Upsilon_{\vect{n}} \mid 
  \vect{0} < \vect{n} \le \vect{k} \text{ and } |\vect{n}| \text{ even} \, \right>.
\end{equation*}
On the dual side, the canonical $\vect{k}$-odd Hopf algebra $\kodd^{\vect{k}} \subseteq \QSyml$ is given by
\begin{equation*} 
  \kodd^{\vect{k}} 
  = \mathrm{span} \{\, \eta_{(\vect{i}_1, \ldots, \vect{i}_m)} \mid 
  \vect{i}_r \le \vect{k} \implies |\vect{i}_r| \text{ odd} \, \}.
\end{equation*} 
\end{theorem}
\begin{proof}
Suppose that $\vect{n}$ has even weight and $\vect{n} \le \vect{k}.$ If  $\vect{I} \comp \vect{n}$ and $\len(\vect{I})$ is odd, then $\vect{I}$ must have a column of even weight, hence $\Phi^{\vect{I}} \in \ideal\kodd^{\vect{k}}$ by Theorem~\ref{T:koddQSym}. This implies $\Upsilon_{\vect{n}} \in \ideal\kodd^{\vect{k}}$ in view of \eqref{E:Upsilon-Phi}. Thus $\left< \, \Upsilon_{\vect{n}} \mid 
  \vect{0} < \vect{n} \le \vect{k} \text{ and } |\vect{n}| \text{ even} \, \right> \subseteq \left< \, \Phi_{\vect{n}} \mid 
  \vect{0} < \vect{n} \le \vect{k} \text{ and } |\vect{n}| \text{ even} \, \right>.$ But these two ideals clearly have the same dimension as $\NNl$-graded vector spaces, hence they are equal. The second equation
follows immediately by duality.
\end{proof}

\noindent
Note that specializing Theorem \ref{T:koddQSym-eta} to $\lev=1$ and $\vect{k} = \svect{\infty}$ gives \cite[Proposition~6.5]{ABS06}.

\begin{remark}
We now explain the connection between the $\etabasis$-basis and one of Ehrenborg's constructions \cite{Ehrenborg01}. Recall that the $\aaa\bb$-index of a graded poset $P$ is a certain polynomial $\Psi(P)$ in two noncommutative variables $\aaa$ and $\bb$ that encodes the flag $f$-vector of $P$. There is a well-defined linear map $\tau: \abring \to \QSym$ that takes the $\aaa\bb$-index of a poset to the $F$-quasisymmetric function of that same poset; i.e., $\tau(\Psi(P)) = F(P)$. Ehrenborg showed \cite[Theorem~2.1]{Ehrenborg01} that the linear span of the polynomials $\Psi(P)$ as $P$ ranges over the $(2k+1)$-Eulerian posets is the subring $\field\langle{\cc,\ee^2, \ee^{2k+1}\rangle}$, where $\cc = \aaa + \bb$ and $\ee = \aaa - \bb$. To describe a linear basis for this subring, we introduce the notation $\mon{\vect{x}}{\vect{y}}^I$, where $\vect{x}, \vect{y} \in \abring$ and $I = (i_1, \ldots, i_m)$ is a composition, defined by
\[ 
\mon{{\bf x}}{{\bf y}}^I = \mathbf{x}^{i_1-1} \cdot \mathbf{y} \cdot \mathbf{x}^{i_2-1} \cdot \mathbf{y}
\cdots \mathbf{y} \cdot \mathbf{x}^{i_m-1}.
\]
For instance $\mon{{\bf x}}{{\bf y}}^{(3,4,1,2)} = {\bf x}^2\, {\bf y}\, {\bf x}^3\, {\bf y}^2 \,{\bf x}$, and $\mon{{\bf x}}{{\bf y}}^{(n)} = {\bf x}^{n-1}$. The set $\{\ecmon^{I}\}$ is a basis for $\field\langle{\cc,\ee^2, \ee^{2k+1}\rangle}$ if we let $I$ range over all compositions whose parts are in the set $\NN - \{ 0, 2, 4, \ldots, 2k\}$.  One can show that $\tau(\mon{\ee}{\bb}^I) = M_I$, and hence
\[
\tau(\mon{\ee}{\cc}^I) = \frac{1}{2} \eta_I.
\]
Using this equation and Theorem~\ref{T:koddQSym-eta}, we can recover our earlier observation, Proposition~\ref{P:kodd-kposet}, that $\kodd^{(k)}$ is the linear span of $F(P)$ as $P$ ranges over all $k$-Eulerian posets. 
\end{remark}

\subsection{Hilbert series of $\kodd^{\vect{k}}$}

We next determine the Hilbert series of the $\NNl$-grading of $\kodd^{\vect{k}}$.

\begin{prop}
\label{P:Hilb}
Let $\vect{k} = (k_0, \ldots, k_{\lev-1})^T \in \NNinfl$. We have 
\begin{equation*} 
\sum_{ \vect{n} \in \NNl} \dim( (\kodd^{\vect{k}})_{\vect{n}} ) \; t^{\vect{n}} = 
\frac{\displaystyle  \prod_{i=0}^{\lev-1} (1-t_i^2)}
{\displaystyle  \prod_{i=0}^{\lev-1}(1-t_i^2) -\prod_{i=0}^{\lev-1} (1+t_i)
+ \sumsub{\vect{b} \in \{0,1\}^\lev \\ |\vect{b}| \text{ even}} \prod_{i=0}^{\lev-1} t_i^{b_i} (1 - t_i^{2 \left\lfloor\frac{k_i - b_i}{2} \right\rfloor + 2})}
\end{equation*}
where if $\vect{n} = (n_0, \ldots, n_{\lev-1})^T \in \NNl$ then $t^{\vect{n}} = \prod_{i=0}^{\lev-1} t_i^{n_i}$. Here the term $t_i^{2 \left\lfloor\frac{k_i - b_i}{2} \right\rfloor + 2}$ is taken to be $0$ when $k_i = \infty$.
\end{prop}
\begin{proof}
Let $f$ be the multivariate generating function for $\lev$-partite numbers $\vect{n}$ of \emph{even} weight such that $\vect{n}\le \vect{k}$; by \eqref{E:OddAlgBasis} these are precisely the $\lev$-partite numbers that are forbidden from appearing as columns in the vector compositions $\vect{I}$ used to index the basis elements $P_{\vect{I}}$ of $\kodd^{\vect{k}}$. Thus, letting $g=\prod_{i=0}^{\lev-1} \frac{1}{1-t_i}$ be the generating function for all $\lev$-partite numbers, we get
\[ 
\sum_{ \vect{n} \in \NNl} \dim( (\kodd^{\vect{k}})_{\vect{n}} )\;  t^{\vect{n}} = \frac{1}{1-(g-f)}. 
\]
To determine $f$, for each subset $S$ of $[0,\lev-1]$ of even size we enumerate $\lev$-partite numbers $\vect{n} = (n_1, \ldots, n_\lev)^T \le \vect{k}$ such that $n_i$ is odd if and only if $i\in S$. This leads to the formula
\[
f  
= \sumsub{ S\subseteq [\lev] \\ |S| \text{ even}} 
\left( \prod_{i\in S} \sum_{j=0}^{\lfloor (k_i-1)/2\rfloor} t_i^{2j+1} \right)
\left( \prod_{i\notin S} \sum_{j=0}^{\lfloor k_i/2\rfloor} t_i^{2j} \right) 
=
\frac{1}{\prod_{i=0}^{\lev-1} (1-t_i^2)} \sumsub{\vect{b} \in \{0,1\}^\lev \\ |\vect{b}| \text{ even}} 
\prod_{i=0}^{\lev-1} t_i^{b_i} (1 - t_i^{2 \lfloor\frac{k_i-b_i}{2}\rfloor + 2}).
\]
Substituting this into $1/(1-(g-f))$ gives the desired result.
\end{proof}

\begin{example}
Let us consider a few specializations of Proposition \ref{P:Hilb}:
\begin{enumerate}
\item 
The Hilbert series of $\kodd^k$ when $\lev=1$ and $k\in\NNinf$ is
\[ 
\sum_{n=0}^\infty \dim((\kodd^{k})_n) \; t^n = 
\frac{1 - t^2}{1 - t - t^2 - t^{2 \lfloor k/2 \rfloor +2}} 
\]
which reduces to the generating function for the Fibonacci numbers when $k=\infty$.
\item
When $\lev=2$ and $\vect{k} = \colvec{ p \\ q} \in \NN^2$, we have
\begin{multline*}
\sum_{\colvec{m \\ n} \in \NN^2} \dim((\kodd^{\colvec{ p \\ q}})_{\colvec{m \\ n}}) x^m y^n \\
= \frac{(1-x^2)(1-y^2)}
{\left[ \begin{aligned}
(1-x^2)(1-y^2)-(1+x)(1+y) + (1- & x^{2 \left\lfloor \frac{p}{2} \right\rfloor + 2}) (1-y^{2 \left \lfloor \frac{q}{2} \right\rfloor + 2}) \\
& - xy (1-x^{2 \left\lfloor \frac{p -1}{ 2} \right\rfloor + 2})(1-y^{2 \left\lfloor \frac{q -1}{ 2} \right\rfloor + 2})
\end{aligned} 
\right]
} 
\end{multline*}

\item For any $\lev\ge 1$ and $S\subseteq [0,\lev-1]$, let $\vect{k}_S = (k_1, \ldots, k_{\lev})^T \in \{0,\infty\}^\lev$ be defined by $k_i = \infty$ if $i\in S$ and $k_i = 0$ otherwise. Then 
\[
  \sum_{\vect{n} \in \NNl} \dim((\kodd^{\vect{k}_S})_{\vect{n}}) \, t^{\vect{n}} = 
  \frac{\displaystyle \prod_{i=0}^{\lev-1}(1-t_i^2)}
  {\displaystyle \prod_{i=0}^{\lev-1}(1-t_i^2) - \prod_{i=0}^{\lev-1} (1+t_i) 
  + \frac{1}{2} \left[\prod_{i\notin S} (1-t_i) + \prod_{i\notin S} (1+t_i) \right]}.
\]
Setting every $t_i$ equal to the same variable $t$ yields the Hilbert series of $\kodd^{\vect{k}_S}$ graded by weight. This series, which depends only on $\lev$ and $r = |S|$, is denoted by
\[
f_{\lev, r}(t) =  \sum_{\vect{n} \in \NNl} \dim( (\kodd^{(\infty^r,0^{\lev-r})})_{\vect{n}}) \; t^{|\vect{n}|}.
\]
Here are examples of $f_{\lev, r}(t)$ for small values of $\lev$ and $r$:
\[
\begin{aligned}
f_{1,1}(t) & =  1+t+t^2+2 t^3+3 t^4+5 t^5+8 t^6+13 t^7+21 t^8+34 t^9+\cdots \\
f_{2,1}(t) & = 1+2 t+6 t^2+20 t^3+64 t^4+206 t^5+662 t^6+2128 t^7+
\cdots \\
f_{2,2}(t) & =  1+2 t+4 t^2+12 t^3+32 t^4+86 t^5+232 t^6 +624 t^7+
\cdots \\
f_{3,3}(t) & = 1+3 t+9 t^2+37 t^3+141 t^4+534 t^5+2035 t^6+7740 t^7+
\cdots 
\end{aligned}
\]
Note that $f_{\lev, \lev}(t)$ is the weight enumerator of vector compositions whose columns have odd weight.
\end{enumerate}
\end{example}

\subsection{The $k$-even subalgebra of $\QSyml$}

We briefly investigate the $\vect{k}$-even Hopf algebra $\keven^{\vect{k}}$ and its orthogonal ideal $\ideal \keven_{\vect{k}}$.

\begin{theorem}\label{T:kevenQSym}
The canonical $\vect{k}$-even Hopf ideal $\ideal \keven_{\vect{k}} \subseteq \NSyml$ is given by 
\begin{equation*} 
  \ideal\keven^{\vect{k}} 
  = \left< \, S_{\vect{n}} \mid 
  \vect{0} < \vect{n} \le \vect{k} \text{ and } |\vect{n}| \text{ odd} \, \right>.
\end{equation*}
On the dual side, the canonical $\vect{k}$-even Hopf algebra $\keven^{\vect{k}} \subseteq \QSyml$ is given by
\begin{equation*} 
  \keven^{\vect{k}} 
  = \mathrm{span} \{\, M_{(\vect{i}_1, \ldots, \vect{i}_m)} \mid 
  \vect{i}_r \le \vect{k} \implies |\vect{i}_r| \text{ even} \, \}.
\end{equation*} 
\end{theorem}

\begin{proof}
The ideal $\ideal\keven^{\vect{k}} = \ideal^{\vect{k}}(\bar\zetaQ,\zetaQ)$ is generated by elements of the form
\[
\bar{\zetaQ}_{\vect{n}} - (\zetaQ)_{\vect{n}} = (-1)^{|\vect{n}|} S_{\vect{n}} - S_{\vect{n}} = 
\begin{cases}
-2 S_{\vect{n}} & |\vect{n}| \text{ odd} \\
0 & \text{ otherwise}
\end{cases} 
\]
where $\vect{n} \le \vect{k}$. Hence the first result
follows. Equivalently we can write this as $\ideal\keven^{\vect{k}} = \mathrm{span} \{ S^{\vect{i}_1, \ldots, \vect{i}_m} \mid \vect{i}_r \le \vect{k} \implies |\vect{i}_r| \text{ odd}\}.$ Since the $M_{\vect{I}}$ are dual to the $S^{\vect{I}}$, the second result now follows.
\end{proof}

\noindent
Note that specializing Theorem \ref{T:kevenQSym} to $\lev = 1$ and $\vect{k} = \svect{\infty}$ gives \cite[Proposition~6.3]{ABS06}.

\section{A $\vect{k}$-analogue of the descents-to-peaks map} 
\label{S:ktheta}

The {\em canonical odd character} $\nuQ = \zetabarQ^{-1} \zetaQ$ was studied in \cite{ABS06}, where it was shown that the unique morphism of Hopf algebras $\thetamap: \QSym \rightarrow \QSym$ such that $\zetaQ \circ \thetamap = \nuQ$ is the descents-to-peaks map investigated by Stembridge \cite{Stembridge97}. Our main goal in this section is to refine the notion of the {\em canonical odd character} to a $\vect{k}$-odd linear functional and to give explicit formulas to compute this character. We also give formulas for the corresponding induced map $\thetamapk:\QSyml \to \QSyml$ in the case where $\lev$ is arbitrary and $\vect{k} = \svect{\infty}$, and in the case $\lev = 1$ and $\vect{k}$ is arbitrary. In the former case, the images of the $F_{\vect{I}}$ under $\thetamap^{\svect{\infty}}$ give rise to a basis of level $\lev$ peak functions.

\subsection{The $\vect{k}$-odd functional $\nuQk$}
For $\vect{k} \in \NNinfl$, let $\zetaQk$ be the linear functional on $\QSyml$ defined by 
\begin{equation} \label{E:zetaQk}
(\zetaQk)_{\vect{n}} = 
\begin{cases} (\zetabarQ)_{\vect{n}} & \vect{n} \le \vect{k} \\
(\zetaQ)_{\vect{n}} & \text{otherwise.}
\end{cases}
\end{equation}
To obtain a $\vect{k}$-analog of $\nuQ$, we replace $\zetabarQ^{-1}$ with $\zetaQk \circ \apodeQ$ and define the linear functional $\nuQk$ on $\QSyml$ by
\begin{equation*}
\nuQk =  (\zetaQk \circ \apodeQ) \, \zetaQ.
\end{equation*}
According to Theorem~\ref{T:kOddInducedMap}, there exists a unique morphism of multigraded coalgebras 
\[
\thetamapk: \QSyml \to \QSyml
\]
such that $\zetaQ \circ \thetamapk = \nuQk$. Moreover, $\thetamapk(\QSyml) \subseteq \kodd^{\vect{k}}$. This map is the natural $\vect{k}$-analog of the descents-to-peaks map.

Note that $\nuQk$ is not necessarily a character. For example when $\lev = 1$ we have 
\[
\nu^{(4)}(M_{(2)}M_{(3)}) = 2 \ne 0 = \nu^{(4)}(M_{(2)}) \nu^{(4)}(M_{(3)}).
\] 
Hence $\thetamapk$ is not necessarily an algebra homomorphism. However, if $\vect{k} \in \{0,\infty\}^{\lev}$ then clearly $\zetaQk$ is a character, and hence so is $\nuQk$. In particular $\nuQ^{\svect{\infty}}$ is a character and so $\thetamap^{\svect{\infty}}$ is a morphism of Hopf algebras.

\begin{prop}
The linear functional $\nuQk$ is $\vect{k}$-odd for all $\vect{k}\in \NNinfl$.
\end{prop}
\begin{proof}
It is clear that $\nuQk(1) = 1$, so $\nuQk$ is invertible. Therefore we only need to show that $(\overline{\nuQk})_{\vect{n}} = ((\nuQk)^{-1})_{\vect{n}}$ whenever $\vect{n}\le \vect{k}$. 

Since $\zetabarQ$ is a character, $\zetaQk$ is multiplicative ``up to rank $\vect{k}$," meaning that $\zetaQk(ab) = \zetaQk(a) \zetaQk(b)$ for all $a,b\in\hopf$ such that $ab \in \bigoplus_{\vect{n} \le \vect{k}} \hopf_{\vect{n}}$. Thus the usual formula for the inverse of a character applies up to rank $\vect{k}$; that is, $(\zetaQk)^{-1}(h) = \zetaQk( \apodeQ(h) )$ whenever $h \in \bigoplus_{\vect{n} \le \vect{k}} \hopf_{\vect{n}}$. It follows that 
\[ 
(\nuQk)_{\vect{n}} = ((\zetaQk)^{-1} \zetaQ)_{\vect{n}} = (\nuQ)_{\vect{n}}
\]
for all $\vect{n} \le \vect{k}$. Consequently, $(\overline{\nuQk})_{\vect{n}} = (\nubarQ)_{\vect{n}}$ for all $\vect{n}\le \vect{k}$. Moreover, since $\nuQk$ is multiplicative up to rank $\vect{k}$, we have $((\nuQk)^{-1})_{\vect{n}} = ((\nuQk) \circ\apodeQ)_{\vect{n}} = (\nuQ \circ \apodeQ)_{\vect{n}} = (\nuQ^{-1})_{\vect{n}}$ for all ${\vect{n}}\le {\vect{k}}$. Since $\nuQ$ is odd, $(\nuQ^{-1})_{\vect{n}} = (\nubarQ)_{\vect{n}}$ for all $\vect{n}$, so we can conclude that $(\overline{\nuQk})_{\vect{n}} = (\nubarQ)_{\vect{n}} = (\nuQ^{-1})_{\vect{n}} = ((\nuQk)^{-1})_{\vect{n}}$ whenever $\vect{n}\le \vect{k}$.
\end{proof}

\begin{theorem} \label{T:nuformulas}
Let $\vect{I} = (\vect{i}_1, \ldots, \vect{i}_m)$ be a vector composition. We have
\begin{equation} \label{E:nuQkM}
\nuQk(M_{\vect{I}})  = 
\begin{cases}
1 & \text{if $\vect{I} = \emptycomp$} \\
2 \cdot (-1)^{\len(\vect{I}) + |\vect{I}|} & \text{if $|\vect{i}_m|$ is odd and $\Sigma\vect{I} \le \vect{k}$} \\
2 \cdot (-1)^{\len(\vect{I})} & \text{if $|\vect{I}| - |\vect{i}_m|$ is odd, $\Sigma\vect{I} - \vect{i}_m \le \vect{k}$ and $\Sigma\vect{I} \not\le \vect{k}$} \\
0 & \text{otherwise,}
\end{cases}
\end{equation}
and
\begin{equation} \label{E:nuQkF}
\nuQk(F_{\vect{I}})  = 
\begin{cases}
1 & \text{if $\vect{I} = \emptycomp$} \\
2 & \text{if $\vect{I} = (\vect{E}, \vect{i})$ where $\vect{E}$ is a coordinate vector composition,}\\
  &  \text{ $\vect{i}$ is a column vector, and one of the following holds:} \\
& \quad \text{(1) $\Sigma\vect{I} \le \vect{k}$;} \\
& \quad \text{(2) $|\vect{E}|$ is odd and $\Sigma\vect{E} \le \vect{k}$;} \\
& \quad \text{(3) $|\vect{E}|$ is even, $|\vect{i}| > 1$, and $\Sigma\vect{E} + \vect{e} \le \vect{k}$, where $\vect{e}$ is the unique} \\
& \quad \quad \text{coordinate vector such that $(\vect{e}, \vect{i}-\vect{e}) > \vect{i}$} \\
0 & \text{otherwise.}
\end{cases}
\end{equation}
\end{theorem}
\begin{proof}
The defining formula $\nuQk(M_{\vect{I}})  =  \sum_{j=0}^m \zetaQk(\apodeQ(M_{(\vect{i}_1, \ldots, \vect{i}_j)})) \zetaQ(M_{(\vect{i}_{j+1},\ldots, \vect{i}_m)})$ reduces to 
\begin{equation} \label{E:nunu} 
\nuQk(M_{\vect{I}})
=  \zetaQk(\apodeQ(M_{(\vect{i}_1,\ldots, \vect{i}_{m-1})})) 
+ \zetaQk( \apodeQ (M_{(\vect{i}_1,\ldots, \vect{i}_m)}))
\end{equation}
using \eqref{E:zetaQSyml}. Applying the antipode formula \eqref{E:M-antipode} and the defining equation \eqref{E:zetaQk} for $\zetaQk$ immediately yields \eqref{E:nuQkM}.

Now we prove \eqref{E:nuQkF}. The case $\vect{I} = \emptycomp$ is trivial, so we assume that $\vect{I}$ is not empty. Suppose that there exists $r\ne m$ such that $|\vect{i}_r| > 1$. Then 
\[
\nuQk(F_{\vect{I}}) = \sum_{\vect{J} \ge \vect{I}} \nuQk(M_{\vect{J}}) 
= \sumsub{s \in [m] - \{r\} \\ \vect{J}_s \ge \vect{i}_s} \sum_{\vect{J}_r \ge \vect{i}_r}
 \nuQk(M_{\vect{J}_1 \cdots \vect{J}_m}). 
\]
For each $s \in [m] - \{r\}$ choose an arbitrary $\vect{J}_s \ge \vect{i}_s$ and consider the sum 
\begin{equation}\label{E:tech-sum}
\sum_{\vect{J}_r \ge \vect{i}_r} \nuQk(M_{\vect{J}_1 \cdots \vect{J}_m}).
\end{equation} 
We wish to show that this sum is zero. Let $\vect{j}$ be the last column of $\vect{J}_m$. There are two cases:
\begin{itemize}
\item Suppose that $|\vect{j}|$ is odd and $\Sigma \vect{I} \le \vect{k}$. By \eqref{E:nuQkM}, the sum \eqref{E:tech-sum} is $\sum_{\vect{J}_r \ge \vect{i}_r} 2 (-1)^{\len(\vect{J}_1 \cdots \vect{J}_m) + |\vect{I}|}$, which simplifies to the form $\pm 2 \cdot \sum_{\vect{J}_r \ge \vect{i}_r} (-1)^{\vect{J}_r}$. This last sum equals zero by inclusion-exclusion and the fact that the poset of vector compositions of $\vect{i}_r$ ordered by $\cleq$ is a Boolean lattice. 
\item Suppose that $|\vect{I}| - |\vect{j}|$ is odd, $\Sigma\vect{I} - \vect{j} \le \vect{k}$, and $\Sigma\vect{I}\nleq \vect{k}$. By \eqref{E:nuQkM}, the sum \eqref{E:tech-sum} simplifies to the form $\pm 2 \cdot \sum_{\vect{J}_r \ge \vect{i}_r} (-1)^{\len(\vect{J}_r)},$ which again must equal zero. 
\end{itemize}
If we are not in either of these two cases, then by \eqref{E:nuQkM} every term in the sum \eqref{E:tech-sum} is zero (Note that $\vect{j}$ and $\Sigma \vect{I}$ remain constant in the sum \eqref{E:tech-sum}). It follows that $\nuQk(F_{\vect{I}}) = 0$ unless all the columns of $\vect{I}$, except possibly the last one, are coordinate vectors.

For the rest of the proof we can therefore assume that $\vect{I} = (\vect{E}, \vect{i})$ where $\vect{E}$ is a sequence of coordinate vectors and $\vect{i} = \vect{i}_m$ is a column vector. For a nonempty vector composition $\vect{J}$ let $\lst(\vect{J})$ denote the last column of $\vect{J}$. We have 
\begin{equation} \label{E:tech-formula}
\nuQk(F_{\vect{I}}) = \sum_{\vect{J} \ge \vect{i}} \nuQk(M_{(\vect{E} , \vect{J})}) =
\sum_{r = 1}^{|\vect{i}|} \sumsub{\vect{J}\ge \vect{i} \\ |\lst(\vect{J})| = r} \nuQk(M_{(\vect{E} , \vect{J})}).
\end{equation}
Note that if $\vect{J}, \vect{J}' \ge \vect{i}$ and $|\lst(\vect{J})| = |\lst(\vect{J}')|$ then $\lst(\vect{J}) = \lst(\vect{J}')$.  Thus, in the last expression in \eqref{E:tech-formula}, once $r$ is fixed the inner sum can be written as
\begin{equation*} 
\sum_{\vect{J} \ge \vect{i} - \vect{v}} \nuQk(M_{(\vect{E} , \, \vect{J}, \, \vect{v})})
\end{equation*}
where $\vect{v}$ is some column vector of weight $r$. This sum 
is of the same general form as \eqref{E:tech-sum}, and by the same argument as before we can conclude that it
vanishes unless $|\vect{i} - \vect{v}| = 0$ or $1$. Now \eqref{E:tech-formula} becomes
\begin{equation} \label{E:nunu2}
\nuQk(F_{\vect{I}}) = 
\begin{cases}
\nuQk(M_{(\vect{E}, \vect{i})}) & \text{if $|\vect{i}| = 1$} \\
\nuQk(M_{(\vect{E}, \vect{i})}) + \nuQk(M_{(\vect{E}, \vect{e}, \vect{i}-\vect{e})}) &
 \text{if $|\vect{i}| > 1$, and $\vect{e}$ is the coordinate} \\
& \text{vector such that $(\vect{e}, \vect{i} - \vect{e}) > \vect{i}$.}
\end{cases}
\end{equation}
Using \eqref{E:nunu}, \eqref{E:M-antipode}, \eqref{E:zetaQk}, we compute
\[
\nuQk(M_{(\vect{E},\vect{i})}) = 
\begin{cases}
2 & \text{if $\Sigma\vect{I} \le \vect{k}$ and $|\vect{i}|$ is odd, or} \\
 & \text{if $\Sigma\vect{I} \nleq \vect{k}$, $\Sigma\vect{E} \le \vect{k}$, and $|\vect{E}|$ is odd} \\
0 & \text{otherwise}
\end{cases}
\]
\[
\nuQk(M_{(\vect{E}, \vect{e}, \vect{i} - \vect{e})}) = 
\begin{cases}
2 & \text{if $\Sigma\vect{I} \le \vect{k}$ and $|\vect{i}|$ is even, or}\\
 & \text{if $\Sigma\vect{I} \nleq\vect{k}$, $\Sigma\vect{E} + \vect{e} \le \vect{k}$,  and $|\vect{E}|$ is even} \\
 0 & \text{otherwise.}
\end{cases}
\]
Substituting these formulas back into \eqref{E:nunu2} yields \eqref{E:nuQkF}. 
\end{proof}

In the following two subsections we will apply Theorem~\ref{T:nuformulas} to explicitly compute $\thetamapk$ in two special cases. As preparation, we now supply general formulas for the maps $\QSyml\to \QSyml$ induced by arbitrary linear functionals. 

\begin{prop} \label{P:InducedMapQSyml}
Let $\zeta:\QSyml \to \field$ be a linear functional and let $\Psi:\QSyml \to \QSyml$ be the unique morphism of multigraded coalgebras such that $\zetaQ \circ \Psi = \zeta$. Let $\vect{I}$ be a vector composition. We have
\begin{equation} \label{E:InducedMapQSyml}
\Psi(M_{\vect{I}}) = \sum_{\vect{I} = \vect{I}_1 \cdots \vect{I}_m} \zeta(M_{\vect{I}_1}) \cdots \zeta(M_{\vect{I}_m}) \, M_{(\Sigma\vect{I}_1, \ldots, \Sigma\vect{I}_m)}
\end{equation}
where the sum is over all ways of writing $\vect{I}$ as a concatenation of vector compositions. Let $(\sigma, u) = (\sigma_1 \ldots \sigma_n, u_1 \ldots u_n)$ be any colored permutation such that $\wDes(\sigma, u) = \vect{I}$ and $\Des(u) \subseteq \Des(\sigma)$. Then
\begin{equation}\label{E:InducedMapF}
\Psi(F_{\wDes(\sigma, u)}) = 
 \sumsub{\sigma = \pi_1 \cdots \pi_m \\ u = v_1 \cdots v_m} \zeta(F_{\wDes(\std(\pi_1), v_1)}) \cdots \zeta(F_{\wDes(\std(\pi_m), v_m)}) 
M_{(\mdeg(v_1), \ldots, \mdeg(v_m))}
\end{equation}
where the sum is over all ways of writing $\sigma$  and $u$ as concatenations of an equal number of subwords such that each subword $\pi_i$ of $\sigma$ has the same length as the corresponding subword $v_i$ of $u$.
\end{prop}

\begin{proof}
Equation~\eqref{E:InducedMapQSyml} follows directly from \eqref{E:InducedMap}. By \eqref{E:DofF} we have $\Psi(F_{\vect{I}}) = \Psi(\mathcal{D}(\vect{F}_{\sigma, u}))$. Since $\mathcal{D}$ is a morphism of  coalgebras, 
\[
\begin{aligned}
\Delta(F_{\vect{I}}) = & \Delta(\mathcal{D}(\vect{F}_{\sigma, u})) = \mathcal{D}\otimes \mathcal{D} (\Delta(\vect{F}_{\sigma, u})) \\ 
& = \sumsub{\sigma = \pi_1 \pi_2 \\ u = v_1 v_2} \mathcal{D}(\vect{F}_{\std(\pi_1), v_1}) \otimes \mathcal{D}(\vect{F}_{\std(\pi_2), v_2}) = \sumsub{\sigma = \pi_1 \pi_2 \\ u = v_1 v_2} F_{\wDes(\std(\pi_1), v_1)} \otimes F_{\wDes(\std(\pi_2), v_2)}.
\end{aligned}
\]
Using this coproduct formula we again apply \eqref{E:InducedMap} to get Equation~\eqref{E:InducedMapF}.
\end{proof}

\subsection{The specialization $\vect{k} = \svect{\infty}$}
\label{SS:nuQInducedMap}

Let $\kodd = \kodd^{\svect{\infty}}(\QSyml)$. We will call this the \newword{peak algebra of level $\lev$}. Let $\thetamap = \thetamap^{\svect{\infty}}$. This map is the multigraded version of Stembridge's descents-to-peaks map. We will give formulas for $\thetamap(M_{\vect{I}})$ and $\thetamap(F_{\vect{I}})$. We will also introduce a basis of \newword{peak functions} for $\kodd$ and describe its relationship to the $\eta$-basis which was defined in \S\S\ref{SS:etabasis}.

Once again we will need some preliminaries about 
vector compositions. 

If $\vect{I}$ is a vector composition whose last column has odd weight, then there is a unique factorization  $\vect{I} = \vect{I}_1 \cdots \vect{I}_m$ into a concatenation of vector compositions such that the last column of each $\vect{I}_j$ has odd weight and all other columns have even weight. For such an $\vect{I}$, we define 
\[ 
\odd(\vect{I}) = (\Sigma \vect{I}_1, \ldots, \Sigma\vect{I}_m).
\]
Every column of $\odd(\vect{I})$ has odd weight. For example, 
\[
\odd(
\begin{pmatrix}
0 & 1 & 2 & 1 & 0 & 0 & 2 \\
3 & 1 & 1 & 0 & 2 & 0 & 0 \\
0 & 0 & 1 & 0 & 0 & 3 & 1 
\end{pmatrix}
) =
\begin{pmatrix}
0 & 4 & 0 & 2 \\
3 & 2 & 2 & 0 \\
0 & 1 & 3 & 1
\end{pmatrix}.
\]

Now let $\vect{I}$ be an arbitrary vector composition.  There is a unique factorization $\vect{I} = \vect{I}_1 \cdots \vect{I}_m$ into a concatenation of vector compositions such that $\vect{I}_m$ is a (possibly empty) sequence of coordinate vectors, and  each $\vect{I}_j$, $j\ne m$, has the form $\vect{I}_j = (\vect{E}, \vect{v})$, where $\vect{E}$ is a (possibly empty) sequence of coordinate vectors and $|\vect{v}| \ne 1$. Define $\Lambda(\vect{I})$ by
\[
\Lambda(\vect{I}) = ( \Sigma \vect{I}_1, \ldots, \Sigma \vect{I}_m)
\]
For example,
\[ 
\Lambda(
\left(
\begin{array}{ccccccccccc}
0 & 1 & 2 & 1 & 0 & 0 & 0 & 1 & 3 & 1 & 0 \\
1 & 0 & 1 & 0 & 2 & 0 & 0 & 0 & 0 & 0 & 0 \\
0 & 0 & 1 & 0 & 0 & 3 & 1 & 0 & 2 & 0 & 1
\end{array}
\right)
) =
\begin{pmatrix}
3 & 1 & 0 & 4 & 1 \\
2 & 2 & 0 & 0 & 0 \\
1 & 0 & 3 & 3 & 1
\end{pmatrix}.
\]
We will call $\Lambda(\vect{I})$ the \newword{peak vector composition of $\vect{I}$} for the following reason. Recall that the \newword{peak set} of a permutation $\sigma = (\sigma_1, \ldots, \sigma_n) \in \perm_n$ is the set 
\[
\Peak(\sigma) = \{ i \in [2, n-1] \mid \sigma_{i-1} < \sigma_i > \sigma_{i+1}\}.
\] 
If $(\sigma, u)$ is a colored permutation then
\[ 
 \dofI(\Lambda(\wDes(\sigma, u))) = \Peak(\sigma).
\]
Accordingly, we define the peak set of any vector composition $\vect{I}$ to be 
\[\pofI(\vect{I}) = \dofI(\Lambda(\vect{I})).\]
In essence, $\pofI(\vect{I})$ keeps track of where the columns of weight at least $2$ occur in $\vect{I}$.


\begin{prop}
\label{P:ThetaMap}
For any vector composition $\vect{I}$ we have
\begin{equation} \label{E:ThetaM}
\Theta(M_{\vect{I}}) = \begin{cases}
1 & \text{if $\vect{I} = \emptycomp$} \\
(-1)^{\len(\vect{I}) + |\vect{I}|} \eta_{\odd(\vect{I})} & \text{if the last column of $\vect{I}$ has odd weight} \\
0 & \text{otherwise,} 
\end{cases} 
\end{equation}
and if $u = \cofI(\vect{I})$ is the coloring word of $\vect{I}$, then
\begin{equation} \label{E:ThetaF}
\thetamap(F_{\vect{I}}) = \sumsub{\vect{J} \tleq \vect{E}_u \\ \pofI(\vect{I}) \subseteq \dofI(\vect{J}) \cup (\dofI(\vect{J}) + 1)}
2^{\len(\vect{J})} M_{\vect{J}}.
\end{equation}
\end{prop}

\begin{proof} 
Letting $\nuQ = \nuQ^{\svect{\infty}}$, we obtain the following specializations to $\vect{k} = \svect{\infty}$ of the formulas in Theorem~\ref{T:nuformulas}
\begin{equation} \label{E:nuQM}
\nuQ(M_{\vect{I}}) = \begin{cases}
1 & \text{ if $\vect{I} = \emptycomp$} \\
2 \cdot (-1)^{\len(\vect{I}) + |\vect{I}|} & \text{ if the last column of $\vect{I}$ has odd weight} \\
0 & \text{ otherwise,}
\end{cases}
\end{equation}
and
\begin{equation} \label{E:nuQF}
\nuQ(F_{\vect{I}}) = \begin{cases}
1 & \text{if $\vect{I} = \emptycomp$} \\
2 & \text{if $\vect{I} = (\vect{E}, \vect{v})$, where $\vect{E}$ is a (possibly empty) sequence} \\
 & \text{of coordinate vectors and $\vect{v}$ is any column vector} \\
0 & \text{ otherwise.}
\end{cases}
\end{equation}

Let $\lst(\vect{J})$ denote the last column of any vector composition $\vect{J}$. By \eqref{E:InducedMapQSyml} and then \eqref{E:nuQM}, we have 
\[ 
\thetamap(M_{\vect{I}}) = \sum_{\vect{I} = \vect{I}_1 \cdots \vect{I}_m} \nuQ(M_{\vect{I}_1}) \cdots \nuQ(M_{\vect{I}_m}) M_{(\Sigma\vect{I}_1, \ldots, \Sigma\vect{I}_m)}
= \sumsub{\vect{I} = \vect{I}_1 \cdots \vect{I}_m \\ \forall i \; |\lst(\vect{I}_i)| \text{ is odd}} 
(-1)^{\len(\vect{I}) + |\vect{I}|} 2^m M_{(\Sigma\vect{I}_1, \ldots, \Sigma\vect{I}_m)}.
\]
In the last sum, every $\vect{I}_i$ ends in a column of odd weight if and only if $\Sigma\vect{I}_i$ is the sum of columns in $\odd(\vect{I})$. Thus \eqref{E:ThetaM} follows. 

Turning to the proof of \eqref{E:ThetaF},
let $(\sigma, u) \in \perm_n \times [0, \lev-1]^n$ be a colored permutation such that $\Des(u) \subseteq \Des(\sigma)$ and $\wDes(\sigma, u) = \vect{I}$. By \eqref{E:InducedMapF} and then \eqref{E:nuQF} we have
\begin{equation*} 
\thetamap(F_{\vect{I}}) = \sumsub{\sigma = \pi_1 \cdots \pi_m \\ u = v_1 \cdots v_m \\ \wDes(\std(\pi_i), v_i) = (\vect{e}_*, \ldots, \vect{e}_*, *)} 
2^m M_{(\mdeg(v_1), \ldots, \mdeg(v_m))}.
\end{equation*}
Here the sum is over all ways of writing $\sigma = \pi_1 \cdots \pi_m$  and $u = v_1 \cdots v_m$ as concatenations of an equal number of subwords such that for every $i$, the subwords $\pi_i$ and $v_i$ have the same length, and the vector composition $\wDes(\std(\pi_i), v_i)$ is a (possibly empty) sequence of coordinate vectors followed by some column vector. Since $\wDes(\std(\pi_i), v_i)$ has this form if and only if $\Peak(\pi_i) = \emptyset$ (the colors are irrelevant), a given factorization $\sigma = \pi_1 \cdots \pi_m$ appears in the sum if and only if there is a break immediately before or after every position where a peak occurs in $\sigma$. In other words, if $\alpha = (\len(\pi_1), \ldots, \len(\pi_m))$ is the sequence (composition) of lengths of the $\pi_i$, then in order for the factorization $\sigma = \pi_1 \cdots \pi_n$ to appear in the sum it is necessary and sufficient that $\Peak(\sigma) \subseteq \dofI(\alpha) \cup (\dofI(\alpha) + 1)$.

Since $\len(\pi_i) = \len(v_i) = |\mdeg(v_i)|$ for each $i$, based on our characterization of what length sequences $(\len(\pi_1), \ldots, \len(p_m))$ can occur in the sum, we can write
\[ 
\thetamap(F_{\vect{I}}) = \sum_{\vect{J}} 2^{\len(\vect{J})} \, M_{\vect{J}}
\]
where the sum is over all vector compositions $\vect{J}$ for which there exists a factorization $u = v_1 \ldots v_m$ such that (1): $\vect{J} = (\mdeg(v_1), \ldots, \mdeg(v_m))$, and (2): $\Peak(\sigma) \subseteq \dofI(\vect{J}) \cup (\dofI(\vect{J}) + 1)$. Condition (1) is equivalent to $\vect{J} \tleq \vect{E}_u$, and in Condition (2) one can replace $\Peak(\sigma)$ with $\pofI(\vect{I})$. Hence the formula we obtained for $\thetamap(F_{\vect{I}})$ agrees with \eqref{E:ThetaF}.
\end{proof}

An immediate consequence of \eqref{E:ThetaF} is that $\thetamap(F_{\vect{I}})$ depends only on the peak set of $\vect{I}$ and the coloring word of $\vect{I}$. Imitating the $\lev =1$ case, we will define multigraded analogues of peak functions as the images of the $F_{\vect{I}}$ under $\thetamap$. In the following, by a \newword{peak subset of $[n]$} we mean a subset $S\subseteq \{2,\ldots, n-1\}$ such that $i \in S \implies i-1 \notin S$; i.e., $S$ is the peak set of some permutation of $[n]$. 

\begin{definition}
Let $u$ be a coloring word of length $n$ and let $S$ be a peak subset of $[n]$. The \newword{peak quasisymmetric function of level $\lev$} (or \newword{peak function}) indexed by $S$ and $u$  is defined by
\begin{equation} \label{E:theta-M}
\theta_{S, u} = \sumsub{\vect{J} \tleq \vect{E}_u \\ S \subseteq \dofI(\vect{J}) \cup (\dofI(\vect{J}) + 1)} 
2^{\len(\vect{J})} M_{\vect{J}}.
\end{equation}
Here $\dofI(\vect{J})+1$ stands for the set $\{ s + 1 \mid s \in \dofI(\vect{J})\}$. 
\end{definition}

Observe that, by \eqref{E:ThetaF}, if $\vect{I}$ is any vector composition then
\[
\thetamap(F_{\vect{I}}) = \theta_{\pofI(\vect{I}), \cofI(\vect{I})}.
\]
If $\lev = 1$ then the peak functions are precisely the $K_{\Lambda}$ of Stembridge \cite{Stembridge97}. These form a basis of the peak subalgebra of $\QSym$. However, when $\lev > 1$ the peak functions are not linearly independent. For example, when $\lev = 2$ we have
\[ 
\theta_{\{2\}, 010} = \theta_{\emptyset, 010} - \theta_{\emptyset, 001} + \theta_{\{2\}, 001}.
\]
Note also that, unlike the $\lev = 1$ case, the higher level peak functions are not always $F$-positive. For example, when $\lev = 2$ we have
\[
\theta_{\emptyset, 010} = 
2 F_{\colvec{2 \\ 1}} - 2 F_{\colvec{2 & 0 \\ 0 & 1}} + 2 F_{\colvec{1 & 1\\0 & 1}} 
+ 4 F_{\colvec{1 & 1 \\ 1 & 0}} - 2 F_{\colvec{1 & 1 & 0 \\ 0 & 0 & 1}}.
\]

Our next goal is to describe a subset of peak functions that form a basis for $\kodd$. We will also describe how this basis relates to the $\eta$-basis. Before we proceed we will need some additional definitions and a lemma. Let $u \in [0, \lev-1]^n$ be a color word. Given an \newword{odd vector composition} $\vect{I}$ (meaning every column of $\vect{I}$ has odd weight), we define $\tilde{\vect{I}}$ to be the vector composition with the same coloring word as $\vect{I}$, obtained by replacing each column of $\vect{I}$ by a new vector composition of equal weight whose last column has weight $1$ and all other columns have weight $2$. For example, 
\[
\vect{I} = 
\begin{pmatrix}
2 & 0 & 3 \\
0 & 0 & 2 \\
1 & 1 & 0
\end{pmatrix}
\implies
\tilde{\vect{I}} =
\begin{pmatrix}
2 & 0 & 0 & 2 & 1 & 0 \\
0 & 0 & 0 & 0 & 1 & 1 \\
0 & 1 & 1 & 0 & 0 & 0 
\end{pmatrix}.
\]
Next, let $u$ be a color word of length $n$ and let $S$ be a peak subset of $[n]$. We define $\vect{I}_{S,u}$ to be the unique vector composition satisfying the following three properties: (1) every column of $\vect{I}_{S,u}$ has weight $1$ or $2$; (2) the coloring word of $\vect{I}_{S,u}$ is $u$; and (3) $\pofI(\vect{I}_{S,u}) = S$. Note that the last column of $\vect{I}_{S,u}$ must have weight $1$, so it makes sense to talk about $\odd(\vect{I}_{S,u})$. For example, when $\lev = 3$, if $S = \{2, 6, 8\}$ and $u = 002200011$ then
\[ 
\vect{I}_{S,u} = \begin{pmatrix}
2 & 0 & 0 & 2 & 1 & 0 \\
0 & 0 & 0 & 0 & 1 & 1 \\
0 & 1 & 1 & 0 & 0 & 0
\end{pmatrix} 
\quad\text{ and }\quad
\odd(\vect{I}_{S,u}) = 
\begin{pmatrix}
2 & 0 & 3 \\
0 & 0 & 2 \\
1 & 1 & 0
\end{pmatrix}.
\]
The pair $(S, u)$ will be called \newword{admissible} if $\cofI(\odd(\vect{I}_{S, u})) = u$. The previous example is of an admissible pair. When $\lev = 1$ the correspondence $I \leftrightarrow \pofI(\tilde{I})$ is a bijection between odd compositions and peak subsets of $\{1,\ldots, |I|\}$, as observed by Schocker \cite{Schocker05}. This observation generalizes to vector compositions; the proof of the following lemma is similar to that of Proposition~3.1 in \cite{Schocker05} and will be omitted.

\begin{lemma}
Let $u \in [0,\lev - 1]^n$ be a color word. The map $\vect{I} \mapsto (\pofI(\tilde{\vect{I}}), u)$ is a bijection from the set $\OComp(u)$ to the set of admissible pairs with second coordinate equal to $u$. The inverse of this map is given by $(S, u) \mapsto \odd(\vect{I}_{S, u})$. 
\end{lemma}

Here is our main result on peak functions:

\begin{theorem}
The peak functions $\theta_{S, u}$ indexed by admissible pairs $(S,u)$ form a basis for the higher level peak algebra $\kodd$. Moreover, we can explicitly describe the relationship between the $\theta_{S, u}$ and the $\eta_{\vect{I}}$ as follows: If $\vect{I}$ is an odd vector composition and $u = \cofI(\vect{I})$, then
\begin{equation} \label{E:eta-theta}
\eta_{\vect{I}} = \sum_{S\subseteq \pofI(\widetilde{\vect{I}})} (-1)^{|S|}\, \theta_{S, u}, 
\end{equation}
and if $(S, u)$ is any admissible pair then 
\begin{equation} \label{E:theta-eta}
\theta_{S, u} = 
\sum_{T \subseteq S} 
(-1)^{|T|} \, \eta_{\odd(\vect{I}_{T,u})}.
\end{equation}
\end{theorem}
\begin{proof}
As shown in Theorem~\ref{T:koddQSym-eta}, the $\eta_{\vect{I}}$ such that $\vect{I}$ is an odd vector composition form a basis for $\kodd$. It therefore suffices to prove \eqref{E:eta-theta}, which is equivalent to \eqref{E:theta-eta} by M\"obius inversion. 

To prove \eqref{E:eta-theta}, we apply the definition of $\theta_{S, u}$ (see \eqref{E:theta-M}) to get
\[
\sum_{S\subseteq \pofI(\widetilde{\vect{I}})} (-1)^{|S|}\, \theta_{S, \vect{E}_u}
=
\sum_{\vect{J} \tleq \vect{E}_u } 2^{\len(\vect{J})} M_{\vect{J}} 
\sum_{S \subseteq \pofI(\tilde{\vect{I}}) \cap (\dofI(\vect{J}) \cup (\dofI(\vect{J}) + 1))} (-1)^{|S|}.
\]
Now 
$\pofI(\tilde{\vect{I}}) \cap (\dofI(\vect{J}) \cup (\dofI(\vect{J}) + 1))=\emptyset
\iff
\dofI(\vect{J}) \cup (\dofI(\vect{J}) + 1) \subseteq [n-1] \setminus \pofI(\tilde{\vect{I}}) 
\iff
\dofI(\vect{J}) \subseteq [n-1] \setminus ((\pofI(\tilde{\vect{I}})  - 1) \cup \pofI(\tilde{\vect{I}}))
=\dofI(\vect{I}).$ The last equality is essentially Equation~(8) of \cite{Schocker05}. Thus the inner sum is $1$ if $\dofI(\vect{J}) \subseteq \dofI(\vect{I})$ and $0$ otherwise. Since both $\vect{J} \tleq \vect{E}_u$ and $\vect{I} \tleq \vect{E}_u$, the condition $\dofI(\vect{J}) \subseteq \dofI(\vect{I})$ is equivalent to $\vect{J} \tleq \vect{I}$, as required.
\end{proof}

\begin{remark}
When $\lev = 1$, our $\eta$-basis for $\kodd$ is, up to sign, dual to the $\Gamma$-basis introduced by Schocker \cite[Equation~(9)]{Schocker05}.
\end{remark}

\subsection{The specialization $\lev = 1$}

In this case, $\QSyml = \QSym$, and for each $k\in \NN$, we have a $k$-odd linear functional $\nuQ^{(k)}$ on the Hopf algebra of (ordinary) quasisymmetric functions $\QSym$. Therefore we get an induced morphism of coalgebras $\thetamap^{(k)}:\QSym \to \QSym$ whose image is contained in the $k$-odd Hopf subalgebra $\kodd^{k}(\QSym)$. We will apply Theorem~\ref{T:nuformulas} to give a formula for $\thetamap^{(k)}$.

\begin{prop} \label{P:theta1kF}
Let $k>0$ be an odd integer and $I$ be a composition of $n$. We have $\thetamap^{(k)} = \thetamap^{(k+1)}$ and
\begin{equation*} 
\thetamap^{(k)}(F_I) = 
\sumsub{J \comp n \\ \pofI(I) \subseteq \dofI(J) \cup (\dofI(J) + 1) \\ 
\dofI(I) \subseteq  \dofI(J) \cup (\dofI(J) + 1) \cup \cdots \cup (\dofI(J) + k)} 2^{\len(J)} \, M_J.
\end{equation*}
\end{prop}
\begin{proof}
Equation \eqref{E:nuQkF} becomes
\[
\nuQ^{(k)}(F_I) = \nuQ^{(k+1)}(F_I) = 
\begin{cases}
1 & \text{if $I = \emptycomp$} \\
2 & \text{if $I = (j)$ or $I = (1^i, j)$ or some $i \in \{1,\ldots, k\}$ and $j \ge 1$} \\
0 & \text{otherwise.}
\end{cases}
\]
Let $\sigma = (\sigma_1, \ldots, \sigma_n)$ be a permutation with descent composition $\wDes(\sigma) =\wDes(\sigma, 11\ldots 1) = I$. Then by \eqref{E:InducedMapF},
\[ 
\thetamap^{(k)}(F_I) = \sum
2^{m} \, M_{(\len(\pi_1), \ldots, \len(\pi_m))}
\]
where the sum is over all ways of writing $\sigma = \pi_1 \cdots \pi_m$ as a concatenation of subwords $\pi_i$ such that for every $i$, $\wDes(\std(\pi_i))$ has for form $(j)$ or $(1^i, j)$ for some $i\in \{1,\ldots, k\}$ and $j\ge 1$. The composition $\wDes(\std(\pi_i))$ has this form if and only if $\pi_i$ begins with at most $k$ descents followed by only ascents; in particular $\pi_i$ has no peaks, so a break occurs in the factorization $\sigma = \pi_1 \cdots \pi_m$ immediately before or after every peak. Thus in order for a factorization $\sigma = \pi_1 \cdots \pi_m$ to appear in the sum it is necessary and sufficient that the composition $J = (\len(\pi_1), \ldots, \len(\pi_m))$ satisfy $\wDes(\sigma) \subseteq \dofI(J) \cup (\dofI(J) + 1) \cup \cdots \cup (\dofI(J) + k)$ and $\Peak(\sigma) \subseteq \dofI(J) \cup (\dofI(J) + 1)$.  
\end{proof}

We can also write the expression in Proposition \ref{P:theta1kF} 
using operations on vector compositions rather than sets, in the following way. For any composition $I = (i_1, \ldots, i_m)$, let $I^{\star k}$ be the composition obtained by replacing every component $i_r$ with $(1^{i_r})$ if $i_r\le k$, and with $(i_r^k, i_r-k)$ if $i_r>k$. Write $I^{\star}$ for $I^{\star 1}$. Then 
we can rewrite the expression in Proposition \ref{P:theta1kF} as:
\[
\thetamap^{(k)}(F_I) = 
\sumsub{J \comp n \\ \Lambda(I) \tleq J^{\star} \\ 
I \tleq  J^{\star k}} 2^{\len(J)} \, M_J.
\]

Consider the case $k=2$. It was already observed in Proposition~\ref{P:shifted} that $\kodd^{2}$ is the Hopf algebra spanned by all shifted quasisymmetric functions (sqs-function). Going a bit further, let us remark that the map $\thetamap^{(2)}$ takes each fundamental basis element $F_I$ to an sqs-functions. Let $I=(i_1, \ldots, i_m)$ be a composition with $i_1>1$, and let $\theta_I$ be the sqs-function as defined in \cite[Definition~3.1]{BMSW02}. It is shown in \cite[Theorem~3.6]{BMSW02} that the $\theta_I$ are precisely the functions defined by Billey and Haiman in \cite[Equation~(3.2)]{BH95}. It follows from Proposition~\ref{P:theta1kF} that
\[ 
\thetamap^{(2)}(F_{I}) = \theta_I.
\]
Furthermore, clearly $\thetamap^{(2)}(F_{(1, i_1-1, i_2, \ldots, i_m)}) = \thetamap^{(2)}(F_I)$, so $\thetamap^{(2)}(F_J)$ is an sqs-function for {\it every} composition $J$.






\end{document}